 \documentclass[12pt]{amsart}
\usepackage{amsthm,amsmath,amssymb,slashed,url,bm,mathrsfs}
\usepackage{graphicx}
\usepackage[all]{xy}
\usepackage{epstopdf}
\usepackage{setspace}
\usepackage{verbatim}
\usepackage{mathdots}
\usepackage{lmodern}
\usepackage[T1]{fontenc}
\usepackage[utf8]{inputenc}
\usepackage[english]{babel}
\usepackage{microtype} 

\usepackage{amscd}     
\usepackage{amssymb}
\usepackage{amsmath, amsthm, graphics}
\usepackage{xypic}     
\LaTeXdiagrams        
\usepackage[all]{xy}
\xyoption{2cell} \UseAllTwocells \xyoption{frame} \CompileMatrices
\allowdisplaybreaks[3]
\usepackage{amsfonts}
\usepackage{bm}
\usepackage[normalem]{ulem}
\usepackage{hyperref}
\usepackage{tikz}
\usepackage{mathtools}
\usepackage{verbatim} 
\usepackage{upgreek}
\usepackage{todonotes}

\usepackage{slashed,url,bm,mathrsfs,mathtools,todonotes,xparse,booktabs,graphicx}
\usepackage[centertableaux]{ytableau}

\usepackage{fancyhdr}
\usepackage[font={small,sl}]{caption}
\font\teneurm=eurm10 \font\seveneurm=eurm7 
\font\fiveeurm=eurm5 
\newfam\eurmfam
\textfont\eurmfam=\teneurm \scriptfont\eurmfam=\seveneurm
\scriptscriptfont\eurmfam=\fiveeurm

\font\teneusm=eusm10 \font\seveneusm=eusm7 \font\fiveeusm=eusm5
\newfam\eusmfam
\textfont\eusmfam=\teneusm \scriptfont\eusmfam=\seveneusm
\scriptscriptfont\eusmfam=\fiveeusm
\def\eusm#1{{\fam\eusmfam\relax#1}}
\font\tencmmib=cmmib10 \skewchar\tencmmib='177
\font\sevencmmib=cmmib7 \skewchar\sevencmmib='177
\font\fivecmmib=cmmib5 \skewchar\fivecmmib='177
\newfam\cmmibfam
\textfont\cmmibfam=\tencmmib \scriptfont\cmmibfam=\sevencmmib
\scriptscriptfont\cmmibfam=\fivecmmib
\def\cmmib#1{{\fam\cmmibfam\relax#1}}

\def\^{{\wedge}}
\def\*{{\star}}
\def\bar{\overline}

\newtheorem{theorem}{Theorem}[section]
\newtheorem{lemma}[theorem]{Lemma}
\newtheorem{definition}[theorem]{Definition}
\newtheorem{proposition}[theorem]{Proposition}
\newtheorem{corollary}[theorem]{Corollary}
\newtheorem{conjecture}[theorem]{Conjecture}
\newtheorem{example}[theorem]{Example}

\newtheorem*{theoremstar}{Main Result}

\newtheorem*{theoremstar2}{Interpretation}

\definecolor{pBlue}{RGB}{86,139,190}
\definecolor{pCyan}{RGB}{149,186,201}
\definecolor{pSand}{RGB}{184,166,121}
\definecolor{pAlgae}{RGB}{87,115,135}
\definecolor{pSkin}{RGB}{236,216,167}
\definecolor{pGray}{RGB}{156,175,156}
\definecolor{pPink}{RGB}{215,114,127}
\definecolor{pOrange}{RGB}{211,153,80}

\newcommand\End{\mathop{\rm End}}

\newcommand\Hom{\mathop{\rm Hom}}

\newcommand\Span{\mathop{\rm Span}}
\def\Im{\mathop{\rm Im}}

\def\mod{\mathop{\rm mod}}
\newcommand\Res{\mathop{\rm Res}}

\newcommand\rk{\mathop{\rm rk}}

\newcommand\Tr{{\mathop{\rm Tr}}}

\newcommand\BC{{\mathbb C}}
\newcommand\BN{{\mathbb N}}

\newcommand\BR{{\mathbb R}}

\newcommand\BZ{{\mathbb Z}}

\newcommand\CA{{\mathcal A}}
\newcommand\CC{{\mathcal C}}

\newcommand\CF{{\mathcal F}}

\newcommand\CH{{\mathcal H}}

\newcommand\CL{{\mathcal L}}

\newcommand\CN{{\mathcal N}}

\newcommand\CT{{\mathcal T}}
\newcommand\CX{{\mathcal X}}
\newcommand\CV{{\mathcal V}}
\newcommand\CW{{\mathcal W}}
\newcommand\CU{{\mathcal U}}
\newcommand\CY{{\mathcal Y}}
\newcommand\CZ{{\mathcal Z}}

\newcommand\norm{||}

\newcommand\bareps{{\bar \varepsilon}}
\newcommand\vareps{{\varepsilon}}

\newcommand\Spec{{\mathrm{Spec}}}

\newcommand{\defin}[1]{%
\relax\ifmmode%
\textcolor{blue}{#1}%
\else\textcolor{blue}{\emph{#1}}%
\fi%
}

\newcommand{\jacktop}{\varpi}

\newcommand{\SymmetricFunctionRing}{\Lambda}

\newcommand{\partition}{\,\vdash\,}

\newcommand{\QX}{{\eusm{X}}}
\newcommand{\QY}{{\eusm{Y}}}
\newcommand{\QZ}{{\eusm{Z}}}

\newcommand{\xtr}{\cmmib{x}}
\newcommand{\ytr}{\cmmib{y}}
\newcommand{\ztr}{\cmmib{z}}

\newcommand{\SHc}{\mathrm{SH}^c}

\newcommand{\ttau}{{\tilde\tau}}
\newcommand{\tpsi}{{\tilde\psi}}
\newcommand{\hpsi}{{\hat\psi}}
\newcommand{\htau}{{\hat\tau}}

\newcommand{\dtrace}{{\Delta}}

\newcommand{\addset}{\mathcal{A}}
\newcommand{\remset}{\mathcal{R}}
\newcommand{\remsetp}{\mathcal{R}^+}

\newcommand{\BCe}{{  \BC_{{\bm{\vareps}}}  }}

\newcommand\latticehyp{{\Omega}}

\newcommand\one{{{\mathrm{Id}}}}
\newcommand\hj{{\hat j}}

\begin{document}

\title[Jack L-R Coefficients and the Nazarov-Sklyanin Lax Operator]{Jack Littlewood-Richardson Coefficients and \\ the Nazarov-Sklyanin Lax Operator}
\author{Ryan Mickler}
\address{
Melbourne, Victoria, Australia}
\email{ry.mickler@gmail.com}

\begin{abstract}
We continue the work begun by Mickler-Moll \cite{Mickler:2022} investigating the properties of the polynomial eigenfunctions of the Nazarov-Sklyanin quantum Lax operator. By considering products of these eigenfunctions, we produce a novel generalization of a formula of Kerov relating Jack Littlewood-Richardson coefficients and residues of certain rational functions. Precisely, we derive a system of constraints on Jack Littlewood-Richardson coefficients in terms of a simple multiplication operation on partitions.
\end{abstract}


\maketitle

\begin{onehalfspace}

\setcounter{tocdepth}{1}
\tableofcontents
\section{Introduction}
Let $\Lambda$ be the ring of symmetric functions and $\BCe = \BC(\vareps_1,\vareps_2)$. For $\lambda$ a partition, we let $s \in \lambda$ be a box of its corresponding Young diagram. For such a box,  we write $s  = (s_1,s_2) \in \BN^2$, labelling the grid position of its bottom corner, and we define the content map $[s] := s_1 \vareps_1+ s_2 \vareps_2 $. Let $j_\lambda \in \Lambda \otimes \BCe$ be the homogenous versions (c.f. \ref{homogjackdefn}) of the integral Jack symmetric functions $J_\lambda$ from \cite{Macdonald:1995}, which we review in section \ref{jackdefinition}. The Jack Littlewood-Richardson (LR) coefficients $c_{\mu,\nu}^{\lambda}$ are defined as the coefficients of a product of Jack functions expanded in the basis of Jacks:
\begin{equation}\label{introjacklr}
j_\mu \cdot j_\nu = \sum_\lambda c_{\mu,\nu}^{\lambda} \, j_\lambda.
\end{equation}

In this paper, we prove the following `sum-product' combinatorial identity that captures deep structure of these Jack Littlewood-Richardson coefficients:
\begin{theoremstar}[Theorem \ref{mainLRhtheorem}]
For any partitions $\mu, \nu$, the Jack Littlewood-Richardson coefficients $ c_{\mu\nu}^{\gamma}$ satisfy the following identity of rational functions in the variable $u$,
\begin{equation}\label{introtheorem}
\sum_{\gamma \supset \mu \cup \nu}  c_{\mu\nu}^{\gamma} \frac{\jacktop_{\gamma}}{\jacktop_\mu \jacktop_\nu}\left( \sum_{s \in \gamma/(\mu \cup \nu)} \frac{1}{u-[s]} \right)  = T_{\mu\* \nu}(u) - 1,
\end{equation}
where $\mu \cup \nu$ is the union as sets of boxes, $\jacktop_\lambda \coloneqq  \prod_{s \in \lambda \setminus (0,0)} [s]$, and
\begin{equation}
T_{\mu\* \nu}(u) := \prod_{x \in \mu, y\in \nu} N(u - [x+y]), \qquad N(u) := \frac{(u-[0,0])(u-[1,1])}{(u-[1,0])(u-[0,1])}.
\end{equation}
\end{theoremstar}

In the case $\mu = 1$, this theorem recovers a well known result of Kerov \cite{Kerov:2000}:
\begin{equation}
\sum_{\nu + s \supset \nu}  c_{1\nu}^{\nu+s} \left( \frac{1}{u-[s]} \right)  = u^{-1} T_{\nu}(u).
\end{equation}

By expanding at various poles in $u$, the identity \ref{introtheorem} gives a family of relations amongst the $c_{\mu\nu}^{\gamma}$. We provide a simple yet illustrative example in \ref{mainthmexample}. Although these equations are underdetermined, they do provide explicit closed form expressions for large families of Jack LR coefficients, which we investigate in a follow up article w/ P. Alexandersson \cite{Alexandersson:2023} along with connections to various conjectures of Stanley on the structure of these coefficients \cite{Stanley:1989}.

We repackage and interpret the above result in terms of a simple map:
\begin{theoremstar2}[Theorem \ref{thm:basicevaluationmap}]
Consider the following `basic' evaluation map on symmetric functions $\Delta: \SymmetricFunctionRing \to \BCe(u)$, defined on the basis of homogeneous Jack symmetric functions $j_\lambda$  as
\begin{equation}\label{deltadefinition}
\dtrace(j_\lambda) := \jacktop_\lambda \, \sum_{s \in \lambda} \frac{1}{u-[s]}.
\end{equation}
For two partitions $\mu, \nu$ of arbitrary size, this evaluation map satisfies
\begin{equation}
\dtrace(j_\mu \cdot j_\nu) = \jacktop_\mu \jacktop_\nu \left(T_{\mu\* \nu}(u) - 1 \right).
\end{equation}
\end{theoremstar2}

Note that the map $\Delta$ is \emph{not} a ring homomorphism, and furthermore it degenerates in the Schur case ($\vareps_1+\vareps_2 =0$) as in this case it vanishes on all non-hook partitions.

This paper is the sequel to \cite{Mickler:2022}, where a spectral theorem was proven for the quantum Lax operator $\CL$ introduced by Nazarov-Sklyanin \cite{Nazarov:2013}. The polynomial eigenfunctions $\psi_\lambda^s \in \Lambda[w]$ of $\CL$ depend on a partition $\lambda$ and a choice of location $s$ where a box can be added to $\lambda$.

The central idea of this second paper is to consider product expansions of these Lax eigenfunctions
\begin{equation}
\psi_\lambda^s \cdot \psi_\nu^t = \sum_{\gamma,u} c_{\lambda,\nu;u}^{s,t;\gamma} \psi_\gamma^u,
\end{equation}
and analyse their structure.
Here, we introduce a new object, the \emph{\bf Jack-Lax Littlewood-Richardson coefficients}, $c_{\lambda,\nu;u}^{s,t;\gamma}$, the structure of which will be illuminated throughout this paper. Indeed, these Jack-Lax LR coefficients reproduce the Jack LR coefficients (\ref{introjacklr}) under summation, 
\begin{equation}
c_{\lambda,\nu}^{\gamma} = \sum_{u} c_{\lambda,\nu;u}^{s,t;\gamma}.
\end{equation}
We will demonstrate that in many ways this refined algebra of eigenfunctions is \emph{easier to understand} than the algebra of Jack functions due to the action of $\CL$, and leads ultimately to a proof of the Main Result \ref{mainLRhtheorem}. 

\subsection{Organization of the paper}

In Section \ref{preliminariessection}, we review some preliminary material, and recall the main results of the previous paper in this series \cite{Mickler:2022}. We introduce the Nazarov-Skylanin Lax operator, and describe is spectrum.

In Section \ref{structionsection1}, we begin the task of understand the structure of the basis of Lax eigenfunctions. Here, we lay out the central new focus of this work, which is the algebra of products of these Lax eigenfunctions. We introduce a family of linear maps, the Trace functionals, that help to explore the properties of the Lax eigenfunction products. These traces are associated with three different decompositions of the Hilbert space. We describe a cohomological approach to the understanding of the combined Trace map and compute its kernel and cokernel. 

In Section \ref{structionsection2}, we produce special elements of the algebra, the $\beta$ and $\theta$ elements, and show a key relation between their traces. We then use this relation to compute the traces of these elements, demonstrating a connection to the Jack Littlewood-Richardson coefficients. We conclude this section with the main theorem (\ref{mainLRhtheorem}).

In Section \ref{shcsection}, we reinterpret the main results in terms of the language of double affine Hecke algebras, motivated by the results of Bourgine-Matsuo-Zhang \cite{Bourgine:2015tv}.

In Section \ref{twistssection}, we close out this work with some conjectures on the deeper structure of the algebra of Lax eigenfunctions. These conjectures would give a more direct explanation of the central results of this article. 

\subsection{Acknowledgements}

The author would like to greatly thank Alexander Moll for introducing him to the key concepts in the work of Nazarov-Skylanin over five years ago, and for many long helpful discussions and his contributions to this work and feedback with editing the drafts of this paper. The author also wants to thank Jean-Emile Bourgine for illuminating discussions on the $\SHc$ and its holomorphic presentation, and Per Alexandersson for helpful discussions on matters of combinatorics.


\section{Preliminaries}\label{preliminariessection}

\subsection{Combinatorics}

\subsubsection{Partitions}\label{addsetdefn}

A partition $\lambda = (\lambda_1, \lambda_2, ...)$ is a sequence of non-increasing non-negative integers with a finite number of nonzero terms. The \emph{size} of a partition is denoted $|\lambda|$. We often use condensed partition notation, e.g. $(1,1,2,3,3) = \{1^2,2,3^2\}$. For a partition $\lambda$, we write $b \in \lambda$ to index the \emph{boxes} of the Young diagram of $\lambda$. We represent boxes by their lower left corner $b = (i,j) \in \BZ^2$, where $0 \leq j \leq \lambda_i-1$. We denote by $\lambda^\times$ the collection of boxes $\lambda^\times := \{ b\in \lambda : b\neq (0,0)\}$. 

Let $\vareps_1 < 0 < \vareps_2 \in \BR$ be parameters\footnote{These are the equivariant Omega background parameters, c.f. \cite{Nekrasov:2006}.}.
For $s = (s_1,s_2) \in \BZ^2$ a box, we denote the \emph{box content} by \begin{equation}
[s] = [s_1,s_2] := s_1 \vareps_1+s_2\vareps_2
\end{equation}
For a partition $\lambda$, we use the following standard conventions.
For a box $b \in \lambda$ let $h^U(b)$ (resp. $h^L(b)$) be the upper (resp. lower) hook length of the box (see e.g. Stanley \cite{Stanley:1989}). For example $h^U_{1^22^23}((2,0))=[2,-2]$. Let $\lambda'$ denote the transposed partition to $\lambda$. Let $\addset_\lambda$ (resp. $\remset_\lambda$) be the collection of boxes that can be added, the \emph{add-set}, (resp. removed, the \emph{rem-set}) from $\lambda$. We use the notation $\remsetp_\lambda = \{ b + [1,1] : b \in \remset_\lambda \}$, to indicate the \emph{outer} corners of the boxes that can be removed.
In this paper, we draw partition \emph{diagrams} (and their associated partition \emph{profiles}) in the Russian form, following the notation of \cite{Nekrasov:2006}. In this way, the elements of $\addset_\lambda$ (resp $\remsetp_\lambda$) correspond to minima (resp. maxima) of the partition profile, as illustrated by the figure \ref{partitionfig}.
 [FIX]
\begin{figure}[htb]
\centering
\includegraphics[width=0.9 \textwidth]{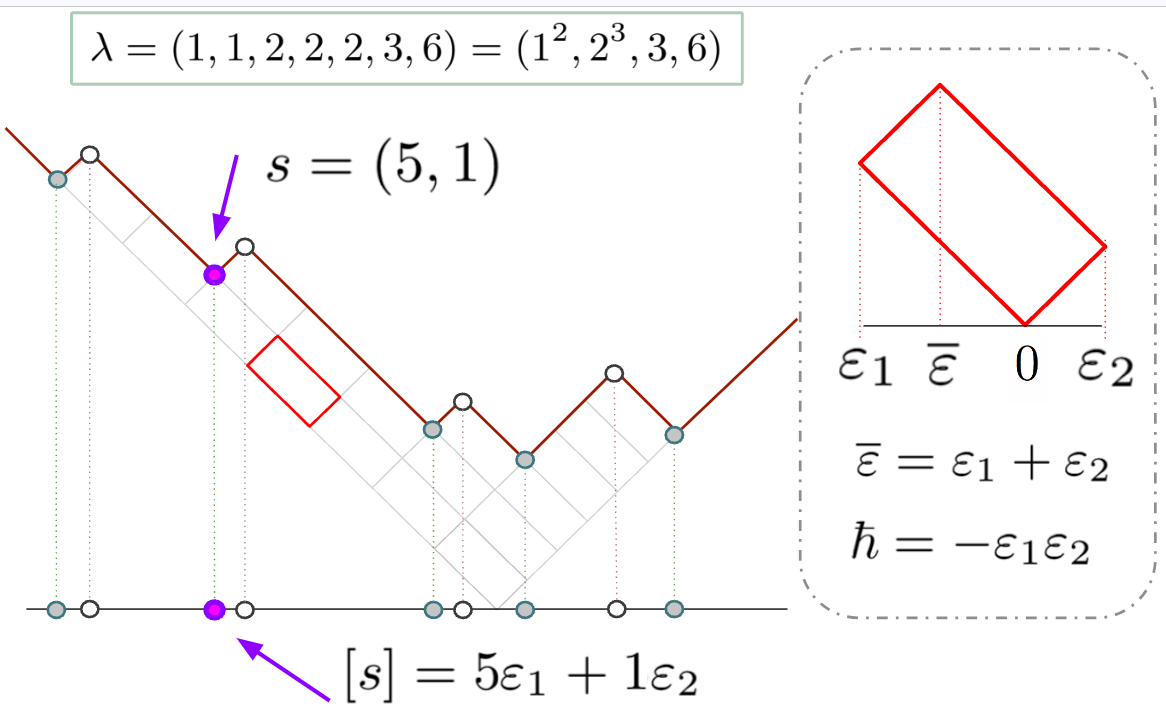}\\
\caption{The basic objects of our notation for partitions.}
\label{partitionfig}
\end{figure}

\subsubsection{Symmetric Functions}

We refer to the canonical source \cite{Macdonald:1995} for all foundational results. Consider the ring of symmetric functions $\Lambda := \BC[x_i]^{\mathfrak{S}}$ in infinitely many variables $x_i$. Define the power sum symmetric functions as $p_k = \sum_{i}{x_i^k}$. For $\mu = (\mu_1,\mu_2,\ldots)$ a partition, we write $p_\mu = \prod_k p_{\mu_k}$, and denote the monomial symmetric functions $m_{\mu} =  x^{\mu} + \ldots$. . For $\alpha \in \BR$, we define the \emph{$\alpha$-deformed Hall inner product}, by

\begin{equation}
\langle p_\mu, p_\nu \rangle_\alpha := \delta_{\mu,\nu} z_\mu \alpha^{|\mu|}, \text{ where } z_\mu = \prod_k (\mu_k!\, k^{\mu_k}).
\end{equation}

\subsubsection{Jack Functions}

Define the real \emph{deformation parameter} $\alpha = - \vareps_2/\vareps_1$. The (integral form) Jack symmetric functions $J_\lambda^{(\alpha)}$, indexed by partitions $\lambda$, are a family of symmetric functions depending on the deformation parameter $\alpha$, introduced in \cite{Jack:1970}. When $\alpha = 1$, these reduce to the standard Schur symmetric functions $s_\lambda$. 
\begin{proposition}[Jack Functions \cite{Macdonald:1995} Chap. VI, (4.5)]\label{jackdefinition}
There exists a unique family of symmetric functions $J_\lambda \in \Lambda[\alpha]$, indexed by partitions $\lambda$, which satisfy the following three properties
\begin{itemize}
\item Orthogonality:  
\begin{equation}
\langle J_\lambda, J_\mu \rangle_\alpha = 0 \text{ when } \lambda \neq \mu.
\end{equation}
\item Triangularity: 
\begin{equation}
J_\lambda = \sum_{\mu <_d \lambda} c_{\lambda\mu} m_\mu,
\end{equation}
where $<_d$ indicates dominance order on partitions. 
\item Normalization: 
\begin{equation}
[m_{1^n}] J_\lambda = n!.
\end{equation}
\end{itemize}
\end{proposition}
This normalization is known as the \emph{integral} form of the Jack symmetric functions, as they have the property that $J_\lambda^{(\alpha)} \in \BZ[\alpha, p_1,p_2, \ldots]$ in the basis of the power-sum symmetric functions with the expansion
\begin{equation} J_{\lambda}^{(\alpha)} = 1. p_1^{|\lambda|} + \ldots \end{equation}

For example, for partitions of size $n=3$, the Jack functions are
\begin{eqnarray}
J_{\{1^3\}}^{(\alpha)}(p) =& p_1^3-3p_2 p_1 +2 p_3 &= 6 m_{1^3}, \\
J_{\{1,2\}}^{(\alpha)}(p) =& p_1^3 + (\alpha-1) p_2p_1 - \alpha p_3 &= 6 m_{1^3} + (\alpha+2) m_{1,2},  \\
J_{\{3\}}^{(\alpha)}(p) =& p_1^3 +3 \alpha p_2 p_1 + 2 \alpha^2 p_3  &= 6 m_{1^3} + 3(\alpha+1)m_{1,2}  + (\alpha+1)(2\alpha+1) m_{3}.
\end{eqnarray}
The triangularity of the Jack functions in the monomial basis is evident.

\subsubsection{Fock Module}
We use the ring of coefficients $\BCe = \BC(\vareps_1,\vareps_2)$. Out of the two deformation parameters, we build two secondary parameters: The \emph{quantum} parameter $\hbar = -\vareps_1\vareps_2= -[(1,0)][(0,1)]$, and the \emph{dispersion} parameter $\bareps = \vareps_1+\vareps_2= [(1,1)]$.  We consider a $\hat {\mathfrak{gl}}_1$ Heisenberg current $V(z) = \sum_k V_k z^k$, with $V_0 = 0$ and $[V_n, V_m] = \hbar n \delta_{n+m,0}$. This current acts on the Fock module $\CF = \BCe[V_1,V_2,\cdots ]$, via
\begin{equation} V_{-k} = \hbar k \partial_{V_k},\quad k>0. \end{equation}
In this paper, we work with an alternate presentation of the ring of symmetric functions by embedding them into the Fock module via $p_k \to (-\vareps_1)^{-1}V_k$.
In this basis, the Hall inner product becomes
\begin{equation}
\norm V_1^{d_1}V_2^{d_2}\cdots \norm_\hbar^2 = \prod_{k=1}^{\infty} (\hbar k)^{d_k} d_k!.
\end{equation}
With this, we have
\begin{equation} V_{-k} = V_k^\dag. \end{equation}
This ring has the natural grading operator $\CN$, where $V_k$ has degree $k$.

\subsubsection{Homogeneous Integral Normalization}
We will use \emph{homogeneous} normalization of the integral Jack functions (henceforth denoted with a lower case $j$), considered as elements in the Fock module $\CF$, given by: 
\begin{equation}\label{homogjackdefn}  j_\lambda(V | \vareps_1, \vareps_2) :={(-\vareps_1)^{|\lambda|}} \cdot J^{(\alpha = - \vareps_2/\vareps_1)}_{\lambda}( p=(-\vareps_1)^{-1}V) \in \CF.\end{equation}
With this normalization, the three homogeneous Jack functions for $n=3$ are given by
\begin{eqnarray}j_{\{1^3\}}  =& V_1^3+ 3\vareps_1 V_{1}V_{2}+ 2\vareps_1^2 V_3 ,\\
 j_{\{1,2\}}  =& V_1^3+ (\vareps_1+\vareps_2) V_{1}V_{2}+ \vareps_1\vareps_2 V_3 ,\\
j_{\{3\}}  =& V_1^3+ 3\vareps_2 V_{1}V_{2}+ 2\vareps_2^2 V_3  .
\end{eqnarray}
Note that these are all homogenous integral polynomials in $\BZ[\vareps_1,\vareps_2, V_1, V_2, \ldots]$, and that in this normalization we have the explicit transpositional symmetry:
\begin{equation}
j_\lambda(V| \vareps_1 ,\vareps_2) = j_{\lambda'}(V|\vareps_2,\vareps_1).
\end{equation}

\begin{lemma}[Principal Specialization \cite{Stanley:1989} Thm 5.4] For a partition $\lambda$, we have
\begin{equation}
j_\lambda(V_i=z) = \prod_{ b \in \lambda} (z+[b]),
\end{equation}
and hence
\begin{equation}\label{jacktopdef}
\jacktop_\lambda := [V_n] j_\lambda = \prod_{b \in \lambda^\times} [b].
\end{equation}
\end{lemma}


\subsubsection{Jack Littlewood-Richardson (LR) Coefficients}
Much of the work in this paper will be concerning the \emph{Jack Littlewood-Richardson coefficients}, $c_{\mu,\nu}^{\lambda}$, defined as the expansion coefficients in the product of Jack functions,
\begin{equation}
j_\mu \cdot j_\nu = \sum_\lambda c_{\mu,\nu}^{\lambda} \, j_\lambda.
\end{equation}
In the literature, these are often denoted as $c_{\mu,\nu}^{\lambda}(\alpha)$ to indicate the dependence on the deformation parameter $\alpha$. In general, it is very difficult to find explicit closed-form expressions for these coefficients, instead most formulas involving them are recursive in nature. In this paper, we find new families of relations between these coefficients. In the sequel paper, \cite{Alexandersson:2023}, we make progress towards finding explicit closed form expressions.

The most well known explicit formula for Jack LR coefficients is given by

\begin{proposition}[Pieri Rule - Stanley '89 \cite{Stanley:1989} Thm 6.1]
Let $\mu \subset \lambda$, and $\lambda/\mu$ be a \emph{horizontal $r$-strip}, i.e. no two boxes in the quotient shape $\lambda/\mu$ are adjacent in a row. Then
\begin{equation}\label{stanleypieri}
c_{1^r,\mu}^{\lambda} = \frac{\left(\prod_{s \in \{1^r\}} h^L_{1^r}(s) \right) \left(\prod_{t \in \mu } A_{\mu}(t)\right)  }{\left(\prod_{v \in \lambda } A_{\lambda}(v)\right) },
\end{equation}
where
\begin{equation}
A_\sigma(b) = h^U_\sigma(b)  \text{  if $\lambda/\mu$ does not intersect the row of $b$, $h^L_\sigma(b)$ otherwise. }
\end{equation}
\end{proposition}


\subsection{The Nazarov-Skylanin Lax Operator}\label{laxoperatorsection}

In this section, we recall the work of the previous paper in this series, \cite{Mickler:2022}, which explores the extraordinary quantum Lax operator introduced by Nazarov-Skylanin in \cite{Nazarov:2013}.

\subsubsection{Preliminaries}

We enlarge our Hilbert space and work in the extended Fock module $\CH = \CF \otimes \BC[w]$, where $w$ is of degree 1. The inner product is
\begin{equation}
\langle V_\mu w^m, V_\nu w^n\rangle = \delta_{n,m} \langle V_\mu, V_\nu \rangle_\hbar.
\end{equation}
In this way, the inclusion $\CF \hookrightarrow \CH$ is an isometry.
The total grading operator for $\CH$ is
\begin{equation}\label{fullgradingoperator}
\CN^* := \CN + w \partial_w= \hbar^{-1} \sum_{k > 0} V_kV_{-k} + w \partial_w.
\end{equation}
This gives the graded decomposition
\begin{equation}
\CH = \bigoplus_{n \geq 0} \CH_n.
\end{equation}

On this space, we have several important projections on $\CH$. Firstly, $\pi_0 : \CH \to \CF \subset \CH$ projects just onto the $w^0$ component. $\pi_+$ is its complement, projecting only onto positive powers of $w$. $\pi_w : F[w,w^{-1}] \to F[w]$ is the map that projects onto non-negative powers of $w$.

\subsubsection{Lax Operator}

We at last come to introducing the main actor in our story.
\begin{definition}
The {\bf Nazarov-Sklyanin Lax Operator} \cite{Nazarov:2013} is the linear operator on $\CH = \CF[w]$ given by
\begin{equation}\label{NSlaxdef}
\CL = \pi_{w} \sum_{k >0} \left( w^{-k} V_k + w^k  V_{-k}\right) + \bareps w\partial_w
\end{equation}
\end{definition}
In the basis $\CH = \oplus_k (w^k \CF)$, we can express $\CL$ as the semi-infinite matrix operator, with coefficients in $\End(\CF)$,

 \begin{equation}
\CL := \left(\begin{matrix}
0 & V_1 & V_2 & V_3 & \cdots \\
 V_{-1} & \bareps & V_1 & V_2 & \cdots \\
 V_{-2} &  V_{-1}  & 2\bareps & V_2 & \cdots \\
 V_{-3} &  V_{-2} &  V_{-1} & 3\bareps & \cdots \\
\vdots & \vdots & \vdots & \vdots & \ddots \\
\end{matrix}\right)
\end{equation}

One can check that $\CL$ commutes with grading operator \ref{fullgradingoperator}, so let $\CL_{n} = \CL |_{\CH_n}$. 
Furthermore, let $\CL^+ = \pi_+\CL|_{\CH_+}$ be the restrictions of $\CL$ to only the positive powers of $w$.

\begin{corollary}[Shift property]\label{epsshift}
\begin{equation}
w^{-1}\CL_{n+1} ^+w  = \CL_n + \bareps
\end{equation}
\end{corollary}

Note that from the definition \ref{NSlaxdef} its clear that $\CL$ acts as derivation if either of the factors is in $\CH^0 := \pi_0 \CH = \CF$, i.e.
\begin{equation}\label{Lderivation}
\CL(\zeta \cdot \xi) = \CL(\zeta) \cdot \xi +\zeta \cdot \CL( \xi) , \qquad \text{ if } \zeta \in \CH^0.
\end{equation}

\subsubsection{Spectral Factors I}\label{spectralfactorssection}


We will make extended use of the following rational functions that are associated to partitions $\lambda$.

\begin{equation}\label{spectralfactors1} T_\lambda(u) := \prod_{s \in \lambda} N(u-[s]) , \qquad \text{ where }
N(u) = \frac{(u-[0,0]) ( u - [1,1]) }{(u-[1,0])(u-[0,1])}.
\end{equation}

For example, in the simplest case $T_{1}(u)$ has a zero at the top and bottom corners and poles at each of the side corners, illustrated in figure \ref{figN1}.
\begin{figure}[htb]
\centering
\includegraphics[width=0.1 \textwidth]{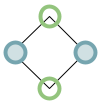}\\
\caption{$T_{1}(u)  = N(u) = \frac{(u-[0,0])(u-[1,1])}{(u-[1,0])(u-[0,1]) } $.}
\label{figN1}
\end{figure}
\begin{figure}[htb]
\centering
\includegraphics[width=0.2 \textwidth]{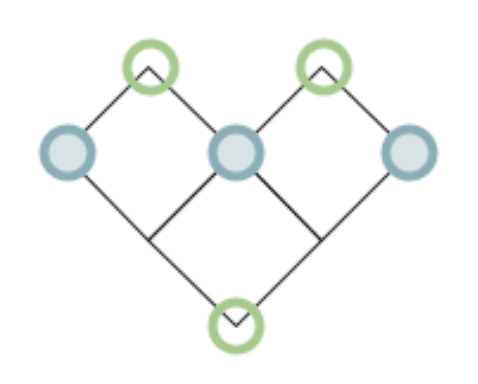}\\
\caption{$T_{\{1,2\}}(u)  = N(u)N(u-[1,0])N(u-[0,1]) = \frac{(u-[0,0])(u-[2,1])(u-[1,2])}{(u-[2,0])(u-[1,1])(u-[0,2]) } $.}
\label{figN2}
\end{figure}

Note that we have the cancellations of poles and zeros on the internal corners of the partition, and we are left with poles at the `inner' corners, and zeros at the `outer' corners, that is,
\begin{equation}
T_\lambda(u) = u \cdot \frac{ \prod_{t \in \remsetp_\lambda}(u-[t]) } {\prod_{s \in \addset_{\lambda}}(u-[s]) }.
\end{equation}

\subsubsection{Integrable Hierarchy}

We consider the following `transfer' operator for $\CL$,
\begin{equation} \label{NStransferop}
\CT(u) := \pi_0 (u-\CL)^{-1} : \End(\CF) \otimes \BC(u).
\end{equation}

The motivating result for most of this work is the following remarkable property of the Lax operator $\CL$.
\begin{theorem}[ Nazarov-Sklyanin (2013) \cite{Nazarov:2013}  ]
The transfer operator $\CT(u)$ is diagonalized on the homogenous Jack functions $j_\lambda \in \CF$,
\begin{equation}\label{NSspectralthm}
\CT(u)\, j_\lambda = u^{-1} T_\lambda(u) \cdot j_\lambda
\end{equation}
where $T_\lambda(u)$ is the spectral factor \ref{spectralfactors1}.
\end{theorem}

\subsubsection{Transition Measures}
The influential work of Kerov \cite{Kerov:2000} introduces the following objects.
\begin{definition}[\cite{Kerov:2000} eq. (3.1)]
For $s \in \addset_{\lambda}$, define the \emph{co-transition measure}
\begin{equation}
\tau_\lambda^s  := \Res_{u=[s]} u^{-1} T_\lambda(u)= \frac{\prod_{t \in \remsetp_\lambda}[s-t]}{\prod_{s' \in \addset_{\lambda}, s'\neq s}[s-s']}.
\end{equation}

\end{definition}

Note that for $\vareps_2 < 0 < \vareps_1$, it can be shown that $\tau_\lambda^s > 0$ and  $\sum_{s \in \addset_\lambda} \tau_\lambda^s = 1$, hence these coefficients define a probability measure on the add-set of $\lambda$. A connection between these measures and Jack LR coefficients was shown by Kerov.

\begin{lemma}[Kerov '97 \cite{Kerov:2000} thm. (6.7)]
The simplest Jack LR coefficient coincides with the co-transition measure,
\begin{equation}\label{kerovcoefficient}
c_{1,\lambda}^{\lambda+s} = \tau_\lambda^s.
\end{equation}
That is, Jack functions satisfy the following simple multiplication formula (`Pieri' rule)
\begin{equation}\label{jackpieri}
j_{1} \cdot j_\lambda = \sum_{s \in \addset_{\lambda}} \tau_\lambda^s \, j_{\lambda+s}.
\end{equation}
\end{lemma}
From this, we can write
\begin{equation}\label{kerovLRrule}
u^{-1} T_\lambda(u) = \sum_{s \in \addset_{\lambda}} c_{1,\lambda}^{\lambda+s} \frac{1}{u-[s]}.
\end{equation}
One of the main results (Theorem \ref{mainLRhtheorem}) of this paper is a generalization this Kerov relation between Jack LR coefficients and residues of certain `spectral' rational functions.

Kerov also introduces the \emph{transition measures}, for $t \in \remsetp_\lambda$,
\begin{equation}
\tilde \tau^{t}_\lambda = \Res_{u=[t]} u\, T_\lambda(u)^{-1}.
\end{equation}

\subsection{Spectral Theorem }

\newcommand{\TCL}{{\widetilde\CL}}

Here we recall the results of Mickler-Moll \cite{Mickler:2022} on the spectrum of the Nazarov-Skylanin Lax operator.
Fix $n \in \BN$. Consider the $\CL$-cyclic subspaces of $\CH_n$ generated by the Jack polynomials $j_\lambda \in \CH_n^0$.
\begin{definition}
The "Jack-Lax" cyclic subspaces of $j_\lambda$ under $\CL_{n}$ are denoted by
\begin{equation}
\QZ_\lambda := Z(j_\lambda,\CL_n) \subset \CH_{n}.
\end{equation}
\end{definition}
One immediate corollary of the NS theorem (\ref{NSspectralthm}) in this language is
\begin{corollary} $\pi_0 \QZ_\lambda = j_\lambda$.
\end{corollary}
We can now state the main result of the first paper in this series.
\begin{theorem}[Spectral Decomposition \cite{Mickler:2022}]\label{spectralthm} {$\,$}
\begin{itemize}
\item Under the action of the Nazarov-Skylanin Lax operator $\CL$, the space $\CH_n$ has the following cyclic decomposition
\begin{equation}\label{Zdecomp}
\CH_n = \bigoplus_{\lambda \partition n } \QZ_\lambda.
\end{equation}
under which $\CL$ acts in block diagonal form $\CL = \oplus \CL_\lambda$, where $\CL_\lambda = \CL_n | \QZ_\lambda$.
\item The eigenfunctions of $\CL_\lambda$ on $\QZ_\lambda$ are labelled $ \{ \psi_\lambda^{s} : s \in \addset_{\lambda} \}$ with eigenvalues given by the corresponding box content
\begin{equation}
 \CL\, \psi_\lambda^{s} = [s] \cdot \psi_\lambda^{s}.
\end{equation}
Thus, the cyclic subspace $\QZ_\lambda$ is given as
\begin{equation}
\QZ_\lambda = \Span_{s \in \addset_{\lambda}} \{ \psi_\lambda^{s} \} \subset \CH_{n}.
\end{equation}
\item
The eigenfunctions of $\CL_\lambda^+$ on $\QZ_\lambda^+ := \pi_+ \QZ_\lambda$ are labelled $ \{ \tilde \psi_\lambda^{t} : t \in \remsetp_\lambda \}$ with eigenvalues given by the corresponding box content
\begin{equation}
 \CL^+ \, \tilde\psi_\lambda^{t} = [t] \cdot \tilde\psi_\lambda^{t}.
\end{equation}
\item These eigenfunctions can be normalized to satisfy
\begin{equation}\label{kerovpsiformulae}
\pi_0  \psi_\lambda^{s} = j_\lambda, \qquad \tilde\psi_\lambda^{t} = \frac{1}{[t]-\CL} j_\lambda.
\end{equation}
With this normalization, we have
\begin{equation}\label{jacksum}
j_\lambda = \sum_{s \in \addset_{\lambda}} \tau_{\lambda}^{s}  \psi_{\lambda}^{s} \in \QZ_\lambda,
\quad\text{   where   }\quad
\tau_\lambda^s := \Res_{u=[s]} \left( u^{-1} T_{\lambda}(u) \right).
\end{equation}
\end{itemize}
\begin{equation}\label{jacksum2}
\CL_n j_\lambda = \sum_{t\in \remsetp_\lambda} \ttau_\lambda^t  \tpsi_\lambda^t \in \QZ_\lambda^+, 
\quad\text{   where   }\quad
\ttau_\lambda^t := \Res_{u=[t]} \left( u T_{\lambda}(u)^{-1} \right).
\end{equation}
\end{theorem}
We also recall the following useful result:
\begin{lemma}[Principal Specialization \cite{Mickler:2022}]
\begin{equation}\label{psiprincipalspec}
\psi_\lambda^s(z,1) = \prod_{b\in (\lambda+s)^\times}(z+[b]),
\end{equation}
and hence
\begin{equation}\label{psiprincipalspecialization}
[w^{|\lambda|}]\psi_\lambda^s(V,w) =  [s]\varpi_\lambda.
\end{equation}
\end{lemma}

\section{Three Decompositions}\label{structionsection1}

In this section, we begin the new work of this paper.

As mentioned in the introduction, our primary objective will be to provide a new approach the classical problem of understanding the structure coefficients of products of Jack functions
\begin{equation}
j_\lambda \cdot j_\nu = \sum_\gamma c_{\lambda,\nu}^{\gamma} j_\gamma.
\end{equation}
The central claim of this paper is that by considering the structure of the algebra of Lax eigenfunctions
\begin{equation}\label{psialgebra}
\psi_\lambda^s \cdot \psi_\nu^t = \sum_{\gamma,u} c_{\lambda,\nu;u}^{s,t;\gamma} \psi_\gamma^u,
\end{equation}
we will gain insight into the products of Jack functions. The \emph{Jack-Lax Littlewood-Richardson} coefficients, $c_{\lambda,\nu;u}^{s,t;\gamma}$, will be illuminated throughout this paper. This algebra reproduces the Jack LR coefficients under the projection to $\pi_0$, and hence
\begin{equation}
c_{\lambda,\nu}^{\gamma} = \sum_{u\in \addset_\gamma} c_{\lambda,\nu;u}^{s,t;\gamma}.
\end{equation}

We will build up towards goal of understanding these Jack-Lax LR coefficients, by first developing a structural theory for $\psi_\lambda^s$.

\subsection{Norm Formulae}\label{normstanleysubsection}
We recall an important formula of Stanley for the norm squared of Jack functions.  \begin{proposition}[Stanley \cite{Stanley:1989} Thm 5.8]
\begin{equation}\label{stanleynormformula}
| j_\lambda |^2 = \prod_{b \in \lambda} h^U_\lambda(b) h_\lambda^L(b) 
 \end{equation}
\end{proposition}

We next prove a similar formula for Lax eigenfunctions. Let $c_b(\lambda)$ (resp. $r_b(\lambda)$) be the subset of boxes of $\lambda$ that are in the same column (resp. row) as $b$.

\begin{lemma}[$\psi$ Norm Formula]
\begin{equation}\label{psinormformula}
|\psi_\lambda^s|^2  = \prod_{b \in \lambda} C^U_{\lambda,s}(b) C_{\lambda,s}^L(b) 
\end{equation}
where
\begin{equation}
C^U_{\lambda,s}(b) = h^U_\lambda(b)  \text{  if $b \notin c_s(\lambda)$, $h^U_{\lambda+s}(b)$ otherwise. }
\end{equation}
\begin{equation}
C^L_{\lambda,s}(b) = h^L_\lambda(b)  \text{  if $b \notin r_s(\lambda)$, $h^L_{\lambda+s}(b)$ otherwise. }
\end{equation}
\end{lemma}
\begin{proof}
As observed in the first paper in this series, the two ways of expanding the expression $\langle \psi_\lambda^s, j_\lambda \rangle$ lead to the formula
\begin{equation}
|\psi_\lambda^s|^2  = \frac{|j_\lambda|^2}{\tau_\lambda^s}.
\end{equation}
Using Kerov's identity \ref{kerovcoefficient} and Stanley's Pieri formula \ref{stanleypieri}, we get
\begin{equation}
\tau_\lambda^s = c_{1,\lambda}^{\lambda+s} = \frac{ h_{1}^L((0,0)) \left(\prod_{b \in \lambda } A_{\lambda}(b)\right)  }{\left(\prod_{b \in \lambda+s } A_{\lambda+s}(b)\right) }.
\end{equation}
Expanding this out, the factors not in the row-column shared with $s$ cancel, and by using $h_{1}^L((0,0)) = h_{\lambda+s}^L(s)$, we get
\begin{equation}
\tau_\lambda^s =  \frac{  \left(\prod_{b \in r_s(\lambda) } h^L_{\lambda}(b)\right)  }{\left(\prod_{b \in r_s(\lambda) } h^L_{\lambda+s}(b)\right) }\frac{  \left(\prod_{b \in c_s(\lambda) } h^U_{\lambda}(b)\right)  }{\left(\prod_{b \in c_s(\lambda) } h^U_{\lambda+s}(b)\right) },
\end{equation}
and the result follows from \ref{stanleynormformula}.
\end{proof}

{\bf[CLAIM?] }
\[ \sum_s \tau_\lambda^s \prod_{b \in \lambda} C^U_{\lambda,s}(b)  = \prod_{b \in \lambda} h^U_{\lambda,}(b) \]

\begin{example}
We can compute (using $(\vareps_1,\vareps_2) = (X,Y)$ for readability)
\tiny
\begin{eqnarray*}
\psi_{1,2^2}^{(2,1)} &=& w^0 (V_1^5+2(2X+Y)V_1^3V_2 + (3X^2+XY+Y^2)V_2^2 V_1 +2X(X+3Y) V_1^2 V_3 \\
&& \quad \quad + 2X(X^2+Y^2)V_2V_3 + XY(7X+Y)V_1V_4+ 2X^2 Y(X+Y)V_5) + \\
&& w^1 ((2X+Y)V_1^4+2(3X^2+XY+Y^2)V_1^2V_2 + Y(5X^2-3XY+Y^2)V_2^2 \\
&& \quad  \quad+ 4X(X^2+Y^2)V_1 V_3 + XY(4X^2-XY+Y^2) V_4) + \\
&& w^2 (2X(X+3Y)V_1^3+6X(X^2+Y^2)V_1V_2  +2X(2X^2-3XY+3Y^2)V_3 ) + \\
&& w^3 (2XY(7X+Y)V_1^2+2XY(4X2-XY+Y^2)V_2  ) + \\
&& w^4 ( 10X^2Y(X+Y)V_1  ) + \\
&& w^5 ( 2X^2Y(2X^2+3XY+Y^2) ).
\end{eqnarray*}
\normalsize
We also use the Stanley Pieri formula (\ref{stanleypieri}) to compute,
\begin{equation}
|j_{1,2^2}|^2 = [1,0][2,-1][3,-1][1,0][2,0]\cdot [0,-1][1,-2][2,-2][0,-1][1,-1],
\end{equation}
where we have grouped as L/U hooks. Then (using red to highlight the hooks that have changed), we have
\begin{equation}
|\psi_{1,2^2}^{(2,1)}|^2 = {\color{red}[1,-1]}[2,-1][3,-1][1,0][2,0]\cdot [0,-1][1,-2][2,-2]{\color{red}[1,-1][2,-1]}.
\end{equation}
\end{example}

\subsection{Action of $w$}

Next, we investigate the relationship between eigenvalues of different degrees via the action of multiplication by $w : \CH_n \to \CH_{n+1}$ in terms of the basis of Lax-eigenfunctions $\psi$. We begin with a $\CL_{n}$ eigenfunction $ \psi_{\gamma}^{t} $, for $\gamma \partition  n$. We note that the shift property $w^{-1} \CL_{n+1}^+ w = \CL_{n} + \bareps$ (\ref{epsshift}), yields
\begin{equation}\label{winvpsispect}
\CL^+_{n+1} \left( w \psi_{\gamma}^{t} \right) = w (\CL_{n}+\bareps) \psi_{\gamma}^{t}= [t+(1,1)] w \psi_{\gamma}^{t} .
\end{equation}
That is, $w \psi_{\gamma}^{t} \in \CH_{n+1}$ is in the $[t+(1,1)]$ eigenspace of $\CL^+_{n+1}$. First, we determine precisely what eigenfunction this is, as this eigenspace is generically greater than one dimensional. The spectrum of $\CL^+_{n+1}$ was provided in the spectral theorem \ref{spectralthm}.

\begin{theorem}[Shift Theorem]\label{waction}
We have followed equality of eigenfunctions of $\CL^+_{n+1}$,
\begin{equation}\label{tildepsiw}
w\, \psi_{\gamma}^{t} = \tilde \psi_{\gamma+t}^{t+(1,1)}.
 \end{equation}
\end{theorem}

\begin{proof} We prove by induction on size of the partition $\gamma$. For the base case, we check $ \tilde\psi_{1}^{(1,1)} = w = w \psi_{\emptyset}^{(0,0)}$. For the inductive step, assume equation (\ref{tildepsiw}) holds for all $\gamma$ with $|\gamma|\leq n $.

We begin with the Pieri rule (\ref{jackpieri}), for $\lambda \partition n$:
\begin{equation}
j_{1} j_{\lambda} = \sum_{s\in \addset_{\lambda}} \tau_{\lambda}^s\, j_{\lambda+s}.
\end{equation}
We first act with $\CL$ on both sides of this equation. In general $\CL$ is not a derivation. However, when acting on terms of degree zero, it is (\ref{Lderivation}).  After dividing both sides by $w$, this yields
\begin{equation}\label{qjformula}
q_{1}   j_{\lambda} + j_{1} q_\lambda = \sum_{s\in \addset_{\lambda}} \tau_{\lambda}^s\, q_{\lambda+s},
\end{equation}
where we have used the following definition for $\gamma \partition (n +1)$,
\begin{equation}
 q_\gamma := w^{-1}  \CL_{n+1} j_\gamma \in \CH_n.
 \end{equation}
Note, from \ref{jacksum2} we have
\begin{equation}
 q_\gamma = \sum_{t \in \remset_{\gamma}}\tilde \tau_{\gamma}^{t+(1,1)} w^{-1}  \tilde \psi_{\gamma}^{t+(1,1)} \in \pi^+\CH_{n+1}.
\end{equation}

The strategy will be to hit both sides with the projector $P_{[s]}$ onto the $[s]$ eigenspace of $\CL$, for a choice of $s \in \addset_{\lambda}$, and equate the results.

On the right hand side  of (\ref{qjformula}), the only term in the sum is not annihilated by the spectral projector $P_{[s]}$ is $q_{\lambda+s}$, since by (\ref{winvpsispect}) the eigenvalues appearing in the $\psi$ decomposition of some $q_\gamma$ are precisely $\{ [t] : t \in \remset_{\gamma}\}$, and only $\lambda+s$ has a maximum at $s+(1,1)$. Thus, only term that survives the projection is
\begin{equation}
P_{[s]} q_{\lambda+s} = \tilde \tau_{\lambda+s}^{s+(1,1)} w^{-1} \tilde \psi^{s+(1,1)}_{\lambda+s}.
\end{equation}
Now for the left hand side of (\ref{qjformula}), we expand
\begin{equation}
q_{1}   j_{\lambda} + j_{1} q_\lambda = \hbar  \sum_{u\in \addset_{\lambda}} \tau_{\lambda}^u \psi_{\lambda}^u + j_{1}  w^{-1} \sum_{t\in \remset_\lambda} \tilde\tau_{\lambda}^t \tilde\psi^{t+(1,1)}_{\lambda}.
\end{equation}
We then use the inductive hypothesis, $w^{-1}\tilde\psi_{\lambda}^{t+(1,1)}=  \psi_{\lambda-t}^{t}$, to rewrite the last term in this expression
\begin{equation}\label{htaupsuu}
\hbar \sum_{u\in \addset_{\lambda}} \tau_{\lambda}^u \psi_{\lambda}^u + j_{1}  \sum_{t\in \remset_\lambda} \tilde\tau_{\lambda}^t \psi^{t}_{\lambda-t}.
\end{equation}
To expand this further, we need to explicitly expand out $j_{1} \cdot \psi_{\lambda-t}^{t} $. We first note that from the derivation property \ref{Lderivation}, we have
\begin{equation}\label{clsform}
(\CL-[t])\left(  j_{1} \cdot \psi_{\lambda-t}^{t} \right) = \hbar\, w \, \psi_{\lambda-t}^{t}.
\end{equation}
This tells us that
\begin{equation}
 j_{1} \cdot \psi_{\lambda-t}^{t}   = (\CL-[t])^{-1} \hbar \,w \, \psi_{\lambda-t}^{t} + r^{t},
\end{equation}
where $r^{t}$ is the $[t]$ eigenspace of $\CL$. We use the inductive hypothesis a second time, in conjunction with the formula (\ref{kerovpsiformulae}), to write
\begin{equation}\label{kerovformulaforpsitilde}
w \,  \psi_{\lambda-t}^{t} = \tilde\psi_{\lambda}^{t+(1,1)} = \frac{1}{[t+(1,1)]-\CL } j_\lambda = \sum_{s \in \addset_{\lambda}}  \frac{1}{[t+(1,1)-s]} \tau_{\lambda}^s\psi^s_{\lambda} 
\end{equation}
So we find
\begin{equation}\label{jpsidecomp1}
 j_{1} \cdot \psi_{\lambda-t}^{t}   = \sum_{s \in \addset_{\lambda}} \frac{\hbar \, \tau_{\lambda}^s}{[t-s][t+(1,1)-s]}\psi^s_{\lambda}  + r^{t}
\end{equation}
But we are not yet done. We decompose the term $r^{t}$ it into its $\QZ_\gamma$ components, 
\begin{equation}\label{f11decomp}
r^{t} = \sum_{ \gamma\, :\, t \in \addset_{\gamma}} r_\gamma \psi_\gamma^{t}, \quad r_\gamma \in \BC .
\end{equation}
We know only a single $\psi$ can appear in each $\QZ_\gamma$, since the spectrum of $\CL_\gamma := \CL|_{\QZ_\gamma}$ is multiplicity free. Note that no $\QZ_\lambda$ component appears in this decomposition of $r^{t}$, since $t$ is not in $\addset_{\lambda}$. Rather, all the $\QZ_\lambda$ components in equation (\ref{jpsidecomp1}) are in the first term on the right hand side.
We now apply $\pi_0$ to (\ref{jpsidecomp1}) after using the decomposition (\ref{f11decomp}), to get
\begin{eqnarray}
 j_{1} \cdot j_{\lambda-t}  &=& \left( \sum_{s \in \addset_{\lambda}} \frac{\hbar \, \tau_{\lambda}^s}{[t-s][t+(1,1)-s]}\right)j_{\lambda}   +  \sum_{ \gamma\, :\, t \in \addset_{\gamma}} r_\gamma j_\gamma \\
 &=& \tau_{\lambda-t}^{t} j_{\lambda}   +  \sum_{ \gamma\, :\, t \in \addset_{\gamma}} r_\gamma j_\gamma
 \end{eqnarray} 
where we have used the identity (\ref{tausumident}) to evaluate the coefficient of $j_\lambda$. Comparing this result with the Pieri rule (\ref{jackpieri}), we can read off the $r_\gamma$ coefficients to give us
\begin{equation}
r^{t} = \sum_{u\in \addset_{\lambda-t}, u\neq t} \tau_{\lambda-t}^u  \psi^{t}_{\lambda-t+u}.
\end{equation}
After all this we find the explicit expression we sought
\begin{eqnarray}\label{frjpsifirstformula}
j_{1} \cdot \psi_{\lambda-t}^{t} &=& \sum_{v \in \addset_{\lambda}}  \frac{\hbar\, \tau_{\lambda}^v}{[t-v][t+(1,1)-v]}\psi^v_{\lambda} \\
 &&+ \sum_{u\in \addset_{\lambda-t}, u\neq t} \tau_{\lambda-t}^u  \psi^{t}_{\lambda-t+u}.
\end{eqnarray}
Plugging all this back into (\ref{htaupsuu}) we find
\begin{eqnarray}
q_{1} j_\lambda + j_{1} q_\lambda  &=& \hbar \sum_{u\in \addset_{\lambda}} \tau_{\lambda}^u \psi_{\lambda}^u \\
&&+ \sum_{t\in \remset_\lambda} \tilde\tau_{\lambda}^{t+(1,1)}  \sum_{v \in \addset_{\lambda}} \frac{\hbar\, \tau_{\lambda}^v}{[t-v][t-v]}\psi^v_{\lambda} \\
 &&+ \sum_{t\in \remset_\lambda} \tilde\tau_{\lambda}^{t+(1,1)} \sum_{u\in \addset_{\lambda-t}, u\neq t} \tau_{\lambda-t}^u  \psi^{t}_{\lambda-t+u}.
\end{eqnarray}
Because $s \in \addset_{\lambda}$ can not be of the form $t$ for any $t \in \remset_\lambda$, we can see that the only eigenfunction with eigenvalue $[s]$ for our specific choice of $s \in \addset_{\lambda}$ that can appear on the left hand side  is $\psi_\lambda^s$, with coefficient
\begin{eqnarray}
P_{[s]} \left( q_{1} j_\lambda + j_{1} q_\lambda \right) &=& \tau_{\lambda}^s \left( \hbar + \sum_{t \in \remset_\lambda} \tilde \tau_{\lambda}^{t+(1,1)} \frac{\hbar}{[t-s][t+(1,1)-s]} \right) \psi_\lambda^s\\
&=& \tau_{\lambda}^s \left( \tilde \tau_{\lambda+s}^{s+(1,1)} \right) \psi_\lambda^s
\end{eqnarray}
where we have used the identity (\ref{tildetausumident}) to evaluate the sum.

Thus, the result of hitting equation (\ref{qjformula}) with $P_{[s]}$ is the equality,
\begin{equation}
\tau_{\lambda}^s\tilde \tau_{\lambda+s}^{s+(1,1)}\psi_\lambda^s = \tau_{\lambda}^s\tilde \tau_{\lambda+s}^{s+(1,1)} w^{-1} \tilde \psi^{s+(1,1)}_{\lambda+s},
\end{equation}
That is, $w \cdot \psi_\lambda^s =  \tilde \psi^{s+(1,1)}_{\lambda+s}$. Seeing as $s \in \addset_{\lambda}$ was arbitrary, we find that equation (\ref{tildepsiw}) holds for all $\gamma \partition (n+1)$, and so we have completed the inductive step.
\end{proof}

With this result, we no longer need to mention the eigenvalues of $\CL^+$, as they are determined by $w$ and the eigenvalues of $\CL$. In this spirit, we use the result (\ref{kerovformulaforpsitilde}) to determine the action of $w$ in the $\psi$ basis,
\begin{corollary}
The action of $w$ in the basis of Lax eigenfunctions $\psi$ is given by
\begin{equation}\label{wactionpsi}
w\cdot \psi_{\lambda}^{t} = \sum_{s \in \addset_{\lambda+t}} \frac{\tau_{\lambda+t}^s }{[s-t-(1,1)]}\psi_{\lambda+t}^s.
\end{equation}
\end{corollary}

In the remainder of this paper, we will make use of the following elements that appeared in the above proof,
\begin{definition}
Define the `quack' polynomial
\begin{equation}\label{qdef}
 q_\gamma := w^{-1}  \CL_{n+1} j_\gamma \in \CH_n.
 \end{equation}
 \end{definition}
 Note that $q_\gamma \in \CF[w]_n$ involve all powers of $w$, unlike $j_\lambda \in \CF$ which are polynomials only in the $V_k$.

 \subsection{$w$-Jack Expansion}
 
 We define the $w$-Jack expansion coefficients $\varphi_{\lambda,\mu}^s \in \BCe$ by
 \begin{equation}
 \psi_\lambda^s = \sum_\mu \varphi_{\lambda,\mu}^s w^{|\lambda|-|\mu|} j_\mu.
 \end{equation}
 \begin{lemma}\label{wjacklemma}
 We have $\varphi_{\lambda,\mu}^s = 0$ unless $\mu \subseteq \lambda$. Furthermore, for $s\in \addset_\lambda$ and $t \in \remset_\lambda$, we have
 \begin{equation}
 \varphi_{\lambda,\lambda-t}^s = \frac{\tilde \tau^{t+(1,1)}_{\lambda} }{[s-t-(1,1)]}.
  \end{equation}
Lastly, let $SYT(\lambda/\mu)$ be the collection of of standard young tableau of skew-shape $\lambda/\mu$. That is, $\Sigma \in SYT(\lambda/\mu)$ is a collection of partitions $\Sigma=\{\sigma_i\}$ with $\mu=\sigma_0 \subset \sigma_1 \subset \cdots \subset \sigma_N = \lambda$, where $N=|\lambda|-|\mu|$, with $\sigma_{i+1} = \sigma_{i} + b_{i}$ differing by a single box $b_i$. Then we have
 \begin{equation}
 \varphi_{\lambda,\mu}^s = \sum_{\Sigma \in SYT(\lambda/\mu)} \prod_{i=0}^{N-1} \varphi_{\sigma_{i+1},\sigma_{i}}^{b_{i+1}}
 \end{equation}
 where $b_{N} = s$.
 \end{lemma}
 \begin{proof}
 From \cite{Mickler:2022}, we know that $\psi_\lambda^s = j_\lambda + \frac{1}{[s]-\CL^+} w q_\lambda =  j_\lambda + w \frac{1}{[s]-\CL-\bareps} q_\lambda$. Using \ref{qdef} and \ref{tildepsiw}, we have
 \begin{equation}\label{recursivepsiformula}
\psi_\lambda^s = j_\lambda +  w \sum_{t \in \remset_{\lambda}}\frac{\tilde \tau_{\lambda}^{t+(1,1)}}{[s-t-(1,1)]} \psi_{\lambda-t}^{t}.
\end{equation}
Applying \ref{recursivepsiformula} recursively, we recover the result.
 \end{proof}
 Thus we can write the restricted summation expression:
  \begin{equation}
 \psi_\lambda^s = \sum_{\mu\subseteq \lambda} \varphi_{\lambda,\mu}^s w^{|\lambda|-|\mu|} j_\mu.
 \end{equation}

 \begin{lemma}
  \begin{equation}
 \psi_{1^r}^s = \sum_{k=0}^{r} \varphi_{1^r,1^k}^s w^{r-k} j_{1^k},
 \end{equation}
 where for $k<r$ we have
 \begin{equation}
 \varphi_{1^r,1^k}^s = [s] \prod_{b \in 1^r/1^{k+1}} [b].
 \end{equation}
 \end{lemma}
 \begin{proof}
 We have
\begin{equation}
\varphi_{1^r,1^{r-1}}^s = \frac{\tilde \tau^{(r,1)}_{\lambda} }{[s-(r,1)]}= \frac{ -[0,1][r,0] }{[s-(r,1)]} = [s],
\end{equation}
since $s$ is either $(0,1)$ or $(r,0)$. With this, the result follows from Lemma (\ref{wjacklemma}).
\end{proof}
\begin{example}
$\psi_{1^3}^s = j_{1^3}+ [s] w j_{1^2}  + [s][2,0] w^2 j_1  + [s][2,0][1,0]w^3$.

\end{example}

\subsection{Jack-Lax Littlewood-Richardson Coefficients}

Here, we can state the first example of the Jack-Lax Littlewood-Richardson coefficients (\ref{psialgebra}), by providing a refinement of the simplest Pieri rule (\ref{jackpieri}), 
\begin{lemma}\label{laxpieri}
\begin{eqnarray}
\psi^v_{1}  \cdot \psi_\lambda^s &=& \sum_{u \in \addset_{\lambda+s}} \frac{[-v][s-u-v+(1,1)]}{[s-u][s-u+(1,1)]}\, \tau_{\lambda+s}^{u}\psi^u_{\lambda+s} \\
 &&+ \sum_{t\in \addset_{\lambda}, t\neq s} \tau_{\lambda}^{t}  \psi^s_{\lambda+t}.
\end{eqnarray}
\end{lemma}
\begin{proof}
We note that $\psi^v_{1} = j_{1} + [v]w$. We combine the formula (\ref{frjpsifirstformula})
\begin{eqnarray}
j_{1} \cdot \psi_{\lambda}^{s} &=& \sum_{b \in \addset_{\lambda+s}}  \frac{\hbar\, \tau_{\lambda+s}^b}{[s-b][s+(1,1)-b]}\psi^b_{\lambda+s} \\
 &&+ \sum_{t\in \addset_{\lambda}, t\neq s} \tau_{\lambda}^t  \psi^{s}_{\lambda+t},
\end{eqnarray}
with (\ref{wactionpsi})
\begin{eqnarray}
[v]w \cdot \psi_\lambda^s &=& \sum_{b \in \addset_{\lambda+s}} \frac{[v]\tau_{\lambda+s}^b }{[b-s-(1,1)]}\psi_{\lambda+s}^b,
\end{eqnarray}
to see that the coefficient of $\tau_{\lambda+s}^b\psi^b_{\lambda+s}$ is
\begin{equation}
\frac{\hbar}{[s-b][s+(1,1)-b]} + \frac{[v] }{[b-s-(1,1)]}= \frac{[v][(1,1)-v]-[v][s-b]}{[s-b][s+(1,1)-b]},
\end{equation}
and the result follows.
\end{proof}
If we apply $\pi_0$ to this formula, we recover the usual Pieri rule (\ref{jackpieri}).

\subsection{$\QX\QY\QZ$ Decomposition}\label{XYZsection}

One of the ways we will gain insight into the structure of the algebra \ref{psialgebra} is through various decompositions of the space $\CH$. The first of these is given by the subspaces $\QZ_\lambda := \Span_{s \in \addset_{\lambda}} \{\psi_\lambda^s\} \subset \CH_{|\lambda|}$, that is,

\begin{equation}
\CH_n  = 
 \bigoplus_{\lambda \partition n} \QZ_\lambda.
\end{equation}

\begin{example}
\begin{figure}[htb]
\centering
\includegraphics[width=0.7 \textwidth]{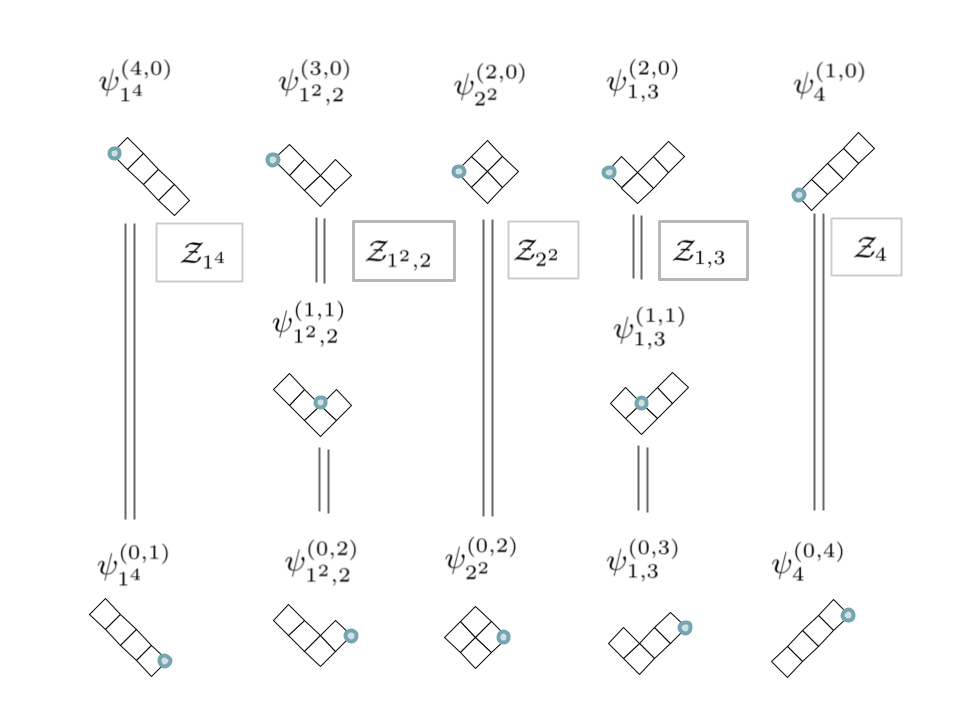}
\caption{The decomposition $\CH_n  = 
 \bigoplus_{\lambda \partition n} \QZ_\lambda$, where vertical double lines represent $\QZ$ subspaces. For example, $\QZ_{2^2} = \Span\{ \psi^{(2,0)}_{2^2}, \psi_{2^2}^{(0,2)} \} \subset \CH_4$.}
\label{figZdecomp}
\end{figure}
\end{example}

The second decomposition of $\CH$ is the eigen-decomposition under the Lax operator $\CL$. Denote the $[s]$ eigenspace of $\CL_n$ on $\CH_n$ as $\QY^{s}_n$. 
We then have
\begin{equation}
\CH_n = \bigoplus_{s: \,[s] \in \Spec \CL_n} \QY_n^{s}.
\end{equation}
\begin{figure}[htb]
\centering
\includegraphics[width=0.7 \textwidth]{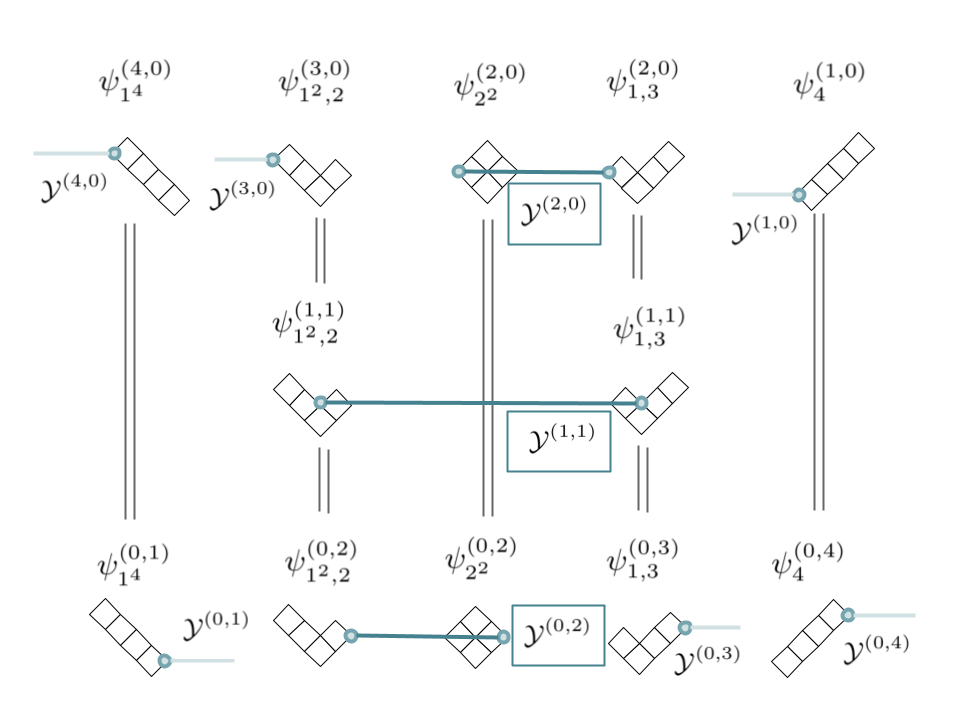}
\caption{The decomposition $\CH_n = \bigoplus_{s: \,[s] \in \Spec \CL_n} \QY_n^{s}$, where horizontal lines represent $\QY$ spaces. For example, $\QY_4^{(2,0)} = \Span\{ \psi^{(2,0)}_{1,3}, \psi_{2^2}^{(2,0)} \} \subset \CH_4$.}
\label{figYdecomp}
\end{figure}

Motivated by the results of the previous section, we define the third decomposition into subspaces

\begin{definition}
\begin{equation}
\QX_\gamma := \Span_{t \in \remset_{\gamma}} \{ \psi_{\gamma-t}^{t} \} \subset \CH_{|\gamma|-1}
\end{equation}
\end{definition}
These spaces also give a decomposition of $\CH$.
\begin{corollary}
\begin{equation}
\CH_n= \bigoplus_{\gamma \partition (n+1)} \QX_\gamma.
\end{equation}
\end{corollary}
\begin{figure}[htb]
\centering
\includegraphics[width=0.7 \textwidth]{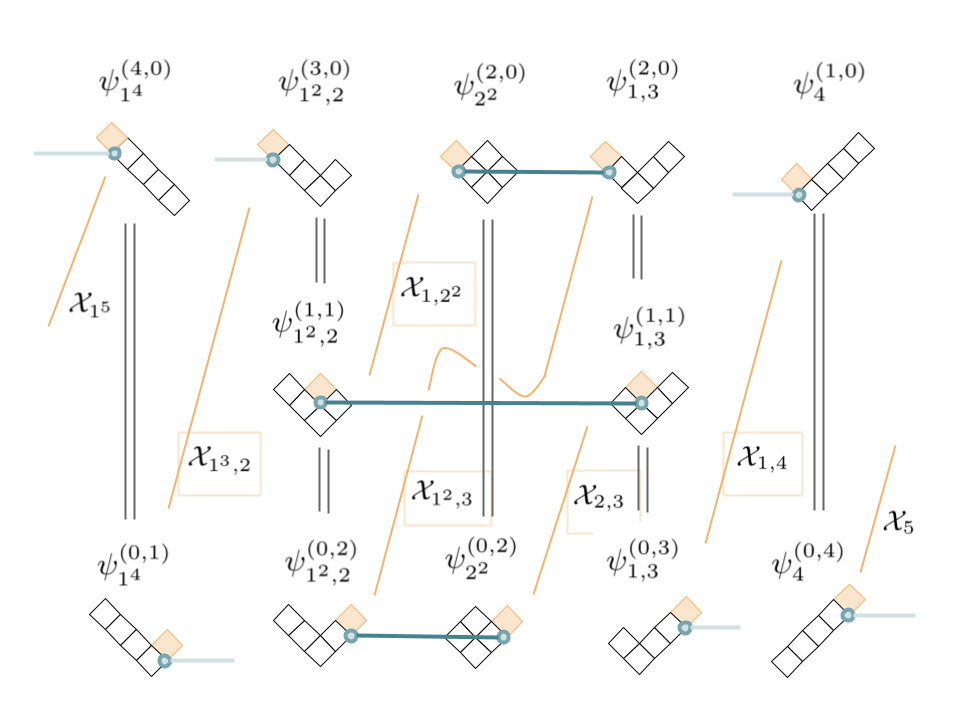}
\caption{The decomposition $\CH_n= \bigoplus_{\gamma \partition (n+1)} \QX_\gamma$, where diagonal colored lines represent $\QX$ spaces. For example, $\CX_{1^3,2} = \Span\{ \psi^{(0,1)}_{1^4}, \psi_{1^2,2}^{(3,0)} \} \subset \CH_4$.}
\label{figXdecomp}
\end{figure}
With this definition, the results of the previous section can be re-stated in the following simple way.
\begin{corollary} For $\gamma \partition (n+1)$, the $\QX$ spaces are also Lax orbits
\begin{equation}
\QX_\gamma = Z(q_\gamma,\CL_n) \subset \CH_n.
\end{equation}
 or equivalently,
 \begin{equation}\label{qexpans}
 q_\gamma = \sum_{t \in \remset_{\gamma}}\tilde \tau_\gamma^{t+(1,1)} \psi_{\gamma-t}^{t} \in \QX_\gamma.
 \end{equation}
Furthermore, multiplication by $w$, 
\begin{equation}
w: \QX_\lambda\to \QZ_\lambda^+ ,
\end{equation} 
is an isomorphism.
\end{corollary}

If we let $\Pi = w^{-1} \pi_+ : \CH_n \to \CH_{n-1}$, we have the inverse statement to \ref{wactionpsi}
\begin{lemma}
\begin{equation}\label{Piactionpsi}
\Pi \psi_{\lambda}^{s} = \sum_{t \in \remset_{\lambda}}\frac{\tilde \tau_\lambda^{t+(1,1)}}{[t+(1,1)-s]} \psi_{\lambda-t}^{t} \in \QX_\lambda.
\end{equation}
\end{lemma}
\begin{proof}
We note from \cite{Mickler:2022} (A.5.3) we have $\psi_\lambda^s = j_\lambda + \frac{1}{\CL^+-[s]} w q_\lambda$, from which we find $\pi_+ \psi_{\lambda}^{s} = w \frac{1}{\CL+\bareps-[s]} q_\lambda$.
\end{proof}


With these definitions, we can provide a lifting  of the obvious identity $| \addset_{\lambda} | =1+ |\remset_{\lambda}| $
to the level of vector spaces.
\begin{corollary}[Structural Theorem]
\begin{equation}
\QZ_\lambda =  \BC \, j_\lambda \oplus w\cdot\QX_\lambda.
\end{equation}
\end{corollary}

We have shown that we have three decompositions of $\CH$,

\begin{equation}\label{threedecompositions}
\CH_n = \bigoplus_{\gamma \partition n+1} \QX_\gamma =  \bigoplus_{[s] \in \Spec \CL_n} \QY^s = 
 \bigoplus_{\lambda \partition n} \QZ_\lambda .
\end{equation}
We note that intersection of any two of $\QZ_\lambda, \QX_{\lambda+s}$ or $\QY^s$ is the line $\BC \cdot \psi_\lambda^s  \subset \CH_{n}$, and that the intersection of any two of any type of these subspaces is at most one dimensional.

\newcommand\qp{{\pi_{\diamond}}}
\subsubsection{$\qp$ operator}
For a brief interlude, we introduce another projection operator.

\begin{lemma} Let $A := \pi_0 \CL_{n+1}w : \CH_{n} \to \pi^0H_{n+1}$, and $B := w^{-1} \CL_{n+1} : \CH_{n+1}^0 \to \CH_{n}$, then we have
\begin{equation}
AB = n \hbar \,{\one}_{\pi^0 \CH_{n+1}}.
\end{equation}

\end{lemma}

Let $\qp  \in \End \CH_n$ be defined by 
\begin{equation}\label{pidiamonddef}
\qp := (n \hbar)^{-1} BA.
\end{equation} Then $\qp^2 = \qp$, that is, $\qp$ is a projection operator of rank $\dim \pi^0 \CH_{n+1} = p(n+1)$.

We have shown that $B j_\lambda = q_\lambda$ and $A q_\lambda = n \hbar j_\lambda$, indeed $A \psi_{\gamma-s}^{s} = j_{\gamma}$. Hence $\qp$ preserves the $\QX$ decomposition of $\CH_n$, and
\begin{equation}
\qp|_{\QX_\gamma} =  \pi_{q_\gamma},
\end{equation}
that is, $\qp$ restricted to $\QX_\lambda$ is rank one and is equal to the projection operator onto $q_\lambda$. Compare this to the companion statement
\begin{equation}
\pi_0 |_{\QZ_\lambda} =  \pi_{j_\lambda}.
\end{equation}







\subsection{Traces}

\subsubsection{Top Powers}
The results so far concern looking at $\pi_0$ of a resolvent, that is, in the lowest powers of $w$. In this section, we rather look at the top component in powers of $w$. For any $\zeta \in \CH_{n }$, define
\begin{equation}
\pi_* \zeta := [w^n] \zeta = \langle w^n, \zeta \rangle \in \BCe .
\end{equation}
\begin{lemma}\label{resolventjacktop}
The highest and lowest $w$-components of the Lax resolvent acting on a Jack function are given by
\begin{equation}
\frac{1}{u-\CL} j_\lambda = u^{-1} T_\lambda(u) \cdot j_\lambda w^0 + \cdots + \varpi_\lambda\left(T_\lambda(u) - 1\right)  \cdot w^{|\lambda|},
\end{equation}
where $\varpi_\lambda \coloneqq [V_{|\lambda|}]j_\lambda$ as before (\ref{jacktopdef}).
\end{lemma}
\begin{proof}
Recall from \ref{NSspectralthm}, we have $\pi_0 \frac{1}{u-\CL} j_\lambda = u^{-1} T_\lambda(u) \cdot j_\lambda$. We use $ j_\lambda = \sum_{s \in \addset_{\lambda}} \tau_\lambda^s \psi_\lambda^s$, and compute the top component
\begin{eqnarray}
\pi_* \frac{1}{u-\CL} j_\lambda &=& \sum_{s \in \addset_{\lambda}} \tau_\lambda^s  \frac{1}{u-[s]}\pi_* \psi_\lambda^s  \\
&=& \varpi_\lambda \sum_{s \in \addset_{\lambda}}  \frac{[s] \tau_\lambda^s}{u-[s]}\\
&=&  \varpi_\lambda\left( \sum_{s \in \addset_{\lambda}}  \frac{u\, \tau_\lambda^s}{u-[s]} - \sum_{s \in \addset_{\lambda}} \tau_\lambda^s \right)\\
&=&  \varpi_\lambda \left( T_\lambda(u) -1 \right).
\end{eqnarray}
On the second line we used the principal specialization result (\ref{psiprincipalspecialization}) $\pi_* \psi_{\lambda}^s = [s] \varpi_\lambda$.
\end{proof}

To reorient the rest of our work towards working with these top powers of $w$, we introduce new conventions.

\begin{definition} Define the rescaled `hatted' elements
\begin{equation}
\hpsi_\lambda^s := \psi_\lambda^s/ \pi_*( \psi_\lambda^s) =  \psi_\lambda^s/[s]\varpi_\lambda, \quad \hat j_\lambda := j_\lambda/ \varpi_\lambda, \quad\hat q_\lambda := q_\lambda/ \varpi_\lambda,
\end{equation}
\begin{equation} \htau_\lambda^s := [s] \tau_{\lambda}^s.\end{equation}
\end{definition} 
With these redefinitions we have
\begin{equation}\label{newdefinitions}
\hat j_\lambda = \sum_{s \in \addset_{\lambda}} \htau_\lambda^s \hpsi_\lambda^s, \qquad \sum_{s\in\addset_{\lambda}} \htau_\lambda^s = 0, \qquad  \hat q_\gamma = \sum_{t \in \remset_{\gamma}} \tilde \tau_{\gamma}^{t+(1,1)} \hpsi^{t}_{\gamma-t} .
\end{equation}
\begin{example}
\begin{equation}
\hat j_{\{2^2\}}  = \frac{1}{\vareps_1\vareps_2(\vareps_1+\vareps_2)} V_14+ \frac{2}{\vareps_1\vareps_2} V_{1}^2V_{2}+\frac{\vareps_1^2-\vareps_1\vareps_2+\vareps_2^2}{\vareps_1\vareps_2(\vareps_1+\vareps_2)}V_2^2+ \frac{4}{\vareps_1+\vareps_2}V_1V_3+  V_4.
\end{equation}
\end{example}
It is important to note that this normalization does {\bf{not}} work in the Schur case, where $(\vareps_2,\vareps_1) = (1,-1)$, since for any partition $\mu$ containing the box $(1,1)$, the factor $[1,1] = \vareps_1+\vareps_2$ in $\varpi_\mu = \prod_{s \in \mu^\times}[s]$ vanishes. For the remainder of this article, we will assume that $\vareps_1+\vareps_2 \neq 0$.  We leave the investigation of Schur polynomials and ordinary Littlewood-Richardson coefficients via the $\vareps_1+\vareps_2 \to 0$ degeneration of our methods and results to future research.

\subsubsection{Trace functionals}\label{tracefunctionalssection}

We continue the shifting of emphasis in our investigation to the coefficients of top powers of $w$ by introducing the first of three \emph{trace} functionals.
\begin{definition}
The $\ytr$-trace functional $\ytr_u : \CH_n \to \BCe(u)$, is given by
\begin{equation}\label{yutrdef}
\ytr_u(\zeta) := \pi_* \frac{1}{u-\CL} \zeta. \end{equation} 
\end{definition}
With the above redefinitions, we can restate Lemma \ref{resolventjacktop} concisely
\begin{corollary}\label{traceexpressions}

\begin{equation}\label{yujack}
\ytr_u(\hat j_\lambda) = T_\lambda(u) - 1.
\end{equation}
\end{corollary}

We note that this trace is closely associated with the $\QY$ subspaces (the $\CL$-eigenspaces). We extend this definition the following three families of linear functionals $\CH_n \to \BCe$ associated to each of the three decompositions (\ref{threedecompositions}) of $\CH$ 
\begin{definition}
\begin{equation} \xtr_\gamma(\zeta)= \pi_* P_{\QX_\gamma} \zeta,\qquad  \ytr^s(\zeta) =   \pi_* P_{\QY^{s}} \zeta,\qquad \ztr_\lambda(\zeta) = \pi_* P_{\QZ_\lambda} \zeta .
\end{equation}
\end{definition}
Note that we have $ \ytr^s(\zeta) =  \Res_{u=[s]}  \ytr_u(\zeta)$.

\begin{lemma}\label{traceexpressionslemma}
The traces are given by the following expressions,
\begin{equation}\label{traceexpressions1}
\xtr_\gamma(\zeta)  = \langle\frac{ \hat q_\gamma}{| \hat j_\gamma |^2} , \zeta\rangle, \qquad \ytr_u(\zeta)  = \langle \frac{   1}{u-\CL} w^{n} , \zeta\rangle, \qquad \ztr_\lambda(\zeta)  = \langle\frac{  w \hat q_\lambda}{| \hat j_\lambda |^2}, \zeta\rangle .
\end{equation}
Equivalently,
\begin{equation}\label{traceexpressions2}
\frac{\hat q_{\gamma}}{| \hat j_\gamma|^2} =  \sum_{t \in \remset_{\gamma}} \frac{\hpsi_{\gamma-t}^{t}}{|\hpsi_{\gamma-t}^{t}|^2}, \qquad w^n = \sum_\lambda \frac{w\hat q_{\lambda}}{| \hat j_\lambda|^2}, \qquad  \frac{w\hat q_{\lambda}}{| \hat j_\lambda|^2} =  \sum_{s \in \addset_{\lambda}} \frac{\hpsi_{\lambda}^s}{|\hpsi_{\lambda}^s|^2}.
\end{equation}
\end{lemma}
\begin{proof} The middle result of \ref{traceexpressions1} is just a restating of \ref{yutrdef}.
By definition, $\ztr_\lambda$ is the linear functional that takes the value $1$ on every $\hpsi_\lambda^s \in \CZ_\lambda$, and vanishes for every other basis element in $\CH$. Similarly for $\xtr_\gamma$ and $\hpsi_{\gamma-t}^{t} \in\CX_\gamma$. Hence the equivalence between the first and last entries of each of \ref{traceexpressions1} and \ref{traceexpressions2}. To check \ref{traceexpressions2}, we start with $\CL V_n = n \hbar w^n$, so
\begin{equation}
w^n =  (n\hbar)^{-1}\CL V_n = (n\hbar)^{-1}\sum_\lambda \langle V_n, \hat j_\lambda\rangle \frac{ \CL \hat j_\lambda}{|\hat j_\lambda|^2} = \sum_\lambda \frac{w\hat q_\lambda}{|\hat j_\lambda|^2}.
\end{equation}
Next, we expand
\begin{equation}
\frac{ \CL \hat j_\lambda}{|\hat j_\lambda|^2} = \sum_{s \in \addset_{\lambda}} \frac{ [s] \htau_\lambda^s\hpsi^s_\lambda}{|\hat j_\lambda|^2} = \sum_{s \in \addset_{\lambda}} \frac{ \hpsi_\lambda^s}{|\hpsi_\lambda^s|^2}.
\end{equation}
Here we've used $|\hpsi_\gamma^s|^2 = |\hat j_\gamma|^2/[s]\htau_\gamma^s$, from \ref{psinormformula}.

For the first of \ref{traceexpressions2}, we compute
\begin{equation}
\langle \hpsi_{\gamma-t}^t , \hat q_\gamma\rangle = \langle w \hpsi_{\gamma-t}^t ,w \hat q_\gamma\rangle
\end{equation}
Next we use \ref{wactionpsi}, which still holds with the new conventions,
\begin{equation}
w\cdot \hpsi_{\gamma-t}^{t} = \sum_{s \in \addset_\gamma} \frac{\htau_{\gamma}^s }{[s-t-(1,1)]}\hpsi_{\gamma}^s.
\end{equation}
To get
\begin{equation}
 \sum_{s \in \addset_\gamma} \frac{\htau_{\gamma}^s }{[s-t-(1,1)]}\langle\hpsi_{\gamma}^s, \CL \hat j_\gamma\rangle  =  \sum_{s \in \addset_\gamma} \frac{\htau_{\gamma}^s }{[s-t-(1,1)]}[s]\htau_{\gamma}^s|\hpsi_{\gamma}^s|^2 = \sum_{s \in \addset_\gamma} \frac{\htau_{\gamma}^s }{[s-t-(1,1)]}|\hat j_{\gamma}|^2
\end{equation}
Then we use
\begin{equation}
\sum_{s \in \addset_\gamma} \frac{\htau_{\gamma}^s }{u-[s]} = T_{\gamma}(u)-1
\end{equation}
To find
\begin{equation}
\langle \hpsi_{\gamma-t}^t , \hat q_\gamma\rangle = |\hat j_{\gamma}|^2\left(1-T_{\gamma}([t+(1,1)])\right) =|\hat j_{\gamma}|^2,
\end{equation}
since $t\in \remset_\gamma \implies T_{\gamma}([t+(1,1)]) = 0$. Thus we have 
\begin{equation}
 \hat q_\gamma = \hpsi_{\gamma-t}^t \frac{ |\hat j_{\gamma}|^2}{| \hpsi_{\gamma-t}^t |^2} + \cdots,
\end{equation}
and the result follows.
\end{proof}

Comparing the two formulas for $\hat q_\gamma$, \ref{newdefinitions} and \ref{traceexpressions2}, we find the following striking identity
\begin{corollary}\label{jackratio}
For $s \in \addset_{\lambda}$, we have
\begin{equation} 
\frac{| j_{\lambda+s}|^2}{ |j_{\lambda}|^2} =\frac{\tilde \tau_{\lambda+s}^{s+(1,1)} }{ \tau^{s}_{\lambda} }.
\end{equation} 
\end{corollary}
\begin{proof}
By looking at the $\hpsi_{\gamma-t}^{t}$ components of the two expressions for $\hat q_\gamma$, we get $\tilde \tau_{\gamma}^{t+(1,1)} =  | \hat j_\gamma|^2|\hpsi_{\gamma-t}^{t}|^{-2}$. Then we use \ref{psinormformula} again, i.e. $|\hpsi_{\gamma-t}^t|^2 = |\hat j_{\gamma-t}|^2/[t]\htau_{\gamma-t}^t$, to find $\tilde \tau_{\gamma}^{t+(1,1)}/[t]\htau_{\gamma-t}^t =  | \hat j_\gamma|^2/ |\hat j_{\gamma-t}|^2$.
\end{proof}


We call these $\xtr,\ytr$ and $\ztr$ {\emph{trace}} maps because of the following important properties,
\begin{corollary}
For any $\zeta \in \CH$, we have
\begin{equation}\label{traceprops}
\xtr_\gamma( \zeta ) = \xtr_\gamma( \qp \zeta), \qquad \ztr_\lambda( \zeta ) = \ztr_\lambda( \pi_+ \zeta), \qquad \ztr_\sigma( w \zeta ) = \xtr_\sigma( \zeta).
\end{equation}
\begin{equation}\label{yuLprop}
\ytr_u\left(\CL\zeta \right)=u\cdot  \ytr_u(\zeta ) -\pi_* \zeta.
\end{equation}
\begin{equation} \xtr_\sigma( \Pi \zeta) = \ztr_\sigma(\zeta). \end{equation}
\end{corollary}




\subsection{The Full Trace}

Next, we investigate the combined trace map, given by taking the direct sum of the three traces defined so far.
\begin{definition}
\begin{equation}
\Tr_n =  \oplus_\gamma \xtr_\gamma \oplus_s  \ytr^s \oplus_\lambda \ztr_\lambda : \CH_n \to \CW_n,
\end{equation}
where the range $\CW_n$ of the trace map is
\begin{equation}\label{imagetracedef}
 \CW_n := \BC_{}^{p(n+1)} \oplus \BC_{}^{q(n)-1} \oplus \BC_{}^{p(n)}.
 \end{equation}
\end{definition}  

Here $p(n)$ are the partition numbers (c.f. \ref{partitionapps}), and $q(n) = |{{ \latticehyp}}(n)|$ is the size of the set of points ${{ \latticehyp}}(n) \subset \BZ^2$  that are inner-corners for partitions of size $n$. Equivalently, it is the number of integral points beneath the hyperbola
\begin{equation}
q(n)= | \{ (m,n) : (m+1)(n+1) \leq n+1 \} |  \sim (n +1)\log(n+1).
\end{equation}

These numbers are given by the generating function\footnote{c.f. Sloane's "On-Line Encyclopedia of Integer Sequences" \url{https://oeis.org/A006218}},
\begin{equation}
Q(x) \coloneqq \sum q(n) x^n = 1+\frac{1}{1-x}\left(-1+ \frac{1}{x} \sum_{k \geq 1} \frac{x^k}{1-x^k}\right).
\end{equation}

\begin{lemma}
The graded dimension of the extended Fock module $\CH = \CF[w]$ is given by 
\begin{equation}\label{dimensionhnfun}
\sum (\dim \CH_n) x^n = \frac{1}{1-x} P(x),
\end{equation}
where $P(x)$ is the generating function for the partition numbers (c.f. \ref{partitionapps}).
\end{lemma}
\begin{proof}
This follows from $\dim \CH_n =  \dim \bigoplus_{k\leq n} w^{n-k} \CF_k = \sum_{k\leq n} p(k)$.
\end{proof}

When working with traces, we often use the condensed notation:
\begin{equation}
\Tr(\zeta) = (\{x_\gamma\},\{y^s\},\{z_\lambda\}).
\end{equation}
We will omit indices if they can be inferred. For example,
\begin{equation}\label{tracepsi}
\Tr_n (\hpsi_\lambda^s) = ( \{\delta_{\lambda+s}\} , \{\delta^{s} \},\{ \delta_{\lambda}\} ),
\end{equation}
where $\delta_\lambda$ indicates the $p(|\lambda|)$-vector of values that has $1$ only in the $\lambda$ position, etc.

Our goal for the remainder of this section will be to determine the kernel and cokernel of the full trace,
\begin{equation}\label{totaltrace}
\ker \Tr_n \rightarrow \CH_{n} \stackrel{\Tr_n}{\longrightarrow} \CW_n \rightarrow \mathrm{coker}\, \Tr_n.
\end{equation}

\subsubsection{Cokernel}

\begin{theorem}
The cokernel of $\Tr_n$ contains the following relations
\begin{equation}
\mathrm{coker}\,\,  \Tr_{n} \supset Span \{ R^{s}_{n} :  s \in \latticehyp(n)\} ,
\end{equation}
where
\begin{equation}\label{cokerrel}  R_{n}^{s} := \ytr^s - \sum_{\gamma \partition n+1: s \in \gamma} \xtr_\gamma + \sum_{\lambda \partition n: s \in \lambda} \ztr_\lambda, \end{equation}
are the residues of the function
\begin{equation}\label{cokerru} R_n(u) := \ytr_u- \sum_{\gamma \partition n+1} \xtr_\gamma \left( \sum_{t\in \gamma} \frac{1}{u-[t]}\right)  + \sum_{\lambda \partition n} \ztr_\lambda \left(\sum_{v\in \lambda} \frac{1}{u-[v]}\right), \end{equation}
i.e. $R_{n}^{s} := \Res_{u=[s]}R_n(u) $.
\end{theorem}

\begin{proof}
We evaluate the generating relation $R_n(u)$ \ref{cokerru} on the traces of the basis elements $\zeta = \hpsi_\lambda^s$. We have $\Tr_n (\hpsi_\lambda^s) = ( \{\delta_{\lambda+s}\} , \{\delta_{s} \},\{ \delta_\lambda\} )$ (\ref{tracepsi}), and so
\begin{equation}
R_n(u) \Tr_n (\hpsi_\lambda^s) =\frac{1}{u-[s]} -\left( \sum_{t \in \lambda + s}\frac{1}{u-[t]} \right) +\left( \sum_{v \in \lambda}\frac{1}{u-[v]} \right) = 0.
\end{equation}
Thus the relations hold on the traces of all basis elements in $\CH$, i.e they hold on $\CW$.
Clearly the relations \ref{cokerrel} are all linearly independent, because $R^s_n$ is the only such relation that contains $\ytr^s$.
\end{proof}

\begin{example}
In $n = 4$ we have the 10 relations given by
\begin{eqnarray}
R_4^{(4,0)} &=& \ytr^{(4,0)} - \xtr_{1^5},\\
R_4^{(3,0)} &=& \ytr^{(3,0)}+ \ztr_{1^4} - \xtr_{1^5}- \xtr_{1^3,2},\\
R_4^{(1,1)} &=& \ytr^{(1,1)} +\ztr_{2^2} - \xtr_{1^32} - \xtr_{1,4},\\
R_4^{(2,0)} &=& \ytr^{(2,0)} +\ztr_{1^4}+ \ztr_{1^2,2}- \xtr_{1,4}-\xtr_{1^5}-\xtr_{1,2^2}-\xtr_{1^23},\\
R_4^{(1,0)} &=& \ytr^{(1,0)} +\ztr_{1^4}+\ztr_{1^22} + \ztr_{2^2} + \ztr_{1,3} - \xtr_{1^5}-\xtr_{1^32} -\xtr_{12^2}-\xtr_{1^23}-\xtr_{2,3} -\xtr_{1,4},\\
R_4^{(0,0)} &=& \ytr^{(0,0)} +\sum_{\lambda} \ztr_{\lambda} -\sum_\gamma \xtr_\gamma,
\end{eqnarray}
and their transposes (i.e. transpose all partitions and boxes appearing in the relation).
\end{example}

Note, for $n>0$, we know that we also have the cokernel relation $\ytr^{(0,0)} = 0$, since the only partition with minima at $(0,0)$ is the empty partition. So, for $n>0$ we redefine $R_n^{(0,0)} = \sum_{\lambda} \ztr_{\lambda} -\sum_\gamma \xtr_\gamma$. For $n=0$, we have only the single relation $R_0^{(0,0)} = \ytr^{(0,0)} - \xtr_{\{1\}}$, since $\ztr_\emptyset$ does not appear in $R_0(u)$.

Later, in corollary \ref{cokerdimresult}, we show that these relations exhaust the cokernel. If we write the generating relation \ref{cokerru} using the formulas for the traces given by Lemma \ref{traceexpressionslemma}, we recover a key identity in $\CH$.
\begin{corollary}
\begin{equation}\label{resolvewident}
\frac{1}{u-\CL} w^n = \sum_{\gamma \partition (n+1)} \left( \sum_{t\in \gamma} \frac{1}{u-[t]} \right) \frac{\hat q_\gamma}{|\hat j_\gamma|^2} - \sum_{\lambda \partition n} \left( \sum_{s\in \lambda} \frac{1}{u-[s]} \right) \frac{w\hat q_\lambda}{|\hat j_\lambda|^2}.
\end{equation}
\end{corollary}
Later, in section \ref{shcsection}, we explore the algebraic significance of this relation.

\subsubsection{The bicomplex}
To continue our study of the traces and their kernels, we need to introduce a homological construction.

Let $K_{n}^{\ell}$, for $n, \ell\geq 0$ be the $\BC$-linear span on the space of symbols $\Psi_\gamma^{a_1, \cdots, a_\ell}$ with  $\gamma \partition n$ and $\ell$ distinct points $a_i \in \addset_{\gamma}$ in the addset of $\gamma$ \ref{addsetdefn} and which are antisymmetric in the $a_i$. We have $K_{n}^{0} = \Span_{\gamma \partition n}\{ \Psi_\gamma \} \cong \BC^{p(n)}$.
The spaces $K_n^\ell$ sit in a double complex $(K, d, \partial)$, with first differential $d :  K_{n}^{\ell} \to K_{n}^{\ell-1} $, given by
\begin{equation}
d : \Psi_\sigma^{ a_1,\ldots, a_{ \ell}} \mapsto \sum_{i=1}^{\ell} (-1)^i \Psi_\sigma^{a_1,\ldots, \dot a_{ i}\ldots, a_{ \ell}}.
\end{equation}
The second is $\partial : K_{n-1}^{\ell} \to K_{n}^{\ell-1}$, given by
\begin{equation}
\partial : \Psi_\sigma^{a_1,\ldots, a_{ \ell}} \to \sum_{i=1}^{\ell} (-1)^i \Psi_{\sigma+ {a_i}}^{a_1,\ldots, \dot a_{ i}\ldots, a_{ \ell}}.
\end{equation}
One can easily check that the following digram commutes
\begin{equation}\label{boundarymaps}
\xymatrix{
K_{n+1}^{\ell-1} \ar[r]^{d_{\ell-1}} & K_{n+1}^{\ell-2} \\
K_{n}^{\ell} \ar[u]^{\partial_{\ell}} \ar[r]^{d_{\ell}} & K_{n}^{\ell-1} \ar[u]_{\partial_{\ell-1}}
}
\end{equation}
This construction is motivated by the bijection  $\iota : K_{n}^{1} \to \CH_n$ given by
\newcommand{\psiiso}{\iota}
\begin{equation}
\psiiso : \Psi_\eta^ a \mapsto \hpsi_{\eta}^{a}.
\end{equation} 
The following result gives a cohomological interpretation of the $\ztr$ and $\xtr$ traces.
\begin{corollary}
Under the bijection $\psiiso$, we have $\ztr = d_1 \circ \psiiso^{-1}$, where we identify $K_n^0 = \oplus_\eta \BC\{ \Psi_\eta \} \cong\BC^{p(n)}$ with the image of the trace $\ztr = \oplus_\eta \ztr_\eta$. Similarly, we have $\xtr = \partial_1 \circ \psiiso^{-1}$.
\end{corollary}
After setting $r=1$ and applying $\psiiso$, the diagram \ref{boundarymaps} becomes
\begin{equation}
\xymatrix{
\BC^{p(n+1)}  & \\
 \CH_n \ar[u]^{\xtr} \ar[r]^{\ztr}&  \BC^{p(n)}
}
\end{equation}

\subsubsection{Kernel}
We now determine the kernel of the full trace map $\Tr_n$.
\begin{theorem}
The kernel of $\Tr_n$ is given by 
\begin{equation}
\ker \Tr_n =\psiiso (\Im  d\partial : K_{n-1}^{3} \to K_{n}^{1})= \Span \{ \Gamma_{\eta}^{s_1,s_2,s_3} : \eta \vdash (n-1),  s_i \in \addset_\eta \},
\end{equation}
where $\Gamma$ are the `hexagon' elements
\begin{equation}\label{gammadefinition}
\Gamma_{\eta}^{a,b,c} := \psiiso d\partial\Psi_\eta^{a,b,c}= \hpsi_{\eta+a}^{c}- \hpsi_{\eta+a}^{b}+  \hpsi_{\eta+b}^{a}-   \hpsi_{\eta+b}^{c}+   \hpsi_{\eta+c}^{b}-  \hpsi_{\eta+c}^{a} \in \CH_n,
\end{equation}
for $\eta \vdash (n-1),  a,b,c \in \addset_\eta$.
\end{theorem}
\begin{proof} 

From the bijection, we know that $\ztr_n \sim d_1 : K_n^1 \to K_n^0$, and $ \xtr_n \sim \partial_1 : K_{n}^1 \to K_{n+1}^0$. Note that we know from \ref{cokerrel} that $\ker \xtr \cap \ker \ztr \subset \ker \ytr$. Thus, $\ker \Tr =  \ker \xtr \cap \ker \ytr \cap \ker \ztr =  \ker \xtr \cap \ker \ztr = \psiiso(\ker \partial_1 \cap \ker d_1 )$. Since the bicomplex is acyclic, we have $\ker \partial_1 \cap \ker d_1 = \Im d_2\partial_3$.
 \end{proof}
 
\begin{example}
$\ker \Tr_4 = \BC\cdot\Gamma_{1,2}^{(2,0),(1,1),(0,2)} $, and $\ker \Tr_5 = \Span\{\Gamma_{1,3}^{(2,0),(1,1),(0,3)},\Gamma_{1^2,2}^{(3,0),(1,1),(0,2)} \}.$
\end{example}
 
\begin{proposition}\label{kernelcalc}
The dimension the kernel of the full trace is given by the generating function
\begin{equation} \sum_{n\geq 0} ( \dim \ker \Tr_n ) x^n  = \frac{1+(x^2+x-1)P(x)}{x(1-x)}  = x^4 + 2x^5 + 5x^6 + \ldots\,. \end{equation}
\end{proposition}
\begin{proof}
The bicomplex provides a resolution of the kernel
\begin{equation}
\xymatrix{
\cdots \ar[r]^d &  K_{n-1}^{5} \ar[r]^d & K_{n-1}^{4} \ar[r]^d  & K_{n-1}^{3}  \ar[r]^{\psiiso \circ d\partial} & \ker \Tr_{n} \\
\cdots \ar[r]^d\ar[ur]^\partial &K_{n-2}^{5}\ar[ur]^\partial \ar[r]^d & K_{n-2}^{4} \ar[ur]^\partial &                       &   \\
\cdots \ar[ur]^\partial\ar[r]^d & K_{n-3}^{5} \ar[ur]^\partial&  &                       &  \\
\iddots  \ar[ur]^\partial & &  &                       &  \\
}
\end{equation}
We use this to compute its dimension. We do this by looking at the horizontal sub-complexes:
\begin{equation}
K^{\ell+\bullet}_k := \cdots \rightarrow K^{\ell+2}_k \stackrel{d}{\rightarrow} K^{\ell+1}_k \stackrel{d}{\rightarrow} K^{\ell}_k. 
\end{equation}
In terms of these sub-complexes, we have
\begin{eqnarray*}
\sum_{n\geq0} \dim(\ker \Tr_n) x^n &= &\sum_{n\geq0} \left( \sum_{i=0}^{\infty} (-1)^i \chi(K^{3+i+\bullet}_{n-(i+1)}) \right) x^n \\
&=& \sum_{i=0}^{\infty}(-1)^i \left(\sum_{n\geq 0}  \chi(K^{3+i+\bullet}_{n} )x^{n} \right) x^{i+1}. \end{eqnarray*}

Consider just the sub-complex contribution of a partition $\eta \partition k$ with $r$ minima. At each stage of the complex we choose the appropriate number of those corners $S \subset \addset(\eta)$ to include in the symbol $\Psi_\eta^S$, and thus the contribution from $\eta$ is
\begin{equation} K^{\ell+\bullet}_{k,r}:= \cdots \rightarrow \BC^{\binom{r}{\ell+2}} \stackrel{d}{\rightarrow}  \BC^{\binom{r}{\ell+1}}  \stackrel{d}{\rightarrow} \BC^{\binom{r}{\ell}}.  \end{equation}
The dimension of this sub-complex is
\begin{equation} \chi(K^{\ell+\bullet}_{k,r})=\sum_{s=\ell}^{r} (-1)^{s-\ell}  \binom{r}{s}= \binom{r-1}{\ell-1} =  \left[ \left( \frac{1}{(\ell-1)!} (\partial_t)^{\ell-1} t^{-1}\right) t^r\right]_{t=1}. \end{equation}
We then have
\begin{equation}
K^{\ell+\bullet}_k \cong \bigoplus_r \left(K^{\ell+\bullet}_{k,r} \right)^{ p(k,r)}.
\end{equation}
Thus, the dimension of the horizontal complex is
\begin{equation}  \chi(K^{\ell+\bullet}_k) = \sum_{r>0} p(k,r)\chi(K^{\ell+\bullet}_{k,r})= \sum_{r>0} p(k,r) \binom{r-1}{\ell-1}= \left[\frac{1}{(\ell-1)!} (\partial_t)^{\ell-1} t^{-1}\sum_{r>0} p(k,r) t^r\right]_{t=1}  \end{equation}
where $p(k,r)$ is the number of partitions of $k$ with $r$ minima.
Thus, the generating function for the dimension of the kernel is
\begin{eqnarray*}
\sum_{n\geq0} \dim(\ker \Tr_n) x^n &= 
& \sum_{i=0}^{\infty}(-1)^i \left(\sum_{n\geq 0}  \chi(K^{3+i+\bullet}_{n} )x^{n} \right) x^{i+1} \\
&=&\sum_{i=0}^{\infty} (-1)^ix^{i+1} \left[\frac{1}{(3+i-1)!} (\partial_t)^{3+i-1}( t^{-1} P(x,t))\right]_{t=1} \\
&=&x^{-1} \left[\sum_{i=2}^{\infty}\frac{(-x)^{i}}{i!} (\partial_t)^{i} (t^{-1} P(x,t))\right]_{t=1} \\
&=&x^{-1} \left( \left[\sum_{i=0}^{\infty}\frac{(-x)^{i}}{i!} (\partial_t)^{i} (t^{-1} P(x,t))\right]_{t=1}-\left[-x \partial_t (t^{-1} P(x,t))\right]_{t=1} -  P(x,1) \right).
\end{eqnarray*}
Using Taylor's formula, and the properties \ref{Pxtprops} of $P(x,t)$, we find
\begin{eqnarray*}
&=&x^{-1}\left( (1-x)^{-1} P(x,1-x)+ (1-x)P(x,1) +P_{t}(x,1)\right) \\
&=&x^{-1}(1-x)^{-1} + x^{-1}(1-x)P(x) + x^{-1} (1-x)^{-1}P(x) \\
&=&\frac{1+\left(   x^2+x-1\right)P(x)}{x(1-x)}.
\end{eqnarray*}

\end{proof}

\subsubsection{Back to the Cokernel}

Using the calculation of the dimension of the kernel (proposition \ref{kernelcalc}), we can conclude that:

\begin{lemma}\label{cokerdimresult}
The set of relations \ref{cokerrel} exhaust the cokernel of the total trace, i.e.
\begin{equation}
\mathrm{coker}\,\,  \Tr_{n} = Span \{ R^{s}_{n} :  s \in \latticehyp(n)\}.
\end{equation}
\end{lemma}
\begin{proof}
The vanishing Euler characteristic of the exact sequence \ref{totaltrace} is
\begin{equation}
0= \chi \left(\ker \Tr_n \rightarrow \CH_{n} \stackrel{\Tr_n}{\longrightarrow} \CW_n \rightarrow \mathrm{coker}\, \Tr_n \right).
\end{equation}
The generating function of this yields
\begin{eqnarray}
\sum_{n\geq 0} ( \dim \mathrm{coker} \Tr_n ) x^n &=& \sum_{n\geq 0} ( \dim \mathrm{ker} \Tr_n - \dim \CH_n + \dim \CW_n ) x^n \\
&=& \left(\tfrac{1+\left(   x^2+x-1\right)P(x)}{x(1-x)}\right) - \left(\tfrac{P(x)}{1-x} \right)\\
&& +\left( \tfrac{P(x)-1}{x}+\left(Q(x)-\tfrac{1}{1-x}\right) + P(x)\right).
\end{eqnarray}
With this we confirm that the dimension of the cokernel is given by
\begin{equation} 
\sum_{n\geq 0} ( \dim \mathrm{coker} \Tr_n ) x^n  = Q(x).
\end{equation}
\end{proof}

\section{Distinguished Elements}\label{structionsection2}

The goal of this section will be to produce interesting algebra elements $\zeta \in \CH_n$ whose traces we can determine explicitly, and we'll then relate these traces through the fundamental cokernel relation (\ref{cokerru}),
\begin{equation}
 R_n(u) \Tr_n (\zeta) = 0.
 \end{equation}

To work towards the construction of these interesting elements, we provide more results on the structure of the algebra $\CH$. First, we'll construct various maps that are $\CF= \BC[V_1, V_2, \ldots]$ linear, which will allow us to focus on $\CH = \CF[w]$ as a free $\CF$ module.

\subsection{Homological algebra}

Let $\CA$ be a graded associative algebra. We recall the Hochschild cochain complex
\begin{equation}
\CC^k = \Hom( \otimes^k \CA, \CA)[-k+1]
\end{equation}
Equipped with the Gerstenhaber bracket $\{ , \}$, c.f. \cite{Gerstenhaber:1963}. In particular let $\mu : \otimes^2 \CA \to \CA$ be the multiplication map $ \mu(a,b) = ab$. The differential $\partial = \{  \cdot,\mu \}$, which satisfies $\partial^2 = 0$, turns $(C^\bullet (\CA), \partial)$ into a differential graded (dg) algebra - the Hochschild cochain complex.

In particular, for any linear map $T : \CA \to \CA$, we call the map $\partial T : \otimes^2 \CA \to \CA$ the \emph{derivator} of $T$, since
\begin{equation}
(\partial T)(\zeta , \xi) = T(\zeta\cdot \xi)-(T\zeta)\cdot \xi-\zeta\cdot (T\xi),
\end{equation}
clearly vanishes if $T$ is a derivation of the algebra, i.e. $\ker \partial_1 = \{\text{Derivations of } \CA\}$.
For the next section, the following expression of $\partial^2 T = 0$ will be important:
\begin{corollary}
The derivator $\partial T$ satisfies the following identity
\begin{equation}\label{betadistprop}
\partial T(a\cdot b,c) - \partial T(a,b \cdot c)  = a\cdot \partial T(b,c) - \partial T(a,b) \cdot c.
\end{equation}
\end{corollary}

\subsection{$\beta$ Elements}
We know that the NS Lax operator $\CL$ is not a derivation, however its derivator will play a fundamental role in the analysis to follow.
We use the notation
\begin{equation}
\beta(\xi,\zeta) := \partial\CL(\xi,\zeta) = \CL(\zeta\cdot \xi)-(\CL\zeta)\cdot \xi-\zeta\cdot (\CL\xi).
\end{equation}
\begin{lemma}\label{betaproperties}
The derivator $\beta := \partial \CL$ of the NS Lax operator has the following properties
\begin{itemize}
\item[1)] $\beta$ is $\CF$-bilinear, that is, for $\eta \in \pi_0 H$ we have
\begin{equation}\label{betafactorprop}
\beta (\zeta, \eta \cdot \xi) = \eta\cdot\beta(\zeta,\xi).
\end{equation}
In particular, $\beta$ factors through $\pi_+$
\begin{equation}\label{betafactorpiprop}
\beta (\zeta,\xi) = \beta(\zeta,\pi_+\xi).
\end{equation}
\item[2)] $\beta$ is Hochschild exact, that is, $\beta = \partial(\CL')$ where $\CL'(w^n) := \sum_{0<k\leq n}  w^{n-k} V_k$ is $\CF$-linear. Note that this means that $\beta$ does not depend on $\bareps$, $\hbar$ (i.e. $\vareps_i$).

\end{itemize}
\end{lemma}
\begin{proof}
(1) follows because $\CL$ is a derivation on $\pi_0$. (2) holds because $\CL = \CL' + (\text{derivation})$.
\end{proof}

\begin{lemma}\label{beta1formula}  For any $\zeta \in \CH$, 
\begin{equation}
\beta(w,\zeta) = \pi_0 \CL w \zeta - V_1 \zeta.
\end{equation}
\end{lemma}
\begin{proof}
\begin{eqnarray} \beta(w,\zeta) &=& \CL(w\zeta) - w\CL\zeta - \CL(w)  \zeta \\
&=& \pi_0 \CL(w\zeta) + \pi_+ \CL(w\zeta) - w\CL\zeta- (V_1 + \bareps w)  \zeta  \\
&=& \pi_0 \CL(w\zeta)- V_1 \zeta  + \left( \pi_+ \CL^+w\zeta - w(\CL+\bareps)\zeta\right) .
\end{eqnarray}
Using $\CL^+ w = w (\CL+\bareps)$ we are done.
\end{proof}

A key lemma for our inductive proofs will be
\begin{lemma}
\begin{equation}
\beta(\xi,\zeta) = \beta(\Pi\xi,w\zeta)+ (\Pi \xi)\cdot(\pi_0 \CL w \zeta) - (\pi_0 \CL\xi)\cdot(\zeta).
\end{equation}
\end{lemma}
\begin{proof}
Using \ref{betadistprop}  $\partial \beta(\xi,w \Pi\zeta) = 0$, we have
\begin{equation}
\beta(\xi,\zeta) = \beta(\Pi\xi,w\zeta)+ (\Pi \xi)\cdot\beta(w,\zeta) - \beta(\Pi\xi,w)\cdot(\zeta).
\end{equation}
Using \ref{beta1formula}, and noting that $\pi_0 \CL \xi =   \pi_0 \CL \pi_+ \xi$ we easily recover the result.
\end{proof}

Because of the $\CF=\BCe[V]$ linearity, we will often only need to perform manipulations with the basic elements,
\begin{equation}\label{betanmdef}
\beta^{n, m} := \beta(w^n, w^m).
\end{equation} We note that $\beta^{1,m} = V_{m+1}-w^m V_1$.
\begin{lemma}
Under the principal specialisation $V_k \to z$, $w \to 1$, we have $\beta(\zeta,\xi)\to 0$.
\end{lemma}
\begin{proof}
Because of $\CF$-linearity, it suffices to show this for $\beta^{m,n}$. We have
\begin{eqnarray*}
\beta^{m,n} &=& \CL' (w^{m+n}) - (\CL' w^m)w^n - w^m(\CL' w^n) \\
 &=& \sum_{k=1}^{m+n} V_k w^{m+n-k} - (\sum_{k=1}^{m}  V_k w^{m-k})w^n - w^m(\sum_{k=1}^{n} V_k w^{n-k}) \\
 &\to& \sum_{k=1}^{m+n} z  - (\sum_{k=1}^{m}  z) - (\sum_{k=1}^{n}  z) =0.
\end{eqnarray*}
\end{proof}

\subsection{Null sub-modules}

We explicitly construct two important $\CF$ sub-modules of the free $\CF$-module $\CH = \CF[w]$.

\begin{lemma}\label{nullsubmodlemma}
For $\zeta$ in $\CH_n$, we have
\begin{equation}\label{XYjackaction}
\ztr_{\gamma} (\hat j_\lambda \zeta) = \sum_{\mu} \hat c_{\lambda\mu}^{\gamma} \ztr_{\mu} (\zeta), \qquad \xtr_{\gamma} (\hat j_\lambda \zeta) = \sum_{\mu} \hat c_{\lambda\mu}^{\gamma} \xtr_{\mu} (\zeta). \end{equation}
\end{lemma}

\begin{proof}

\newcommand\nullmap{{\Omega}}
Consider the map 
\begin{equation}\label{omegamap} \nullmap(\zeta) \coloneqq \pi_0 \CL (w \cdot \zeta) = \sum_{\gamma \partition n+1} \xtr_\gamma(\zeta) \hat j_\gamma.
\end{equation}
This map satisfies
\begin{eqnarray} \nullmap(\hat j_\lambda \zeta) &=& \pi_0 \CL ( \hat j_\lambda \cdot w \zeta)  \\
 & =&  \pi_0 (\CL \hat j_\lambda ) (w \zeta) + \pi_0 \hat j_\lambda \CL (w \zeta) \\
&=&  \pi_0 w (\CL \hat j_\lambda ) ( \zeta) + \hat j_\lambda \pi_0  \CL (w \zeta) \\
&= & 0 + \hat j_\lambda \nullmap(\zeta) 
\end{eqnarray}
So $\nullmap(\hat j_\lambda \zeta)=  \hat j_\lambda \nullmap(\zeta)$, i.e. $\nullmap$ is $\BC[V]$-linear.
Expanding this out, we have
\begin{eqnarray} \sum_\gamma \xtr_\gamma(\hat j_\lambda \zeta) \hat j_\gamma &=&\hat j_\lambda \sum_\mu \xtr_\mu( \zeta) \hat j_\mu \\
&=& \sum_\mu \sum_\gamma \hat c_{\mu \lambda}^{\gamma}\xtr_\mu( \zeta) \hat j_\gamma.
\end{eqnarray}
Upon extracting the coefficients of $\hat j_\gamma$ we are done. Using relation \ref{traceprops} the result follows for the other trace $\ztr$. 
\end{proof}

Define the \emph{Null sub-modules} $\QZ^0 , \QX^0 \subset \CF[w]$ as the subspaces of all elements with vanishing $\ztr$ and $\xtr$ respectively. One can easily show that it follows from lemma \ref{nullsubmodlemma} that:

\begin{corollary}
$\QZ^0$ and $\QX^0$ are $\CF$-submodules of the free $\CF$-module $\CH=\CF[w]$. Furthermore, multiplication by $w$, i.e. $w\cdot : \QX^0 \to \QZ^0$, is a morphism of $\CF$-modules.

\end{corollary}


\subsection{Spectral Factors II}
For the next set of results, we will need an extension of the `spectral' factors that appeared earlier in section \ref{spectralfactorssection}. Let $\Gamma$ be a collection of boxes (possibly with multiplicities). We extend the previous definition \ref{spectralfactors1} to this case,
\begin{equation}
T_{\Gamma}(u) \coloneqq \prod_{b \in \Gamma} N(u-[b]).
\end{equation} 
\begin{equation}
T_{\emptyset}(u) \coloneqq 1.
\end{equation}

Furthermore, we define a generalized Kerov (co-)transition measure
\begin{equation}\label{tauboxesdef}
\htau_{\Gamma}^s = \Res_{u=[s]} T_{\Gamma}(u).
\end{equation}

One of the novel constructions that appears in this work is a simple product on the space of partitions. We will show in the next section that this product is deeply related to the structure of Jack LR coefficients.

\begin{definition}\label{starproductdef} The {\bf star product} of two partitions is the collection of boxes (with multiplicities) given by
\begin{equation}
\lambda\,\*\, \nu  \coloneqq \bigsqcup_{\substack{s\in \lambda \\ t \in \mu}} s+t.
\end{equation}
\end{definition}
With this, we have $\mu\, \*\, \{1\}  = \mu$.
For example, $\{1,3\} \,\*\, \{1,2^2\}$ is computed as
\begin{equation}
\ytableausetup{boxsize=1.0em} \ydiagram{1,3}\, \*\, \ydiagram{1,2,2} \,=\,\ytableausetup{boxsize=1.0em}
\begin{ytableau}
\, \\ 
2 & 2  & \, \\ 
2 & 3 & 2 & \, \\ 
\, & 2& 2 & \, 
\end{ytableau}
\end{equation}
where the number inside a box indicates its multiplicity (if $> 1$).
The spectral factors\footnote{Note the similarity with the work of Bourgine et al \cite{Bourgine:2017a} and the so called `Nekrasov factors' that appear therein, e.g. the first of the three terms on the RHS of their equation (2.33).  There, however, the object that appears is $\prod_{a \in \lambda, b \in \nu} S(\chi_a/\chi_b) \sim \prod_{a \in \lambda, b \in \nu} N([a-b])$, we are yet to see a correspondence with this type of factor that involves a \emph{difference} of the summed boxes.} of star products will appear in the next section,

\begin{equation}
T_{\lambda\* \nu}(u) = \prod_{c \in \lambda \* \nu} N(u-[c]) = \prod_{a \in \lambda, b \in \nu} N(u-[a+b]).
\end{equation}

\begin{figure}[htb]
\centering
\includegraphics[width=0.8 \textwidth]{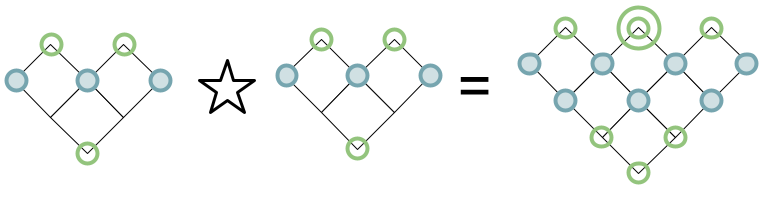}
\caption{$T_{\{1,2\}\*\{1,2\}}(u) = \frac{{ \color{olive} (u-[3,1])(u-[2,2])^2(u-[1,3])(u-[1,0])(u-[0,0])(u-[0,1]) } }{ \color{cyan} (u-[3,0])(u-[2,0])(u-[2,1])(u-[1,1])(u-[1,2])(u-[0,2])(u-[0,3])}$}
\end{figure}

\subsection{$\ytr$-trace Formulae}

We arrive at one of the substantial results of this work, hinting at a surprising amount structure in the products of Lax eigenfunctions. This is the first appearance of the novel star product \ref{starproductdef} introduced in the last section. 
\begin{theorem}\label{yuofproduct}
The $\ytr$-trace of a product of Lax eigenfunctions is given by the formula
\begin{equation}\label{yuofproducteqn}
\ytr_u( \hpsi_\lambda^s \cdot \hpsi_\nu^t ) = \frac{T_{\lambda\* \nu}(u)}{u-[s+t]}.
\end{equation}
\end{theorem}

We will prove this momentarily. First we note that the numerator only depends on the partitions, and the denominator only depends on the minima. For example,

\begin{equation}
\ytr_u( \hat\psi_{{\color{blue}1,2}}^{{\color{red}(1,1)}} \hat \psi_{{\color{blue}1,2}}^{{\color{red}(2,0)}} ) = \frac{T_{\{{\color{blue}1,2}\}\*\{{\color{blue}1,2}\}}(u)}{{\color{red}(u-[3,1])}}.\end{equation}

Back to the proof. First, we'll need a small lemma.
\begin{lemma}\label{yuofbetainprod}
The $\ytr$-trace of a product of Lax eigenfunctions can be expressed in terms of the $\ytr$-trace of $\beta$ applied to them,
\begin{equation}
 \ytr_u( \hpsi_\lambda^s \cdot \hpsi_\nu^t ) = \frac{\ytr_u\left(\beta(\hpsi_\lambda^s, \hpsi_\nu^t) \right)+1}{u-[s+t]}.
\end{equation}
\end{lemma}
\begin{proof}
\begin{equation}
\beta(\hpsi_\lambda^s, \hpsi_\nu^t) = \CL(\hpsi_\lambda^s \hpsi_\nu^t) - \CL(\hpsi_\lambda^s) \hpsi_\nu^t- \hpsi_\lambda^s\CL( \hpsi_\nu^t)=  (\CL -[s+t])(\hpsi_\lambda^s \hpsi_\nu^t).
\end{equation}
Now, using (\ref{yuLprop}), we find
\begin{equation}
\ytr_u\left(\beta(\hpsi_\lambda^s, \hpsi_\nu^t)\right)=\ytr_u\left((\CL -[s+t])(\hpsi_\lambda^s \hpsi_\nu^t)\right)=(u-[s+t]) \ytr_u( \hpsi_\lambda^s \hpsi_\nu^t ) -1.
\end{equation}
\end{proof}
With this lemma we state and prove a result clearly equivalent to Theorem \ref{yuofproduct}:



\begin{theorem}\label{ytraceofderiv}
The $\ytr$-trace of $\beta$ of a pair of Lax eigenfunctions is given by the simple formula
\begin{equation}\label{yubeta}
\ytr_u\left(\beta(\hpsi_\lambda^s, \hpsi_\nu^t) \right) = T_{\lambda\*\nu}(u) - 1,
\end{equation}
in particular, it is independent of the choices of corners $s, t$.
\end{theorem}
\begin{proof}
We prove this inductively on $k = \min( |\lambda|,|\nu|)$, the minimum of the degrees of the two arguments. Without loss of generality, we assume $|\lambda| \leq |\nu|$. 
For the base case, $k = 0$, we have $ \hpsi_\emptyset^{(0,0)}= 1$, and we have $\beta(1, \hpsi_\lambda^s)=0$, thus
\begin{equation} \ytr_u (\beta(\hpsi_\emptyset^{(0,0)}, \hpsi_\lambda^s)) = 0 = T_{\emptyset}(u) -1. \end{equation}

For the inductive step, let us assume that \ref{yuofproducteqn} (and hence \ref{yubeta}) holds for all $|\lambda| = k \leq K $. We will show that this implies \ref{yubeta} holds for $\beta(\hpsi_{\lambda'}^s, \hpsi_\nu^t)$ with $|\lambda'| = K+1 $. We use \ref{betafactorpiprop}, with $\pi_+ \hpsi_\lambda^s = w \Pi \hpsi_\lambda^s$, and \ref{betadistprop} to get
\begin{eqnarray}
\beta(\hpsi_\lambda^s, \hpsi_\nu^t) &=&  \beta( \Pi\hpsi_\lambda^s, w\hpsi_\nu^t)+  \Pi\hpsi_\lambda^s  \cdot \hat j_{\nu+t}  - \hat j_{\lambda}  \cdot \hpsi_\nu^t  \\
&=& \sum_{q \in \remset_{\lambda}} c_{\lambda}^{q,s} \left( \beta(  \hpsi_{\lambda-q}^q, w\hpsi_\nu^t)+   \hpsi_{\lambda-q}^q \cdot  \hat j_{\nu+t} \right) - \hat j_{\lambda}  \cdot \hpsi_\nu^t,
\end{eqnarray}
where we have used the expansion \ref{Piactionpsi}
\begin{equation}
\Pi \hpsi_\lambda^s = \sum_{q \in \remset_{\lambda}}a_{\lambda}^{q,s} \hpsi_{\lambda-q}^q , \text{ with } \quad \sum_{q \in \remset_{\lambda}} a_{\lambda}^{q,s} = 1.
\end{equation}
Thus we have
\begin{equation}\label{bigsum22}
\ytr_u(\beta(\hpsi_\lambda^s, \hpsi_\nu^t)) = \sum_{q \in \remset_{\lambda}} a_{\lambda}^{q,s} \left( \ytr_u(\beta(  \hpsi_{\lambda-q}^q, w\hpsi_\nu^t))+   \ytr_u(\hpsi_{\lambda-q}^q \cdot  \hat j_{\nu+t} )\right) - \ytr_u(\hat j_{\lambda}  \cdot \hpsi_\nu^t).
\end{equation}
For the two terms inside the sum, we can use the inductive hypothesis on each, since $|\lambda-q| = K < K+1$. For the first of these two terms,
\begin{eqnarray}
 \ytr_u(\beta(  \hpsi_{\lambda-q}^q, w\hpsi_\nu^t)) &=& \sum_{v \in \addset_{\nu}+t} d_{\nu}^{t,v}  \ytr_u(\beta(  \hpsi_{\lambda-q}^q, \hpsi_{\nu+t}^v )) \\
 &=& \left(  T_{(\lambda-q) \* (\nu+t)} - 1 \right) \sum_{v \in \addset_{\nu}+t} d_{\nu}^{t,v} \\
 &=&  T_{(\lambda-q) \* (\nu+t)} - 1,
\end{eqnarray}
where we have used the expansion \ref{wactionpsi}
\begin{equation}
w\hpsi_\nu^t = \sum_{v \in \addset_{\nu+t}} d_{\nu}^{t,v}\hpsi_{\nu+t}^v, \text{ with } \quad \sum_{v \in \addset_{\nu+t}} d_{\nu}^{t,v} = 1.
\end{equation}
For the second of the two terms inside the sum of (\ref{bigsum22}), we use the inductive hypothesis again
\begin{eqnarray}
\ytr_u( \hpsi_{\lambda-q}^q \cdot \hat j_{\nu+t} ) &=& \sum_{v \in \addset_{\nu+t}} \htau_{\nu+t}^v \ytr_u( \hpsi_{\lambda-q}^q \cdot \hpsi_{\nu+t}^v) \\
&=& T_{{(\lambda-q)}\* (\nu+t)}(u)\sum_{v \in \addset_{\nu}} \frac{\htau_{\nu+t}^v}{u-[q+v]}\\
& =& T_{{(\lambda-q)}\* (\nu+t)}(u)\sum_{v \in \addset_{\nu}} \frac{\htau_{\nu+t}^t}{(u-[q])-[v]} \\
& =& T_{{(\lambda-q)}\* (\nu+t)}(u)(T_{\nu+t}(u-[q]) -1) \\
& =& T_{{(\lambda-q)}\* (\nu+t)}(u)(T_{q \* (\nu+t)}(u) -1).
\end{eqnarray}
That is
\begin{equation}\label{yupsij}
\ytr_u(\hpsi_{\lambda-q}^q \cdot  \hat j_{\nu+t} ) = T_{\lambda\* (\nu+t)}(u) - T_{(\lambda-q)\* (\nu+t)}(u).
\end{equation}
Putting these two together, we get
\begin{equation}\label{formulayu1}
\ytr_u( \beta(  \hpsi_{\lambda-q}^q, w\hpsi_\nu^t)+ \hpsi_{\lambda-q}^q \cdot  \hat j_{\nu+t}  ) = T_{\lambda \* (\nu+t)} - 1.
\end{equation}
With this, we can write \ref{bigsum22} as
\begin{eqnarray}
\ytr_u(\beta(\hpsi_\lambda^s, \hpsi_\nu^t)) &= &\sum_{q \in \remsetp_\lambda} a_{\lambda}^{q,s} \left(T_{\lambda \* (\nu+t)} - 1\right)   - \ytr_u(\hat j_{\lambda}  \cdot \hpsi_\nu^t )\\
  &= & \left(T_{\lambda \* (\nu+t)} - 1\right)  - \ytr_u(\hat j_{\lambda}  \cdot \hpsi_\nu^t ).\label{eqnliney}
\end{eqnarray}
Note that from this formula we can see that all the $s$-dependence falls out in the RHS.
We can't use the inductive hypothesis on the second term in this expression, as it has lowest degree $|\lambda| = K+1$. However, we can continue by expanding
\begin{eqnarray}
\ytr_u\left( \hat j_\lambda \cdot \hpsi_\nu^t \right) &=& \sum_{s \in \addset_{\lambda}} \htau_\lambda^s \,\ytr_u\left( \hpsi_\lambda^s \cdot \hpsi_\nu^t \right)  \\
&=& \sum_{s \in \addset_{\lambda}} \htau_\lambda^s\,  \frac{\ytr_u\left(\beta(\hpsi_\lambda^s, \hpsi_\nu^t) \right)+1}{u-[s+t]}  \\
&=& \left(\ytr_u\left(\beta(\hpsi_\lambda^s, \hpsi_\nu^t)\right)+1 \right)\sum_{s \in \addset_{\lambda}} \htau_\lambda^s\,  \frac{1}{u-[s+t]}  \\
&=& \left(\ytr_u\left(\beta(\hpsi_\lambda^s, \hpsi_\nu^t)\right)+1 \right)(T_{\lambda\*t  }(u)-1). 
\end{eqnarray}
On the second line we have used \ref{yuofbetainprod}, and in the third line we have used the $s-$independence of $\ytr_u\left(\beta(\hpsi_\lambda^s, \hpsi_\nu^t)\right)$.
Combining the above with the expression \ref{eqnliney}, we get
\begin{equation}
\ytr_u(\beta(\hpsi_\lambda^s, \hpsi_\nu^t)) =  \left(T_{\lambda \* (\nu+t)} - 1\right)  -  \left(\ytr_u\left(\beta(\hpsi_\lambda^s, \hpsi_\nu^t)\right)+1 \right).(T_{\lambda\*t  }(u)-1)
\end{equation}
Rearranged, this yields
\begin{equation}
\ytr_u(\beta(\hpsi_\lambda^s, \hpsi_\nu^t)) T_{\lambda\* t  }(u) = T_{\lambda \* ( \nu+t)}(u)  - T_{\lambda\* t  }(u),
\end{equation}
and completes the inductive step:
\begin{equation}
\ytr_u(\beta(\hpsi_\lambda^s, \hpsi_\nu^t)) =  T_{\lambda \* \nu}(u) - 1.
\end{equation}
\end{proof}

Thus Theorems \ref{yuofbetainprod} and \ref{ytraceofderiv} are proven.

\subsection{Back to Jack-Lax LR Coefficients}

Using the $\ytr$-trace formula \ref{yuofbetainprod}, we can determine certain Jack-Lax Littlewood-Richardson coefficients. First, we state a straightforward result:
\begin{lemma}\label{rectanglelemma} Let $v_* = (m-1,n-1)$.
The partition $m^n-{v_*}$ (i.e. a rectangle less a box) is the only partition of its size with a corner at $v_*$. 
\end{lemma}
With this Lemma, we can show a simple form that Jack-Lax LR coefficients can take.
\begin{proposition} We have
\begin{equation} \hpsi_\mu^s \hpsi_{\nu}^{t} = \Res_{u=[v_*]} \left(\frac{T_{\mu \*\nu }(u)}{u-[s+t]} \right) \hpsi_{m^n-{v_*}}^{v_*} + \ldots,
\end{equation}
that is,
\begin{equation}
\hat c_{\mu,\nu;v_*}^{s,t;m^n-v_*} =  \Res_{u=[v_*]} \left(\frac{T_{\mu \*\nu }(u)}{u-[s+t]} \right).
\end{equation}
\end{proposition}
\begin{proof}
In the general Lax eigenfunction expansion, we have 
\begin{equation}
\hpsi_\mu^s \hpsi_{\nu}^{t} = a \hpsi_{m^n-{v_*}}^{v_*} + \cdots,
\end{equation}
where the coefficient $a$ is to be determined. The $\ytr$-trace \ref{yuofbetainprod} of both sides is
\begin{equation}
\frac{T_{\mu \*\nu }(u)}{u-[s+t]} = \frac{a}{u-[v_*]} + \cdots,
\end{equation}
and we know by Lemma \ref{rectanglelemma} that $a$ is the only residue at $u=[v_*]$ on the RHS. 
\end{proof}
We investigate and develop further these kind of explicit formulae for Jack-Lax LR coefficients in the sequel \cite{Alexandersson:2023}.

\subsection{Conjecture on Selection Rules}

Recall the well-known selection rules for multiplication of Jack functions (c.f. \cite{Macdonald:1995}),
\begin{equation}
j_\mu \cdot j_\nu \subset  \bigoplus_{\gamma: \mu \cup \nu \subset \gamma} \BC j_\gamma 
\end{equation}
We conjecture that this rule extends to the multiplication of Lax eigenfunctions,
\begin{conjecture}\label{selectionruleconjecture}
\begin{equation} \QZ_\mu \cdot \QZ_\nu \subset \bigoplus_{\gamma: \mu \cup \nu \subset \gamma} \QZ_\gamma \end{equation}
\end{conjecture}
One can easily show that conjecture \ref{selectionruleconjecture} is equivalent to 
\begin{equation} \QX_\nu \cdot \QZ_\gamma \subset  \bigoplus_{\lambda \supseteq \nu, \gamma} \QX_\lambda.\end{equation}
We offer the following computations as evidence for this conjecture
\begin{equation} \hpsi_{1^r}^{(r,0)}  \hpsi_{m}^{(0,m)} = \frac{[0,-m]}{[r,-m]} \hpsi_{\{1^r, m\}}^{(0,m)}+ \frac{[r,0]}{[r,-m]} \hpsi_{\{1^{r-1},m+1\}}^{(r,0)}. \end{equation}
\begin{equation} \hpsi_{1^r}^{(r,0)}  \hpsi_{m}^{(1,0)} =  \frac{[0,-m][r,0]}{[r,-m][r+1,-m]} \hpsi_{\{1^r, m\}}^{(0,m)}+\frac{[r+1,0][1,-m]}{[1,0][r+1,-m]} \hpsi_{\{1^r, m\}}^{(r+1,0)}+ \frac{[r,0][0,-m]}{[r,-m][-1,0]} \hpsi_{\{1^{r-1},m+1\}}^{(r,0)}. \end{equation}
\begin{equation} \hpsi_{1^r}^{(0,1)}  \hpsi_{m}^{(1,0)}  = a \hpsi_{\{1^r, m\}}^{(0,m)}+b \hpsi_{\{1^r, m\}}^{(1,1)}+c \hpsi_{\{1^{r-1},m+1\}}^{(1,1)}+d \hpsi_{\{1^{r-1},m+1\}}^{(r,0)}. \end{equation}

\begin{definition}

\end{definition}

\subsection{Trace twist}

Comparing \ref{yubeta} and \ref{yujack}, we notice that we have an seemingly coincidental equality of $\ytr$-traces,
\begin{equation}\label{yutraceequal1}
\ytr_u(\beta(\hpsi_{1}^{v},\hpsi_\lambda^s)) = T_{\lambda}(u) - 1 = \ytr_u(\hat j_\lambda).
\end{equation}
This is peculiar, as the right hand side is an element of degree $|\lambda|+1$, and the right hand side is of degree $|\lambda|$. Here we begin to explore this connection further, showing that it is in fact not a coincidence, but rather hints at deeper structure of the algebra of Lax eigenfunctions.

\begin{proposition}\label{basetwistedtrace}
The following relations between traces hold,
\begin{itemize}
\item $\ytr_u(\beta(\hpsi_1^{v},\hpsi_\lambda^s)) = \ytr_u(\hat j_\lambda)$,
\item $\xtr_\nu(\beta(\hpsi_1^{v},\hpsi_\lambda^s)) = 0$, $\forall \nu$,
\item $\ztr_\mu(\hat j_\lambda) = 0$, $\forall \mu$,
\item $\ztr_\gamma(\beta(\hpsi_1^{v},\hpsi_\lambda^s)) = -\xtr_\gamma(\hat j_\lambda)$, $\forall\gamma$.
\end{itemize}
In other words, we have $\beta(\hpsi_1^{v},\hpsi_\lambda^s) \in \QX^0_{n+1}$, and $\hat j_\lambda \in \QZ^0_n$, and the full traces are related by
\begin{equation}\label{twistedtraces1}
 \Tr_{n+1}(\beta(\hpsi_1^{v},\hpsi_\lambda^s)) = \rho_*\circ \Tr_n ( \hat j_\lambda). \end{equation}
where 
\begin{equation}\label{tracetwistdef}
\rho_* : (\{x_\gamma\},\{y^s\},\{0\}) \mapsto (\{0\},\{y^s\},\{-x_\gamma\}).
\end{equation}
\end{proposition}
\begin{proof}
We note that since $\beta$ factors through $\pi_+$, so we have $\beta(\hpsi_1^{v},\hpsi_\lambda^s)=\beta(\pi_+\hpsi_1^{v},\hpsi_\lambda^s) = \beta(w,\hpsi_\lambda^s)$. Next, we know that (1) has already been observed. For (2), we start with \ref{beta1formula}, $\beta(w,\hpsi_\lambda^s) = \pi_0 \CL w \hpsi_\lambda^s - V_1 \hpsi_\lambda^s = \hat j_{\lambda+s} - \hat j_1 \hpsi_\lambda^s$. We then have
\begin{equation}
\xtr_\gamma(\beta(w,\hpsi_\lambda^s)) = \xtr_\gamma(\hat j_{\lambda+s}- \hat j_1 \hpsi_\lambda^s) = \sum_{t} \delta^\gamma_{\lambda+s+t}\hat \tau_{\lambda+s}^t- \hat c_{1\sigma}^\gamma \xtr_\sigma(\hpsi_\lambda^s)= 0.\end{equation}
where we have used the Kerov relation $\hat \tau_{\lambda+s}^t = \hat c_{1,\lambda+s}^{\lambda+s+t}$ (\ref{kerovcoefficient}). Next, (3) holds because $j_\lambda$ has vanishing top power in $w$. For (4), we have
\begin{equation}
\ztr_\gamma(\beta(w,\hpsi_\lambda^s)) = \ztr_\gamma(\hat j_{\lambda+s}- \hat j_1 \hpsi_\lambda^s) = 0- \hat c_{1\sigma}^\gamma \ztr_\sigma(\hpsi_\lambda^s)= - \hat c_{1\lambda}^\gamma = -\sum_s \delta_{\lambda+s}^\gamma\xtr_\gamma( \htau_{\lambda}^s \hpsi_\lambda^s  ).\end{equation}
\end{proof}
In the next section \ref{thetaelementsubsection} we will demonstrate a general version of this phenomena relating traces of certain elements of different degrees, which will be crucial for our work. We will make continual use of the following map:
\begin{definition}(Trace twist)
$\rho_* : \Im \Tr\,\CZ^0_n \to \Im \Tr\, \CX^0_{n+1}$ given by \ref{tracetwistdef}.
\end{definition}
Note: The range and domains are vector spaces of different dimensions, and it is not clear yet that this map lands in the image of the trace.

\subsection{$\theta$ elements}\label{thetaelementsubsection}
In the previous section, we found in \ref{twistedtraces1} that there was an element $\theta :=\hj_\lambda$ one degree lower than $\beta$ whose trace was related by the twist \ref{tracetwistdef}. We will show that such an element can always be constructed for $\theta(\hpsi_\mu^{t},\hpsi_\lambda^s)$ of any choices of $\mu,t, \lambda,s$. That is, there exists a canonical $\theta \equiv \theta(\hpsi_\mu^{t},\hpsi_\lambda^s)$, with $\theta(\hpsi_1^{v} ,\hpsi_\lambda^s) = \hj_\lambda$, such that
\begin{equation}
\Tr(\beta(\hpsi_\mu^{t},\hpsi_\lambda^s)) = \rho_*\circ \Tr ( \theta(\hpsi_\mu^{t},\hpsi_\lambda^s) ).
\end{equation}
In this case, we'd say that the elements $\beta$ and $\theta$ have \emph{twisted traces}.
\begin{definition}
Let $\theta : \CH_{n} \times \CH_{m} \to \CH_{n+m-1}$ be the degree $-1$ symmetric map defined by
\begin{equation}
\theta := \{ \beta,\Pi \},
\end{equation}
where $\{,\}$ is the Hochschild bracket, i.e.
\begin{equation}
\theta(\zeta,\xi)=\beta(\Pi \zeta,\xi)+\beta(\zeta,\Pi \xi)  - \Pi\beta (\zeta ,\xi ).
\end{equation}
\end{definition}
We can easily verify the following properties, in parallel with those for $\beta$ (\ref{betaproperties}).
\begin{lemma}
$\theta$ has the following properties,
\begin{itemize}
\item[1)] $\theta$ is $\BC[V]$-linear,
\begin{equation}
\theta(\zeta,V_k\cdot \xi) = V_k\cdot\theta(\zeta,\xi) .
\end{equation}
In particular, $\theta$ factors through $\pi_+$
\begin{equation}
\theta(\zeta,\xi) = \theta(\zeta,\pi_+ \xi).
\end{equation}
\item[2)] $\theta$ is an exact operator (hence closed), 
\begin{equation}\label{thetaexact}
\theta = \partial ( \CL' \circ \Pi ). 
\end{equation}
Note that $\theta$ is independent of $\bareps, \hbar$.

\item[3)] The simplest case is given by the formula
\begin{equation} \theta(\hpsi_1^{v} ,\hpsi_\lambda^s) = \hat j_\lambda.
\end{equation}
\end{itemize}
\end{lemma}

Because of the $\BC[V]$ linearity, we often work with the basic elements (compare with \ref{betanmdef})
\begin{equation}
\theta^{n,m} :=  \theta(w^n, w^m).
\end{equation}
One can easily check that these $\theta$ elements include the symmetric function generators:
\begin{equation}\label{theta1mvm}
\theta^{1,m} = V_m.
\end{equation}

\begin{lemma}
The following relation holds,
\begin{equation}\label{betaident2}
\beta(b,\Pi c)  - \Pi\beta ( b ,c ) = (\Pi b) \cdot (\pi_0 \CL c),
\end{equation}
and hence $\theta$ can alternately be given by
\begin{equation}\label{thetaexpression}
\theta( \zeta,\xi) = \Pi \zeta \cdot \pi_0 \CL \xi + \beta ( \Pi \zeta,\xi ).
\end{equation}
\end{lemma}
\begin{proof}
By inspection, we see that both sides of \ref{betaident2} are $\BC[V]$ bilinear. Thus we reduce the statement to the case where $b=w^n$ and $c=w^m$.
Let $T_n^m$ be the quantity
\begin{eqnarray}
T_n^m &=& \beta(w^n,\Pi w^m)  - \Pi\beta ( w^n ,w^m ) - \Pi w^n . \pi_0 \CL w^m \\
&=& \beta^{n,m-1}- \Pi \beta^{n,m} - w^{n-1} V_m.
\end{eqnarray}
We prove by induction that $T_n^m = 0$. For the base case, $T_0^m$ it is obviously true, as both sides are zero. We also prove $T_1^m$ directly,
\begin{eqnarray}
\beta(w,w^{m-1})  - \Pi \beta ( w ,w^m ) &=& (V_{m}-w^{m-1}V_1) - \Pi( V_{m+1} - w^m V_1) \\
&=& V_{m} \\
&=& \pi_0 \CL w^m.
\end{eqnarray}
Now, for $T_n$
\begin{eqnarray*}
\beta(w^n,\Pi w^m)  - \Pi\beta ( w^n ,w^m ) &=& \beta(w^{n-1}.w, w^{m-1})  - \Pi\beta ( w^{n-1}.w ,w^m ) \\
&=&  \beta(w^{n-1},w.w^{m-1})  - \Pi\beta ( w^{n-1},w.w^m ) \\
&&+  w^{n-1}\beta(w, w^{m-1}) - \Pi w^{n-1}.\beta ( w ,w^m )  \\
&&-  \beta(w^{n-1}.w) w^{m-1}  + \Pi\beta ( w^{n-1},w)w^m \\
&=&  \beta(w^{n-1},w^{m})  - \Pi\beta ( w^{n-1},w^{m+1} ) \\
&&+  w^{n-1}\beta(w, w^{m-1}) - w^{n-2} \Pi w.\beta ( w ,w^m )  \\
&&-  \beta(w^{n-1}.w) w^{m-1}  + w^{m-1}\Pi w\beta ( w^{n-1},w)\\
&=&  \beta(w^{n-1},w^{m})  - \Pi\beta ( w^{n-1},w^{m+1} ) \\
&&+  w^{n-1}\beta(w, w^{m-1}) - w^{n-2}(w\Pi+ \pi_0)  \beta ( w ,w^m )  \\
&&-  \beta(w^{n-1},w) w^{m-1}  + w^{m-1} \beta ( w^{n-1},w) \\
&=&  \beta(w^{n-1},w^{m})  - \Pi\beta ( w^{n-1},w^{m+1} ) \\
&&+  w^{n-1}\left(\beta(w, w^{m-1}) - \Pi  \beta ( w ,w^m ) \right)\\
&& - w^{n-2} \pi_0 \beta ( w ,w^m ),
\end{eqnarray*}
where we have used $\Pi w - w \Pi = \pi_0$.
Using $T_1^m =0$ and that
\begin{equation}
w^{n-2} \pi_0 \beta ( w ,w^m )  = w^{n-2} \pi_0 \CL w^{m+1},
\end{equation} 
we find that the big sequence of equalities above reduces to
\begin{equation}
T_n^m = T_{n-1}^{m+1}.
\end{equation}
By induction, all $T_n^m =0$, and we are done.
\end{proof}

Next, we inspect the structure of these $\theta$ basis elements, finding expressions for their projection onto $\CF \subset \CH$ and its compliment.
\begin{proposition}\label{pitheta}
\begin{equation} \pi_0 \theta^{n,m} = \theta^{1,m+n-1} = V_{m+n-1}, \end{equation}
\begin{equation}\label{piplustheta} \pi_+ \theta^{n,m} = w \beta^{n-1,m-1}. \end{equation}
\end{proposition}
\begin{proof}
From \ref{thetaexpression}, we have
\begin{equation} \theta^{n,m}= w^{n-1} V_m + \beta^{n-1,m}. \end{equation}
From $\partial \beta(w^{n-1}, w^{m-1}, w) = 0$, we have
\begin{eqnarray}\label{betahoch}
\beta^{n-1,m} &=& \beta^{n+m-2,1} - w^{n-1}\beta^{m-1,1} + \beta^{n-1,m-1}w \\
& =& V_{n+m-1} - w^{n+m-2}V_1 - w^{n-1}\beta^{m-1,1} + \beta^{n-1,m-1}w.
\end{eqnarray}
So, we have
\begin{eqnarray} \pi_+ \theta^{n,m} &=& w^{n-1} V_m - w^{n+m-2}V_1 - w^{n-1}\beta^{m-1,1} + \beta^{n-1,m-1}w \\
&=& w^{n-1} (V_m - w^{m-1}V_1) - w^{n-1}\beta^{m-1,1} + \beta^{n-1,m-1}w \\
&=& \beta^{n-1,m-1}w. \end{eqnarray}
\end{proof}
By expanding elements in powers of $w$, i.e. $ \zeta = \sum_{i=0}^n \zeta_n w^n$, we easily find the following extension of proposition \ref{pitheta}:
\begin{corollary}\label{pi0theta}
\begin{equation} \pi_0 \theta(\zeta, \xi) = \theta(w, \partial \Pi(\zeta, \xi)) = \pi_0 \CL\partial \Pi(\zeta, \xi), \end{equation}
\begin{equation}\label{piplusthetagen} \pi_+ \theta(\zeta, \xi) = w \beta(\Pi\zeta, \Pi\xi) .
\end{equation}
That is,
\begin{equation}
\theta(\zeta, \xi) = \pi_0 \CL\partial \Pi(\zeta, \xi) + w \beta(\Pi\zeta, \Pi\xi).
\end{equation}
\end{corollary}

The operator $\partial \Pi(\zeta,\xi) = w^{-1}\pi_+(\zeta) \pi_+(\xi)$ will make an important appearance in later results, see \ref{betaequalsrhotheta}.

\subsection{Relation between traces of $\beta$ and $\theta$}

Next, we show a generalization of the results of Proposition \ref{basetwistedtrace} and equation \ref{twistedtraces1} to the case of arbitrary parameters $\beta(\zeta, \xi)$. We make the first steps towards a deeper understanding of this property later in section \ref{twistssection}.

\begin{proposition}\label{twistedtraces}
The Traces of $\beta$ and $\theta$ are related by the twist $\rho_*$, that is for all $\zeta \in \CH_k, \xi \in \CH_{n-k}$, we have
\begin{equation}\label{betathetatwistedtraces}
\Tr_{n}(\beta(\zeta, \xi)) = \rho_* \circ \Tr_{n-1} (\theta(\zeta,\xi)).
\end{equation}
Equivalently, the following hold
\begin{itemize}\label{betathetarelations}
\item $\ytr_u( \theta(\zeta, \xi))=  \ytr_u (\beta(\zeta, \xi))$,
\item
$\xtr_\nu\left( \beta(\zeta, \xi) \right) = 0 $, $\forall \nu$,
\item 
$ \ztr_\sigma\left( \theta(\zeta, \xi) \right) = 0 $, $\forall \sigma$
\item
$\ztr_\gamma (\beta(\zeta,\xi)) = - \xtr_\gamma (\theta(\zeta,\xi))$, $\forall\gamma$.
\end{itemize}
\end{proposition}
\begin{proof}
The first statement, following from formula \ref{yubeta}, is equivalent to
\begin{equation}
\ytr_u(\theta(\hpsi_{\lambda}^{s}, \hpsi_\nu^t)) =T_{\lambda \* \nu}(u) -1.
\end{equation}
Using \ref{thetaexpression} along with \ref{Piactionpsi}, we write $\Pi \hpsi_{\lambda}^{s} = \sum_{q \in \remset_{\lambda}} c^q_{\nu,s} \hpsi^{q}_{\lambda-q}$, to find
\begin{equation}
\theta(\hpsi_{\lambda}^{s}, \hpsi_\nu^t) = \sum_{q \in \remset_{\lambda}} c^q_{\nu,s} \left( \hpsi^{q}_{\lambda-q} . \hat j_\nu + \beta(\hpsi^{q}_{\lambda-q} , \hpsi_\nu^t )  \right).
\end{equation}
Then, using \ref{yupsij} and  \ref{yubeta}, we have
\begin{eqnarray}
 \ytr_u (\theta(\hpsi_{\lambda}^{s}, \hpsi_\nu^t)) &=& \sum_{q \in \remsetp_\lambda} c^q_{\nu,s} \left(( T_{\lambda\*\nu}(u) - T_{(\lambda-q)\*\nu}(u)) +(T_{(\lambda-q)\*\nu}(u)-1)  \right)\\
&=& \left(\sum_{q \in \remsetp_\lambda} c^q_{\nu,s}\right) \left( T_{\lambda\*\nu}(u) -1  \right)\\
&=&T_{\lambda \* \nu}(u) -1.
\end{eqnarray}

For the next three statements, we work just for powers of $w$, since the statements have $\BC[V]$ linearity by property \ref{XYjackaction} . 
Let $T^{m,n}_{\sigma}$ be the statement that $\ztr_\sigma \theta^{n,m}  = 0$, let $B^{m,n}_\nu$ be the statement that $\xtr_\nu \beta^{n, m}  = 0$, and $N^{n,m}_\lambda$ be the statement $\xtr_\lambda \theta^{n,m} + \ztr_\lambda \beta^{n, m} = 0$. The base case $T^{n-1,1}_\sigma$ is true, since $\ztr_\sigma(\theta^{n-1,1}) = \ztr_\sigma(V_{n-1}) = 0$. From the formula \ref{thetaexpression}, we have
\begin{equation}\label{eq100}\quad \ztr_\sigma\theta^{n, m} =  \ztr_\sigma ( w^{n-1}V_{m})+\ztr_\sigma\beta^{n-1, m} \end{equation}
On the other hand, from $\partial \theta(w, w^{n-1}, w^m) = 0$, we have
\begin{equation} \theta^{n,m} - \theta^{1,n+m-1} = w\theta^{n-1,m} -  \theta^{1,n-1}w^m. \end{equation}
Taking the $\ztr$-trace of this, then using the base case and \ref{theta1mvm} ( $\theta^{1,n+m-1} = V_{n+m-1}$, so $\ztr_\sigma\theta^{1,n+m-1}  = 0$), we find
\begin{equation}\label{eq101}\quad \ztr_\sigma \theta^{n,m} =\xtr_\sigma \theta^{n-1,m} - \ztr_\sigma (w^mV_{n-1}). \end{equation}
Adding \ref{eq100} and \ref{eq101} together we get
\begin{equation}\label{eqnwithvnwn} 2 \ztr_\sigma\theta^{n, m} =  \ztr_\sigma ( w^{n-1}V_{m})+\ztr_\sigma\beta^{n-1, m}+\xtr_\sigma \theta^{n-1,m} - \ztr_\sigma (w^mV_{n-1}). \end{equation}
Next we show that $ \ztr_\sigma( w^n V_m)$ is symmetric under $n,m$.
\begin{eqnarray}
 \ztr_\sigma( w^n V_m)  &=& m \hbar \sum_{\lambda \vdash m}  \ztr_\sigma( w^n \hat j_\lambda ) / |\hat j_\lambda |^2 \\
&=& m \hbar \sum_{\lambda \vdash m} \sum_{\mu \vdash n} \hat c_{\mu\lambda}^\sigma \ztr_\mu( w^n ) / |\hat j_\lambda |^2 \\
&=& m \hbar \sum_\lambda \sum_\mu \hat c_{\mu\lambda}^\sigma \ztr_\mu( w \hat q_\mu ) / |\hat j_\lambda |^2 |\hat j_\mu |^2 \\
&=& (m \hbar)(n\hbar) \sum_\lambda \sum_\mu \hat c_{\mu\lambda}^\sigma / |\hat j_\lambda |^2 |\hat j_\mu |^2. \end{eqnarray}
Following from the symmetry of the Jack LR coefficients, this is symmetric.
Using this, equation \ref{eqnwithvnwn} becomes
\begin{equation} 2 \ztr_\sigma\theta^{n, m} =  \ztr_\sigma\beta^{n-1, m}+\xtr_\sigma \theta^{n-1,m}.  \end{equation}
So we see that $T^{n,m}_\sigma \Leftrightarrow N^{n-1,m}_\sigma$. If we assume $T^{m,n}_\sigma$, then we have
 \begin{eqnarray}
 0&=& \ztr_\sigma(\theta^{n,m})\\
 &=&\ztr_\sigma(w^{n-1} V_m) + \ztr_\sigma(\beta^{n-1,m}) \\
&=& \ztr_\sigma(w^{n-1} V_m) -\xtr_\sigma(\theta^{n-1,m}) \mod N^{n-1,m}_\sigma\\
&=&  \ztr_\sigma(w^{n-1} V_m) -\xtr_\sigma(\Pi (w^{n-1}) V_m )-\xtr_\sigma(\beta^{n-2,m}) \\
&=&  \ztr_\sigma(w^{n-1} V_m)  -\xtr_\sigma(\Pi ( w^{n-1}  V_m)) -\xtr_\sigma (\beta^{n-2,m}) \\
&=&  -\xtr_\sigma (\beta^{n-2,m}).
 \end{eqnarray}
Next, we use \ref{piplustheta} and the properties (\ref{traceprops}) to show for all $n,m$, we have
\begin{equation} \ztr_\lambda \theta^{n,m} = \xtr_\lambda \beta^{n-1,m-1}. 
\end{equation}
Thus $0=\xtr_\sigma (\beta^{n-2,m})= \ztr_\sigma (\theta^{n-1,m+1}) = 0$.
And so we find the chain of implications
\begin{equation}
T^{n,m}_\sigma \implies N^{n-1,m}_\sigma \implies B^{n-2,m}_\sigma \implies T^{n-1,m+1}_\sigma.
\end{equation}
So $T^{n,m}_\sigma$ is true for all $n,m$, and thus so is  $N^{n,m}_\sigma$ and $B^{n,m}_\sigma$.
\end{proof}

Lastly, we show the converse of the relation $\ztr_\sigma \theta = 0$ (c.f. \ref{betathetarelations}), that is, if $\ztr_\sigma \xi = 0$ then $\xi$ is in the $\CF$-span of the basic $\theta$ elements.
\begin{proposition}\label{nullgeneration}$\,$
\begin{itemize}
\item
The null module $\QZ^0 \subset \CF[w]$ is generated as a $\CF$-module by the elements
\begin{equation} \theta^{n,m} \in \QZ^0_{m+n-1}. \end{equation}
\item
The null module $\QX^0\subset \CF[w]$ is generated as a $\CF$-module by the elements
\begin{equation} \beta^{n,m} \in \QX^0_{m+n}. \end{equation}
\end{itemize}
\end{proposition}
\begin{proof}
We begin by showing that the second statement follows from the first, and then we prove the first statement directly. 
Assume $\zeta$ satisfies $x_\gamma(\zeta) = 0, \forall \gamma$. By the trace properties \ref{traceprops}, $w\zeta$ then satisfies $\ztr_\sigma(w\zeta) = 0$. If we assume the first statement of \ref{nullgeneration}, we can write $w\zeta = \sum_i \theta(wa_i,wb_i)$ for some collection of $a_i,b_i$. Now as we know this expression is in $\pi_+$, we can use \ref{piplusthetagen} to write it as $\sum_i \pi_+\theta(wa_i,wb_i) = \sum_i w \beta(a_i,b_i)$. Thus $\zeta = \sum_i  \beta(a_i,b_i)$, and so the second statement of \ref{nullgeneration} follows from the first.

Now we prove the first statement directly via a dimension count, that is, we will show
\begin{equation}\label{thetaspanisz0}
\sum_{k}  \left( \dim T_k \right) x^k  = \sum_{k}  \left( \dim \QZ^0_k \right) x^k,
\end{equation}
where $T_k = \Span \{ \theta^{m,k-m+1}\}_{m}$. We have the natural inclusion $T_k \hookrightarrow \QZ^0_k$. Our goal will be to produce a resolution of $T_k$ as an $\CF$-module and show that this inclusion is surjective.

Since we have the symmetry property, $\theta^{m,n} = \theta^{n,m}$, we consider the spanning set for $T_k$ consisting of $\Theta^{a,b}$ for $0\leq a< b$ and $a+b=k$, given by
\begin{equation}
\Theta^{a,b} := \theta^{a+1,b} - \theta^{a,b+1}.
\end{equation}


The motivation for this choice is the property \ref{thetaexact}, which states $\theta^{a+1,b} - \theta^{a,1+b} = w^a \theta^{1,b}-\theta^{a,1}w^b $. These allow us to write
\begin{equation}
\Theta^{a,b} =  w^{a} V_{b}- w^b V_{a},
\end{equation}
where we have used $\theta^{1,n}=V_n$. In particular, this shows that $\Theta^{a,b}$ is linear in the $V_i$, as is $\theta^{a,b}$.

Next, we consider the Koszul complex of $\CF$. That is, let $\CV = \oplus_{k=0}^{\infty} \BC V_k$ be the graded vector space of our power sum variables (including $V_0$), and let $\CA := S(\CV \oplus \CV[1]) = \CF \otimes \Lambda \CV$ be the Koszul complex. 
In particular, $\CA = \oplus_{n> 0,k>0} \CA^{n}_{k}$ generated as a free $\CF$-module by the wedge products
\begin{equation}
W^{\underline a}=W^{a_1,a_2,\ldots,a_n} := W^{a_1} \wedge W^{a_2} \wedge \ldots \wedge W^{a_n}
\end{equation}
As a basis of $\CA^{n}_{k}$, we take $W^{\underline a}$ with $0\leq a_1 < a_2< \ldots < a_n$, where the degree is  $k = \mathrm{deg}\, W^{a_1,a_2,\ldots,a_n} := \sum{a_i}$. This complex is equipped with the differential $\partial := \sum_{k>0} V_k W_{-k}: \CA^{n}_{k} \to \CA^{n-1}_{k}$, that is
\[ \partial W^{a_1,a_2,\ldots,a_n}  = \sum_{k=1}^{n} (-1)^{k+1} V_{a_k}  W^{a_1,a_2,\ldots,\dot a_k \ldots, a_n} \]




\begin{equation}\label{Aresolution}
\cdots \stackrel{\partial}{\longrightarrow}  \CA^3_k \stackrel{\partial}{\longrightarrow} \CA^2_k   \stackrel{\partial}{\longrightarrow} \CA^1_k  \stackrel{\partial}{\longrightarrow} \CA^0_k = \CF_k \to 0 .
\end{equation}
One can directly check that $\partial^2 =0$ and that this complex is exact for $n\neq0$, where we have $\BC = \BC[V_0]/V_0 = H^0(\CA)$.
The Hilbert series for this complex is defined as
\begin{equation}\label{res}
HS_{\!\CF}(\CA,z,x) := \sum_{n,k} (-1)^{n} (\dim_{\CF} \CA^n_k) x^k z^n .
\end{equation}
In appendix \ref{hilbseriesapp}, we show the following formula
\begin{equation}\label{hilbpoly}
HS_{\!\CF}(\CA,z,x) = (z;x)_\infty.
\end{equation}

To relate this Koszul complex to the problem at hand, consider the map $\iota : \CA^1_{\bullet} \stackrel{\sim}{\to} \CF[w]$ by $\iota(W^{a}) := w^{a}$, then we find 
\begin{equation}
\iota \, \partial W^{a,b}=\Theta^{a,b},
\end{equation} i.e. $T_k = \iota( \Im \partial: \CA^2_k)$. 
Hence,
\begin{equation}
HS_{\!\CF}(T,z,x) = \dim \left( \Im \partial : \CA^2_k\to \CA^1_k \right).
\end{equation}
From \ref{hilbpoly}, we know that this is equal to 
\begin{equation}
\sum_{n=2}^{\infty} (-1)^n \dim \CA^n_k z^n = (z;x)_\infty -1+\frac{z}{1-x}
\end{equation}
In particular, we notice that 
\begin{equation}\label{rankthetares}
\rk{}_{\!\CF} T = HS_{\!\CF}(T,1,x) = x/(1-x).
\end{equation}
Next, we compute the dimension of the null module $\QZ^0 \subset \CF[w]$. $\QZ^0$ is given by imposing a single linear constraint $\ztr_\sigma = 0$, in degree $|\sigma|$, for each partition $\sigma$, and there are no relations among the $\ztr_\sigma$. Thus the dimension is
\begin{equation}\label{rankz0}
\dim_x \QZ^0 = \dim_x \CF[w] -  \sum_\sigma x^{|\sigma|} = \frac{1}{1-x}P(x) - P(x) =  \frac{x}{1-x}P(x).
\end{equation}
Comparing \ref{rankthetares} and \ref{rankz0} this, we see $\dim_x T := \rk{}_{\!\CF} T \cdot \dim_x \CF = \dim_x \QZ^0$ and we conclude that the inclusion $T_k \hookrightarrow \QZ^0_k$ is surjective.


\end{proof}
As a consequence of the above proof, we can show the null module is generated by a single element in every degree.

\begin{corollary}
There exist a set of elements $\{\Theta_k\}_{k=1}^{\infty}$ with $\Theta_k \in \QZ^{0}_k$, such that $\QZ^{0} = \CF\langle \Theta_k\rangle$.
\end{corollary}

\subsection{Traces and Jack Littlewood-Richardson Coefficients}

Now that we have computed the relationship between the traces of $\theta$ and $\beta$, we are finally ready to compute the value of these traces. It is at this point that we discover the connection with Jack LR coefficients that was promised in the introduction.

\begin{lemma}\label{xtrthetathm}
The $\xtr$-trace of a $\theta$ element in the $\hpsi$ basis is given by
\begin{equation}
\xtr_\sigma( \theta(\hpsi_\lambda^s, \hpsi_\nu^t)) =  \hat c_{\lambda\nu}^\sigma,
\end{equation}
where $\hat c_{\lambda\nu}^\sigma$ are the hatted Jack Littlewood-Richardson coefficients, given by
\begin{equation}
\hat j_\lambda \hat j_\nu = \sum_\sigma \hat c_{\lambda\nu}^\sigma \hat j_\sigma.
\end{equation}
\end{lemma}
\begin{proof}

First, we show that the $\xtr$ trace of $\theta(\hpsi_\lambda^s  , \hpsi_\nu^t)$ is independent of $t$. Using \ref{thetaexpression} we have
\begin{eqnarray}
\xtr_\sigma( \theta(\hpsi_\lambda^s  , \hpsi_\nu^t))  &=&  \xtr_\sigma\left( \Pi\hpsi_\lambda^s ( \pi_0 \CL( \hpsi_\nu^t))+ \beta(\Pi\hpsi_\lambda^s,  \hpsi_\nu^t) \right)\\
&=& \xtr_\sigma\left( \Pi\hpsi_\lambda^s \,\hat j_\nu\right).
\end{eqnarray}
where we have used $\xtr_\sigma \beta = 0$.  This is explicitly $t$-independent,
and due to the symmetry of $\theta$, it is therefore independent of both $s$ and $t$. Next, recall $ \CL \hat j_\lambda = w \hat q_\lambda = \sum_{t} [s] \htau_{\lambda}^s \hpsi_\lambda^s $, where $ \sum_{s'} [s']\htau_{\lambda}^{s'} = \hbar|\lambda|$. Following from this $s$-independence of the trace, we have
\begin{equation} \sum_{s} [s]\htau_{\lambda}^s \xtr_\sigma( \theta(\hpsi_\lambda^s, \hpsi_\nu^t))=( \sum_{s'} [s']\htau_{\lambda}^{s'}) \xtr_\sigma( \theta(\hpsi_\lambda^s, \hpsi_\nu^t)).  \end{equation}
This equality is thus rewritten as
\begin{equation}\label{lrcomp1}   \xtr_\sigma( \theta(w \hat q_\lambda , \hpsi_\nu^t))=\hbar|\lambda| \cdot  \xtr_\sigma( \theta(\hpsi_\lambda^s, \hpsi_\nu^t)).  \end{equation}
Using \ref{thetaexpression}, we work on the left hand side to get
\begin{eqnarray}\label{lrcomp2} 
\xtr_\sigma( \theta(w \hat q_\lambda , \hpsi_\nu^t)) &=&  \xtr_\sigma( \Pi(w \hat q_\lambda)  \hat j_\nu + \beta(\Pi(w \hat q_\lambda), \hpsi_\nu^t))\\
&= & \xtr_\sigma( \hat q_\lambda \hat j_\nu  ) .
\end{eqnarray}
Then, from \ref{XYjackaction}, we have
\begin{equation}\label{lrcomp3} 
\xtr_\sigma( \hat q_\lambda \hat j_\nu  ) =\sum_{\mu} \hat c^\sigma_{\nu\mu } \xtr_\mu( \hat q_\lambda  ) = \hat c^\sigma_{ \nu\lambda} |\lambda| \hbar
\end{equation}
Stringing the equalities \ref{lrcomp1}, \ref{lrcomp2} and \ref{lrcomp3} together we recover the result.
\end{proof}


\subsubsection{Trace formula}

We summarize the results of these computations of traces.

\begin{theorem}[Trace formula]\label{traceformula}
The full trace of the $\theta$ element is given by
\begin{equation}\label{traceformulaeqn}
\Tr(\theta(\hpsi_\lambda^s, \hpsi_\nu^t)) = (\{ \hat c_{\lambda,\nu}^{\sigma} \}, \{ \htau_{\lambda \*\nu}^v\} , \{0\}  ),
\end{equation}
or, equivalently, 
\begin{equation}
\Tr(\beta(\hpsi_\lambda^s, \hpsi_\nu^t)) = ( \{0\} , \{ \htau_{\lambda \*\nu}^v\}, \{ -\hat c_{\lambda,\nu}^{\sigma} \}),
\end{equation}
where we use the notation \ref{tauboxesdef}, $\htau_{\lambda \*\nu}^v = \Res_{u=[v]}T_{\lambda\*\nu}(u)$.
\end{theorem}
\begin{proof}
By \ref{twistedtraces}, $\ytr$-trace of $\theta$ is equal to the $\ytr$-trace of $\beta$, which was computed in \ref{ytraceofderiv}. The $\xtr$-trace of $\theta$ is given by Theorem \ref{xtrthetathm}. The $\ztr$-trace of $\theta$ is zero by \ref{twistedtraces}. The traces of $\beta$ is similarly determined by \ref{twistedtraces}. 
\end{proof}

Note that, similarly to $\beta(\hpsi_\lambda^s, \hpsi_\nu^t)$, the elements $f_{\lambda,\nu} := -(n\hbar)^{-1} \CL (\hj_{\lambda}\cdot \hj_{\nu})$ also satisfy $\ztr_{\sigma}(f_{\lambda,\nu} ) = -\hat c_{\lambda,\nu}^{\sigma}$. However the other traces of $f_{\lambda,\nu}$ do not have simple expressions.



\subsection{Main Theorem}
From the trace formula \ref{traceformula} of the previous section, we produce one of the main results of this work, which reveals a striking structure to the Jack LR coefficients.

\begin{theorem}\label{mainLRhtheorem}
For any partitions $\lambda, \nu \neq \emptyset$, the hatted Jack Littlewood-Richardson coefficients $\hat c_{\mu\nu}^{\gamma}$ satisfy the following equality of rational functions of $u$,
\begin{equation}
\sum_{\gamma \partition n : \mu,\nu \subseteq \gamma} \hat c_{\mu\nu}^{\gamma } \left( \sum_{s \in \gamma / (\mu \cup \nu)} \frac{1}{u-[s]} \right)  = T_{\mu\* \nu}(u) - 1.
\end{equation}

\end{theorem}
\begin{proof}
We apply the fundamental cokernel relation $R(u)$ (\ref{cokerru}) to the total trace of a $\theta$ element (\ref{traceformula}) (or, equivalently, a $\beta$ element), 
\begin{equation}
R(u) \Tr\,\theta(\hpsi_\mu^s, \hpsi_\nu^t) = 0.
\end{equation}
for any choice of $s,t$. This is 
\begin{equation}
\left(T_{\mu\* \nu}(u) - 1\right) -\sum_\gamma \hat c_{\mu\nu}^{\gamma} \left( \sum_{s \in \gamma} \frac{1}{u-[s]} \right) + \left(0\right) = 0.
\end{equation}
Since $\mu, \nu \neq \emptyset$ we have $\sum_\gamma \hat c_{\mu\nu}^{\gamma } = 0$, and $\hat c_{\mu\nu}^\gamma =0$ unless $\mu,\nu \subseteq \gamma$, 
\begin{equation}
\sum_{\gamma \partition n : \mu,\nu \subseteq \gamma} \hat c_{\mu\nu}^{\gamma } \left( \sum_{s \in  (\mu \cup \nu)} \frac{1}{u-[s]} \right) = \left(\sum_{\gamma \partition n : \mu,\nu \subseteq \gamma} \hat c_{\mu\nu}^{\gamma } \right)\left( \sum_{s \in  (\mu \cup \nu)} \frac{1}{u-[s]} \right)  = 0.
\end{equation}
and so the inner sum reduces to
\begin{equation}
\sum_{\gamma \partition n : \mu,\nu \subseteq \gamma} \hat c_{\mu\nu}^{\gamma } \left( \sum_{s \in \gamma / (\mu \cup \nu)} \frac{1}{u-[s]} \right)  = T_{\mu\* \nu}(u) - 1.
\end{equation}
\end{proof}
This result is a generalization the relationship between Jack LR coefficients and residues of spectral resolvent factors as described by Kerov (\ref{kerovLRrule}), which is clearly equivalent to the $\nu=1$ case,
\begin{equation}
\sum_{s\in \addset_{\mu}} \hat c_{\mu,1 }^{\mu+s } \left(\frac{1}{u-[s]} \right)  = T_{\mu}(u) - 1.
\end{equation}

It can be checked that Theorem \ref{mainLRhtheorem} determines all Jack LR coefficients $c_{\mu\nu}^{\gamma}$ for $|\gamma| < 7$. However, the number of poles on the right-hand side grows like $n \log n$, whereas the number of LR coefficients on the left grows as the partition number $p(n)$, so in general this system of equations for the Jack LR coefficients is undetermined.
\begin{example}\label{mainthmexample}
Let $\mu = 1^2,\nu = 2$. We find 
\begin{equation} T_{\mu\*\nu} = T_{2^2} =\frac{(u-[(0,0)])(u-[(2,2)])}{(u-[(2,0)])(u-[(0,2)])}. \end{equation}
Expanding this in poles we find
\begin{equation} T_{\mu\*\nu} = \frac{[2,0][0,-2]}{[2,-2]}\frac{1}{u-[2,0]} + \frac{[0,2][-2,0]}{[-2,2]}\frac{1}{u-[0,2]}. \end{equation}
On the other hand, by theorem \ref{mainLRhtheorem}, we know this must be equal to
\begin{equation} \hat c_{1^2,2}^{1^4} \left( \frac{1}{u-[3,0]} + \ldots \right) + \hat c_{1^2,2}^{1^22} \left( \frac{1}{u-[2,0]} + \ldots \right) +\hat c_{1^2,2}^{2^2} \left( \frac{1}{u-[1,1]} + \ldots \right)  \end{equation} 
\begin{equation}+ \hat c_{1^2,2}^{13} \left( \frac{1}{u-[0,2]} + \ldots \right) + \hat c_{1^2,2}^{4} \left( \frac{1}{u-[0,3]} + \ldots \right) . \end{equation} 
We can read off the only non-zero (up to transposition) hatted LR coefficient
\begin{equation} \hat c_{1^2,2}^{1^22}  = \frac{[2,0][0,-2]}{[2,-2]}, \end{equation}
which gives the regular Jack LR coefficient
\begin{equation}  c_{1^2,2}^{1^22}  =\hat c_{1^2,2}^{1^22} \frac{\jacktop_{1^2} \jacktop_{2}}{\jacktop_{1^22} } =  \frac{[2,0][0,-2]}{[2,-2]} \frac{[1,0] \cdot [0,1]}{[2,0][1,0][0,1] }  = \frac{[0,-2]}{[2,-2]} = \frac{-\vareps_2}{\vareps_1-\vareps_2}. \end{equation}
This can be checked to agree with the Pieri formula.
\end{example}

\section{The $\SHc$ algebra}\label{shcsection}

We now lay the groundwork for a re-writing of the main result of the previous section \ref{mainLRhtheorem} in a radically different language, in the hope of elucidating its significance.

In \cite{Schiffmann:2012tm},  Schiffman-Vasserot introduce an algebra denoted $\SHc$, described as a centrally extended spherical degenerate double affine Hecke algebra (DAHA). This algebra provides a systematic method to analyse the instanton partition functions of N = 2 supersymmetric gauge theories.

\subsection{Definitions}
Here we won't describe the full $\SHc$ algebra, but rather a special presentation of it relevant for our purposes.
The rank $r=1$ (with central charge $a=0$) \emph{Holomorphic field presentation} was described in \cite{Bourgine:2015tv}, and we use the notation from \cite{Bourgine:2019tb}. This presentation consists of currents $X^\pm, \CY^\pm$ acting on the Fock module $\CF$. 

The operator $\CY(z)$ (c.f. \cite{Nekrasov:2016} eq. (125)), and its inverse $\CY^{-1}$, act diagonally on the Jack basis states $\hat j_\lambda = |\lambda\rangle$,
\begin{equation} \CY(z) |\lambda\rangle = \CY_{\lambda}(z)|\lambda\rangle,\end{equation} 
with eigenvalues given by the familiar spectral functions $T_\lambda$ from \ref{spectralfactors1},
\begin{equation} \CY_{\lambda}(z) \coloneqq z \, T_{\lambda}(z)^{-1}= \frac{\prod_{s \in \addset_{\lambda}}(z-[s]) } {\prod_{t \in \remsetp_\lambda} (z-[t]) }. \end{equation}
We note that this is equal to (the inverse of) the Nazarov-Skylanin transfer operator (\ref{NStransferop}),
\begin{equation} \CY(z) = \CT(z)^{-1}. \end{equation}
The other `box creation' operators $X^\pm(z)$ are defined (in the notation of \cite{Bourgine:2019tb} equation 2.18, where $\phi_x$ denotes the content of a box $x$) by their action on the Jack basis
\begin{equation}\label{bourgxplus}
X^+(z)  |\lambda\rangle  := \sum_{x \in \addset_{\lambda}}\frac{1}{z-\phi_x}  \left( \Res_{u=\phi_x} \CY_{\lambda}(u)^{-1} \right)  |\lambda+x\rangle ,
\end{equation}
\begin{equation}\label{bourgxminus}
X^-(z)  |\lambda\rangle   :=  \sum_{x \in \remset_\lambda}\frac{1}{z-\phi_x}  \left( \Res_{u=\phi_x} \CY_{\lambda}(u+\bareps) \right)  |\lambda-x\rangle .
\end{equation}
These operators satisfy the commutation relation
\begin{equation}
[ X^+(z) , X^-(w) ] = \frac{\Psi(z) - \Psi(w)}{z-w},
\end{equation}
where $\Psi(z)$ is the so called \emph{chiral ring generating operator} (c.f. \cite{Bourgine:2019tb} (2.13))
\begin{equation}
\Psi(z) =\CY(z+\bareps) \CY(z)^{-1}.
\end{equation}



\subsection{New construction}

We show that the algebra $\{ \CY^{\pm}, X^\pm\}$ can be easily constructed out of the action of $\CL$.
\begin{proposition}
The following operators  $\CF \to \CF(z)$
\begin{equation}
X^+(z) = A\frac{1}{z-\CL} \pi_0, \qquad X^{-}(z) = \pi_0 \frac{1}{z-\CL}A^\dag
\end{equation}
\begin{equation}
\CY(z) =  (\hbar\CN )^{-1} A \frac{1}{z-\bareps-\CL} A^\dag,\qquad 
\CY(z)^{-1} = \pi_0 \frac{1}{z-\CL},
\end{equation}
where $A = \pi_0 \CL w,\, A^\dag = w^{-1} \CL\pi_0$,
reproduce the rank $r=1$ Bourgine-Matsuo-Zhang Holomorphic field realization of $\SHc$ in terms of the action of the Nazarov-Sklyanin quantum Lax operator $\CL$.
\end{proposition}
\begin{proof}
\begin{eqnarray}
X^+(z) j_\lambda &=& (\pi_0 \CL w)\frac{1}{z-\CL}  \sum_{s \in \addset_{\lambda}} \tau_\lambda^s \psi_{\lambda}^s \\
&=&\sum_{s \in \addset_{\lambda}}\frac{1}{z-[s]}  \tau_\lambda^s  (\pi_0 \CL w) \psi_{\lambda}^s\\
& =& \sum_{s \in \addset_{\lambda}}\frac{1}{z-[s]}  \tau_\lambda^s  j_{\lambda+s},
\end{eqnarray}
which agrees with \ref{bourgxplus}, using the formulae:

\begin{equation} \tau_\lambda^s := \Res_{u=[s]} \CY_{\lambda}(u)^{-1}, \qquad s \in \addset_{\lambda}, \end{equation}
\begin{equation} \ttau_\lambda^t := \Res_{u=[t]} \CY_{\lambda}(u), \qquad t \in \remsetp_\lambda. \end{equation}
Secondly, we have
\begin{eqnarray}
X^-(z) j_\lambda &= &\pi_0 \frac{1}{z-\CL}(w^{-1}\CL)  j_\lambda\\
& =& \pi_0 \frac{1}{z-\CL} q_\lambda\\
& =& \pi_0 \frac{1}{z-\CL} \sum_{t \in \remset_\lambda} \ttau_{\lambda}^{t+(1,1)} \psi_{\lambda-t}^{t}\\
& =& \pi_0  \sum_{t \in \remset_\lambda}\frac{1}{z-[t]}  \ttau_{\lambda}^{t+(1,1)} \psi_{\lambda-t}^{t}\\ 
&=& \sum_{t \in \remset_\lambda}\frac{1}{z-[t]}  \ttau_{\lambda}^{t+(1,1)} j_{\lambda-t},
\end{eqnarray}
which agrees with \ref{bourgxminus}.
\end{proof}

\subsubsection{The Gaiotto State}
Consider the following so called Gaiotto state in $\CF$,
\begin{equation}
G=e^{V_1/\hbar}= \sum_{n=0}^{\infty} V_1^n/|V_1^n|^2 = \sum_\lambda  j_\lambda / | j_\lambda |^2
\end{equation}
The Gaiotto state is characterized as a Whittaker vector, which in the holomorphic presentation is the statement (\cite{Bourgine:2015tv} 3.12/13 with $X^+ = D_{+1}$)
\begin{equation}\label{whittakerG}
X^{-}(u) G = \CY(u)^{-1} G, \qquad X^{+}(u) G = P_u^- \CY(u+\bareps) G.
\end{equation}
The $j_\lambda$ component of these two equations reproduce the pole expansions
\begin{equation}
\sum_{s} \frac{\tau_\lambda^s}{u-[s]} =  u^{-1} T_\lambda(u), \qquad \sum_{s} \frac{\kappa_\lambda^s}{u-[s]} =  (u+\bareps)\, T_\lambda(u+\bareps)^{-1}
\end{equation}
In other words, the Whittaker condition for $G$ is equivalent to the Kerov identities \ref{kerovcoefficient}.

\subsubsection{The Half-Boson}
The Holomorphic presentation can be alternately described (Bourgine \cite{Bourgine:2015tv} 1809 2.20) in terms of the following \emph{half-boson}\footnote{The name refers to that face that only one half of the Fourier modes of a usual free boson are present.} $\Phi(z)$,
\begin{equation}
\Phi(z) = \log(z) \Phi_0 - \sum_{n=1}^{\infty} \frac{1}{nz^n} \Phi_n,
\end{equation}
which acts diagonally on the Jack states $j_\lambda = |\lambda\rangle$ as
\begin{equation}
\Phi_n |\lambda \rangle \coloneqq \left( \sum_{s \in \lambda} [s]^n \right) |\lambda \rangle, \qquad \partial \Phi(z) |\lambda \rangle = \left( \sum_{s \in \lambda} \frac{1}{z-[s]} \right) |\lambda \rangle.
\end{equation}
The eigenvalues of the half-boson current $\partial\Phi$ are recognisable in our earlier formulae for the cokernel relation (\ref{cokerru}), 
\begin{equation}
 \partial \Phi_{\lambda}(z) \coloneqq \sum_{s \in \lambda} \frac{1}{z-[s]},
\end{equation}
which was the original motivation for the re-interpretation of the earlier work in this paper in the language of the $\SHc$. 
The half-boson current has the following commutation relation with the $\SHc$ operators
\begin{equation}\label{halfbosonXcomm}
[ \partial\Phi(x), X^{\pm}(w) ] = \pm \frac{X^{\pm}(w)}{z-w}.
\end{equation}

\subsubsection{Flavor vertex}

The \emph{flavor vertex operator} $\CU$ (c.f. \cite{Bourgine:2016ww} 3.18  D.2, \cite{Bourgine:2019tb} 2.14) is given by
\begin{equation} \quad \CU|\lambda\rangle =\left(  \prod_{s\in \lambda^\times}[s] \right)|\lambda \rangle, \quad \CU |\emptyset \rangle = 0. \end{equation}
We notice that the eigenvalue of this operator is the familiar constant $\varpi_\lambda$ (c.f. \ref{jacktopdef}). Equivalently, we find its action as implementing the switch between the hatted and unhatted Jacks, i.e. $\CU \cdot \hat j_\lambda = j_\lambda$. The action on the Gaiotto state is given simply by:
\begin{lemma}\label{Hdef}
We have $ \CU |G\rangle = |H\rangle $, where $|H\rangle \in \CF$ is given by
\begin{equation}\label{Hvecdef} |H\rangle \coloneqq \sum_{n=1}^{\infty} V_n/|V_n|^2 = \CL^{-1}\sum_{n\geq1} w^n = \sum_{\lambda \neq \emptyset} \hat j_\lambda / |\hat j_\lambda |^2 .
\end{equation}
\end{lemma}

\subsection{Generalized Whittaker condition}

In this section, we are going to show that the main result (Theorem \ref{mainLRhtheorem}) of this paper can be understood as a generalization of the Whittaker condition for the Gaiotto state $G$ (\ref{whittakerG}).

\begin{corollary}\label{whittakergencoll}
The first Whittaker condition for $|G\rangle$ (\ref{whittakerG}) is equivalent to the following condition for $|H\rangle := \CU |G\rangle$ (c.f. \ref{Hdef})
\begin{equation}\label{Hwhittaker}
X^{-}(z) |H\rangle = P_z^-\left( \CT(z)\right) |H\rangle +z^{-1},\end{equation}
where, $\CT(z) = z\CY(z)^{-1}$, and $P_z^-$ denotes projection onto only negative powers of $z$.
\end{corollary}
Corollary \ref{whittakergencoll} can be proved easily in the language of  \cite{Bourgine:2016ww}, which we omit here. 
We note that by expanding out in components, we find that the $|\lambda\rangle$ component of this relation gives the following familiar pole expansion
\begin{equation} \sum_{s} \frac{\htau_\lambda^s}{z-[s]} =  T_\lambda(z)-1. \end{equation}






We note the following expansion 
\begin{equation}
X^\pm(z) = V_1^{\pm}z^{-1}  + O(z^{-2}) 
\end{equation}
where $V_1^+ := V_1 \cdot $ and $V_1^- := V_1^\dag  = \hbar \partial_{V_1}$
and hence from \ref{halfbosonXcomm} we have
\begin{equation}
[ \partial\Phi(x), V_1^{\pm} ] = \pm X^{\pm}(z).
\end{equation}
Motivated by this, we define the generalized box creation operators:
\begin{equation}
 X_\lambda^{\pm}(z) := \pm [ \partial\Phi(x), \hat j_\lambda^{\pm} ].
\end{equation}

We now show the extension of \ref{Hwhittaker} to these generalized operators, 
\begin{proposition}[Generalized Whittaker Condition]
The main theorem \ref{mainLRhtheorem} is equivalent to the following condition for $H$,
\begin{equation}\label{genwhitcond}
X_\lambda^-(z) |H\rangle =  P_z^-\left(\prod_{s\in \lambda} \mathcal{T}(z-[s]) \right)|H\rangle + z^{-1}.
\end{equation}
\end{proposition}
\begin{proof}
The $|\hat \mu\rangle$ component of the left hand side is given by
\begin{equation}
P_z^-\left(\prod_{s\in \mu} \mathcal{T}(z-[s]) \right)|\hat \mu\rangle/ |\hat\mu|^2 = P_z^-\left( T_{\mu\*\lambda}(z)\right) |\hat \mu\rangle/ |\hat\mu|^2 = \left( T_{\mu\*\lambda}(z) - 1\right) |\hat \mu\rangle/ |\hat\mu|^2.
\end{equation}
On the other hand the right hand side is
\begin{eqnarray}
-[ \partial\Phi(z), \hat j_\lambda^\dag ] | H \rangle &=& \sum_{\sigma} [  \hat j_\lambda^\dag,\partial\Phi(z) ] | \hat \sigma \rangle/ |\hat\sigma|^2 \\
&=& \sum_{\sigma} \sum_\mu \hat c_{\lambda\mu}^{\sigma} \left(\partial\Phi_\sigma(z)-\partial\Phi_\mu(z)\right) | \hat \mu \rangle/ |\hat\mu|^2,
\end{eqnarray}
where we have used
\begin{equation}
\hat j_\lambda^\dag \frac{| \hat \sigma \rangle}{|\hat \sigma|^2 } = \sum_{\mu} \hat c_{\lambda \mu}^{\sigma}  \frac{| \hat\mu \rangle}{|\hat\mu|^2 }. \end{equation}
Equating the $|\hat\mu\rangle$ components from both sides we recover the formula \ref{mainLRhtheorem}.
\end{proof}


Finally, we write the relation \ref{resolvewident} in the language of the holomorphic representation.
\begin{lemma}
\begin{equation}\label{liftedwhittaker}
( w^{-1} -1) \CL \partial\Phi(z)| H \rangle= \frac{\CL}{z-\CL}  | H \rangle+z^{-1}  .
\end{equation}
\end{lemma}
\begin{proof}
We start with the formula \ref{resolvewident}
\begin{equation}
\frac{1}{z-\CL} w^n = \sum_{\gamma \partition (n+1)} \left( \sum_{t\in \gamma} \frac{1}{z-[t]} \right) \frac{\hat q_\gamma}{|\hat j_\gamma|^2} - \sum_{\lambda \partition n} \left( \sum_{s\in \lambda} \frac{1}{z-[s]} \right) \frac{w\hat q_\lambda}{|\hat j_\lambda|^2}.
\end{equation}
which we sum over $n$ to and use the formulae \ref{Hvecdef} to write this as
\begin{equation}
\frac{\CL}{z-\CL}  | H \rangle  =  w^{-1} \CL \partial\Phi(z)(| H \rangle- V_1/\hbar)- \CL \partial\Phi(u)| H \rangle .
\end{equation}
Upon rearranging we recover the result.
\end{proof}

We see that equation \ref{liftedwhittaker} is a lifting of the Whittaker condition \ref{Hwhittaker}, which is recovered under $\pi_0$.

\subsubsection{Basic Evaluation Map}
We can now re-express our Main Theorem \ref{mainLRhtheorem} in terms of objects of the $\SHc$:
\begin{proposition}\label{thm:basicevaluationmap}The `basic' evaluation map $\dtrace : \CF \to \BCe(u)$, defined by
\begin{equation} \Delta: \zeta \mapsto \langle \zeta | \,\partial \Phi(u)\,\CU\, |  G \rangle, \end{equation}
i.e. on the hatted Jack basis it is given by
\begin{equation}\label{deltadfn2}
\dtrace( \hat j_\lambda ) \coloneqq \sum_{b \in \lambda} \frac{1}{u-[b]},
\end{equation}
satisfies the following multiplicative property
\begin{equation}\label{basicmaptheorem}
\dtrace( \hat j_\mu\cdot\hat j_\nu ) = T_{\mu\* \nu}(u) - 1.
\end{equation}
\end{proposition}
\begin{proof}
We compute
\begin{eqnarray}
\Delta(  \hj_\mu\cdot\hj_\nu ) &=& \langle \hat j_\mu\cdot\hat j_\nu | \,\partial \Phi(u)\,\CU\, |  G \rangle \\
 &=& \langle \hj_\nu | \hj_\mu^\dag \partial \Phi(u) | H \rangle \\
  &=& \langle \hj_\nu | X^-_\mu(u) | H \rangle + \langle \hj_\nu |\partial \Phi(u) \hj_\mu^\dag  | H \rangle \\
  &=& \langle \hj_\nu | P_z^-\left(\prod_{s\in \lambda} \mathcal{T}(z-[s]) \right)|H\rangle + z^{-1}\langle \hj_\nu |1\rangle + \langle \hj_\nu |\partial \Phi(u) | 1 \rangle \\
 &=& \left( T_{\mu\* \nu}(u) - 1 \right) \langle \hj_\nu | H\rangle, 
\end{eqnarray}
where we have used the generalized Whittaker condition \ref{genwhitcond} and $\hat j_\mu^\dag | H \rangle = 1$.
\end{proof}
We call this the \emph{basic} evaluation because for $\zeta \in \CF$, we have $\ytr_u(\zeta) = \dtrace( j_1 \cdot \zeta )$. In terms of the usual normalization on Jack functions, we have 
\begin{equation}
\dtrace(j_{\lambda}) := \varpi_\lambda  \sum_{b \in \lambda} \frac{1}{u-[b]}.
\end{equation}

\begin{example} We re-express the previous example \ref{mainthmexample} in terms of this new operator.
Let $\mu = 1^2,\nu = 2$. We have
\begin{equation*}
\dtrace(j_{1^2}) = [1,0] \left( \frac{1}{u-[1,0]} +\frac{1}{u-[0,0]} \right), \quad \dtrace(j_{2}) = [0,1] \left( \frac{1}{u-[0,1]} +\frac{1}{u-[0,0]} \right).
\end{equation*}
We know (e.g. from Stanley's Pieri rule \ref{stanleypieri}) that
\[ j_{1^2} \cdot j_{2} = \frac{[0,-2]}{[2,-2]} j_{1^22} + \frac{[2,0]}{[2,-2]} j_{3,1}. \]
We then have
\begin{eqnarray*}
\dtrace( j_{1^2} \cdot j_{2} )&=& \frac{[0,-2]}{[2,-2]}\Delta( j_{1^22}) + \frac{[2,0]}{[2,-2]}\Delta( j_{3,1}) \\ 
&=& \frac{[0,-2]}{[2,-2]}[2,0][1,0][0,1]\left( \frac{1}{u-[2,0]} +\frac{1}{u-[1,0]}+ \frac{1}{u-[0,1]} +\frac{1}{u-[0,0]}\right) \\
&& + \frac{[2,0]}{[2,-2]}[0,2][0,1][1,0]\left( \frac{1}{u-[0,2]}  +\frac{1}{u-[1,0]}+ \frac{1}{u-[0,1]} +\frac{1}{u-[0,0]}\right) \\
&=&  \frac{[2,0]}{[2,-2]}[0,2][0,1][1,0]\left( \frac{1}{u-[0,2]} - \frac{1}{u-[2,0]} \right) .
\end{eqnarray*} 
On the other hand, we have 
\[ T_{1^2 \*2}(u) = T_{2^2}(u) =\frac{(u-[0,0])(u-[2,2])}{(u-[2,0])(u-[0,2])}. \]
Expanding this in poles we find
\[ T_{1^2 \*2}(u) = \frac{[2,0][0,-2]}{[2,-2]}\frac{1}{u-[2,0]} + \frac{[0,2][-2,0]}{[-2,2]}\frac{1}{u-[0,2]}  + 1. \]
Thus we see that
\[ \dtrace( j_{1^2} \cdot j_{2} )  = \jacktop_{1^2} \jacktop_{2} \left(T_{1^2 \*2}(u) - 1 \right).\]
\end{example}

\subsubsection{Kernel}

The map $\Delta$ is not injective. The following result follows from the definition of $\Delta$ (\ref{deltadfn2}) by direct calculation.
\def\hj{\hat j}
\begin{proposition}
For any collection of partitions $\Lambda = \{ \lambda_i \}_{i=1}^{N}$, where $N$ is even and where $\lambda_i$ and $\lambda_{i+1 \mod N}$ differ by moving a single box, and where every box appears an even number of times (in total), the element
\begin{equation}
k_\Lambda = \sum_{i}^{N} (-1)^i \hj_{\lambda_i}
\end{equation}
is in the kernel of $\Delta$.
\end{proposition}

We call such a collection of partitions an $N$-\emph{cycle}, motivated by the following examples.

\begin{example}
For example, in degree 8, the 4-cycle of partitions (with markers to track the moving boxes)
\begin{equation}
\ytableausetup{boxsize=1.0em}
\begin{ytableau}
\diamondsuit  \\ 
\, & \heartsuit   \\ 
 &    \, \\ 
\, & \,&  \, 
\end{ytableau} 
\to 
\begin{ytableau}
\diamondsuit \\ 
\,  \\ 
 &   \,& \heartsuit  \\ 
\, & \,&  \, 
\end{ytableau} 
\to
\begin{ytableau}
\,  \\ 
 &    \, & \heartsuit   \\ 
\, & \,&  \,  & \diamondsuit  
\end{ytableau} 
\to 
\begin{ytableau}
\, & \heartsuit   \\ 
 &    \, \\ 
\, & \,&  \,  & \diamondsuit  
\end{ytableau}
\to
\begin{ytableau}
\diamondsuit  \\ 
\, & \heartsuit   \\ 
 &    \, \\ 
\, & \,&  \, 
\end{ytableau} 
\end{equation}
gives the kernel element,
\begin{equation}
k_{\Lambda} := \hj_{1,3,4}  - \hj_{2^2,4} + \hj_{1,2^2,3} - \hj_{1^2,3^2} \in \ker\Delta_8.
\end{equation}
\end{example}

\begin{example}
In degree 7, the 6-cycle of partitions (with markers to track the moving boxes)
\begin{equation}
\ytableausetup{boxsize=1.0em}
\begin{ytableau}
\heartsuit \\ 
\diamondsuit  \\ 
  \, \\ 
 & \clubsuit  \\ 
\, &   \, 
\end{ytableau} 
\to 
\begin{ytableau}
\diamondsuit  \\ 
\, & \heartsuit   \\ 
 & \clubsuit   \\ 
\, &   \, 
\end{ytableau} 
\to 
\begin{ytableau}
\, & \heartsuit \\ 
 & \clubsuit \\ 
\, &   \,  & \diamondsuit
\end{ytableau} 
\to 
\begin{ytableau}
  \, \\ 
 \, & \clubsuit \\ 
\, &   \,  & \diamondsuit & \heartsuit
\end{ytableau} 
\to 
\begin{ytableau}
\clubsuit \\
  \, \\ 
 \,  \, \\ 
\, &   \,  & \diamondsuit & \heartsuit
\end{ytableau}
\to 
\begin{ytableau}
\heartsuit \\
\clubsuit \\
  \, \\ 
 \,  \, \\ 
\, &   \,  & \diamondsuit   \,
\end{ytableau}
\to 
\begin{ytableau}
\heartsuit \\
\clubsuit \\
  \, \\ 
 \, & \diamondsuit \\ 
\, &     \,
\end{ytableau} 
\end{equation}
gives the kernel element,
\begin{equation}
k_{\Lambda} := \hj_{2^21^3}  - \hj_{2^31} + \hj_{32^2} - \hj_{421}+ \hj_{41^3} - \hj_{31^4} \in \ker\Delta_7.
\end{equation}
Note that $\ker\Delta_7 = \mathrm{Span}\{ k_\Lambda, k_{\Lambda'} \}$ is the first non-empty kernel. 
\end{example}

\section{Structure of twists}\label{twistssection}

In this final section, we outline some conjectures that hope to further illuminate the structure of the algebra of Lax eigenfunctions, in particular we seek to provide a natural explanation of the trace twist property (\ref{betathetatwistedtraces}), that is,
\begin{equation*}
\Tr_{n}(\beta(\zeta, \xi)) = \rho_* \circ \Tr_{n-1} (\theta(\zeta,\xi)),
\end{equation*}
where as before $\rho_* : (\{x_\gamma\},\{y^s\},\{0\}) \mapsto (\{0\},\{y^s\},\{-x_\gamma\})$.

\subsection{The $\rho$ operator}

We begin by introducing a peculiar operator.

\begin{definition}
Consider the following operator, defined for $s\in \addset_{\lambda}$,
\begin{equation}
\rho_{\lambda}^s : \QZ_\lambda^0 \to \bigoplus_{ \gamma \nearrow (\lambda+s)  }\QX_\gamma^0 \,\subset \CH_{|\lambda|+1}, \qquad \rho_{\lambda}^s : \hpsi_\lambda^t - \hpsi_\lambda^s \mapsto \hpsi_{\lambda+s}^t - \hpsi_{\lambda+t}^s.
\end{equation}
The operator $\rho_{\lambda}^s$ is expressed in the basis of $\QZ_\lambda^0$ given by $\{ \hpsi_\lambda^t - \hpsi_\lambda^s \}_{t \in \addset_{\lambda}, t \neq s}$. 

\end{definition}

The next result follows directly from the definition.

\begin{corollary}\label{rhoprops1}
For $\zeta \in \QZ_\lambda^0$, i.e. $\zeta \in \QZ_\lambda$ and $\ztr_\lambda( \zeta)= 0$, the map $\rho_\lambda^s$ has the properties
\begin{equation}\label{rhotraces}
 \ytr_u ( \rho_{\lambda}^s \zeta ) =  \ytr_u (\zeta) ,\qquad  \ztr_\gamma ( \rho_{\lambda}^s \zeta ) = - \xtr_\gamma(\zeta), \qquad \xtr_\gamma (\rho_{\lambda}^s \zeta )= 0,
\end{equation}
for all $ \gamma \partition |\lambda|+1$. In other words, $\rho_\lambda^s$ induces the trace twist (\ref{tracetwistdef}) 
\begin{equation}\label{rhointertwine}\Tr_{n+1}(\rho_\lambda^s \zeta ) = \rho_* \circ \Tr_n (\zeta).\end{equation}
\end{corollary}

Using corollary \ref{rhoprops1}, the following result gives a direct explanation of the first example (\ref{twistedtraces1}) of the twisted trace property .
\begin{proposition}\label{beta1form}
\begin{equation}
\beta(\hpsi_1^{v},\hpsi_\lambda^s) = \rho_\lambda^s( \hat j_\lambda).
\end{equation}
\end{proposition}


\begin{proof} Starting from the lifted Pieri formula \ref{laxpieri}, we find
\begin{eqnarray}
(\CL-[v+s]) \left( \psi^v_{1}  \cdot \psi_\lambda^s\right) &=& \sum_{u \in \addset_{\lambda +s}} [v]\frac{[s-u+\bar v][s-u+v]}{[s-u][s-u+v+\bar v]}\, \tau_{\lambda+s}^{u}\psi^u_{\lambda+s} \\
 &&+ \sum_{t\in \addset_{\lambda}, t\neq s}[-v] \tau_{\lambda}^{t}  \psi^s_{\lambda+t}
\end{eqnarray}
Note that if $u = s+v$ or $u=s+\bar v$ are minima of $\lambda+s$, then their coefficients in the first sum vanish. Then, using \ref{tauabTidenti}, we have
\begin{eqnarray}
(\CL-[v+s])( \hpsi^v_{1}  \cdot \hpsi_\lambda^s ) &=&  \sum_{t\in \addset_{\lambda}, t\neq s}\htau_{\lambda}^{t}\left( \hpsi^t_{\lambda+s}- \hpsi^s_{\lambda+t}\right)\\
 &=&  \sum_{t\in \addset_{\lambda}, t\neq s}\htau_{\lambda}^{t}\, \rho_\lambda^s\left( \hpsi^t_{\lambda}- \hpsi^s_{\lambda}\right)\\
  &=& \rho_\lambda^s \sum_{t\in \addset_{\lambda}, t\neq s}\htau_{\lambda}^{t}\, \left( \hpsi^t_{\lambda}- \hpsi^s_{\lambda}\right)\\
    &=& \rho_\lambda^s \sum_{t\in \addset_{\lambda}}\htau_{\lambda}^{t}\,  \hpsi^t_{\lambda}.
\end{eqnarray}
Since $\sum_{ t\in \addset_{\lambda}} \htau_{\lambda}^{t} = 0$ implies that $\sum_{ t\in \addset_{\lambda}; t\neq s} \htau_{\lambda}^{t}  = -\htau_{\lambda}^{s}$.
\end{proof}

\subsubsection{Generalization}

Next, we look at extending the relation \ref{beta1form} to all higher degrees to understand the relationship between the total traces of $\beta(\hpsi_\lambda^s,\hpsi_\nu^t) $ and $  \theta(\hpsi_\lambda^s,\hpsi_\nu^t)$ for all $\lambda, \nu$. 

\begin{definition}\label{generalrho}For $\zeta = \sum_\lambda \zeta_\lambda \in \CZ^0$, i.e. $\zeta_\lambda := P_{\CZ_\lambda}(\zeta)$ and $\ztr_\lambda(\zeta_\lambda)=0, \forall \lambda$, we define
\begin{equation}
 \rho(\xi)\zeta := \sum_{\lambda,s} \xi_\lambda^s \rho_\lambda^s  \zeta_\lambda , \text{ where }\xi = \sum_{\lambda,s}\xi_\lambda^s \hpsi_\lambda^s.
 \end{equation}
 \end{definition}
Note that we have $\rho(\hpsi_\lambda^s) = \rho_\lambda^s$. With this, we can rewrite \ref{beta1form} more suggestively as
\begin{equation} \beta(\hpsi_1^{v} ,\hpsi_\lambda^s) = \rho(\hpsi_\lambda^s)  \theta(\hpsi_1^{v} ,\hpsi_\lambda^s). \end{equation}
\begin{lemma}
\begin{itemize}{\,}
\item
For $\xi,\zeta \in \QZ_\lambda^0\subset \QZ_\lambda$, we have $
\rho(\xi)\zeta \in \ker \Tr$.
\item For $\zeta \in \QZ_\lambda^0$, we have $\rho( \zeta) \zeta = 0$, and hence $\rho(\xi)\zeta = - \rho(\zeta)\xi$.
\end{itemize}
\end{lemma}
\begin{proof}
We check
\begin{eqnarray}
\rho_{\lambda}^a-\rho_{\lambda}^b(\hpsi_\lambda^c - \hpsi_\lambda^a )&=& \hpsi_{\lambda+a}^c - \hpsi_{\lambda+c}^a -\rho_{\lambda}^b\left(  \hpsi_\lambda^c -\hpsi_\lambda^b+ \hpsi_\lambda^b  - \hpsi_\lambda^a \right) \\
& =&  \hpsi_{\lambda+a}^c -\hpsi_{\lambda+b}^c+\hpsi_{\lambda+c}^b - \hpsi_{\lambda+a}^b  + \hpsi_{\lambda+b}^a - \hpsi_{\lambda+c}^a = \Gamma_\lambda^{a,b,c}\end{eqnarray}
so
\begin{equation} \rho_{\lambda}^a-\rho_{\lambda}^b : \hpsi_\lambda^c - \hpsi_\lambda^d \mapsto \Gamma_\lambda^{a,b,c}-\Gamma_\lambda^{a,b,d} \in \ker W_{|\lambda|+1} \end{equation}
Thus we have (where $u = \sum_{a\neq b} u_{a} ( \hpsi_\lambda^a - \hpsi_\lambda^b )$)
\begin{eqnarray} \rho(u)u &=& \left( \sum_{a\neq b} u_{a}\rho( \hpsi_\lambda^a - \hpsi_\lambda^b )  \right) \left( \sum_{c\neq b} u_{c}( \hpsi_\lambda^c - \hpsi_\lambda^b )   \right) \\
&=& \sum_{a\neq b} \sum_{c\neq b}  u_{a}u_{c}(\Gamma_\lambda^{a,b,c}-\Gamma_\lambda^{a,b,b}) \\
&=& \sum_{a\neq b} \sum_{c\neq b}  u_{a}u_{c}(\Gamma_\lambda^{a,b,c}) \\
&=& 0.
\end{eqnarray}
As this is a symmetric sum over an asymmetric symbol, it vanishes.
\end{proof}

We then have a generalization of proposition \ref{beta1form}, 
\begin{corollary}\label{generalbeta1form} For general $\zeta \in \CH$, we have
\begin{equation}
\rho(\zeta) \hj_\lambda = \beta(w, \zeta).
\end{equation}
\end{corollary}

\subsection{Conjectures}
Here we extend the result (\ref{rhotraces}) to further understand the relationship between the traces and the more general $\rho$ map (\ref{generalrho}).
\begin{corollary}
For $\zeta \in \QZ^0$, i.e. $\ztr_\lambda( \zeta)= 0, \forall \lambda$, the map $\rho$ has the properties
\begin{equation}
 \ytr_u ( \rho(\xi) \zeta ) = \sum_\lambda \ztr_\lambda(\xi) \ytr_u (P_{\QZ_\lambda}\zeta),
 \end{equation}
 \begin{equation}
 \ztr_\gamma( \rho(\xi)  \zeta)= - \sum_{t \in \remset_\gamma}\ztr_{\gamma-t}(\xi)  \xtr_\gamma(P_{\QZ_{\gamma-t}}\zeta), 
 \end{equation}
 \begin{equation}
\qquad \xtr_\gamma (\rho(\xi) \zeta )= 0,
\end{equation}
for all $ \gamma \partition |\lambda|+1$. 
\end{corollary}

\newcommand{\Fmap}{F}
\begin{definition}\label{gooddefn}
Let $\xi = \sum_\lambda \xi_\lambda \in \CH$, where  $\xi_\lambda := P_{\QZ_\lambda}(\xi)$. We say $\xi$ is \emph{"good"} if for each $\lambda$, if $\xi_\lambda \neq  0$ then $\ztr_\lambda(\xi_\lambda ) \neq 0$. 
For good $\xi$ we define
\begin{equation}
\Fmap(\xi) \coloneqq \sum_{\lambda \,:\, \xi_\lambda \neq 0} \frac{\xi_\lambda}{\ztr_\lambda(\xi_\lambda )}.
\end{equation}
\end{definition}
With this, we can construct a prototype explanation for the trace twisting property \ref{betathetatwistedtraces}:
\begin{corollary}
If $\xi = \sum_\lambda \xi_\lambda \in \CH_n$ satisfies $\ztr_\lambda(\xi) \neq 0$ for all $\lambda$, (which implies that $\lambda$ is "good"), then
\begin{equation} 
\Tr_{n+1}( \rho(F(\xi)) \zeta ) = \rho_* \circ \Tr_n(\zeta).
\end{equation}
\end{corollary}
There is a canonical element in each degree $n$ that satisfies the properties of this corollary: $w^n$. It's easy to see:
\begin{equation}
F(w^n) = \sum_\lambda \frac{w \hat q_\lambda}{n\hbar} =\frac{ 1}{n\hbar} \CL\sum_\lambda \hat j_\lambda.
\end{equation}

\begin{lemma}\label{rhotildesmallaction}
For $\zeta \in \CF \subset \CZ_n^0$, we have
\begin{equation}  \rho(F(w^n)) \zeta = \tfrac{1}{n \hbar } \beta(w, \CL \zeta) =   \left( \tfrac{1}{n\hbar  } \sum_{k\geq 1}    \beta(w, w^k  ) V_{-k} \right) \zeta.  \end{equation}
\end{lemma}
\begin{proof}
let $\zeta = \sum_\lambda \zeta_\lambda \hat j_\lambda$. then
\begin{eqnarray}  \rho(F(w^n)) \zeta &=& \tfrac{1}{n \hbar }\sum_{\lambda, s} v_\lambda  [s] \htau_\lambda^s \rho_\lambda^s \hat j_\lambda \\
&=& \tfrac{1}{n \hbar }\sum_{\lambda, s} \zeta_\lambda  [s] \htau_\lambda^s \beta(w,\hpsi_\lambda^s)  \\
&=& \tfrac{1}{n \hbar } \beta(w, \CL\sum_{\lambda}  \zeta_\lambda \hat j_\lambda) \\
&=& \tfrac{1}{n \hbar } \beta(w, \CL \zeta) \\
&=& \tfrac{1}{n \hbar } \beta(w, \sum_k   w^k V_{-k} \zeta ).
\end{eqnarray}
\end{proof}
We define
\begin{equation}\label{rhotilde}
\tilde  \rho_n \coloneqq \rho(F(w^n)).
\end{equation}

\begin{corollary}\label{rhotildelowform}
\begin{equation}
\beta^{1,n} =\tilde  \rho_n \cdot \theta^{1,n}
\end{equation}
\end{corollary}
\begin{proof}
We know $\theta^{1,n} = V_n$, so $\rho(F(w^n)) \theta^{1,n} = \tfrac{1}{n \hbar } \beta(w, \CL V_n) = \beta(w, w^n)$.
\end{proof}

Lemma \ref{rhotildesmallaction} demonstrates a formula for $ \tilde\rho$ as a differential operator (i.e. defined in terms of $V_{-k}$) when acting on $\pi_0$, we conjecture that this formula holds in full generality.

\begin{conjecture}\label{rhownconjecture}
For $\zeta \in \QZ^0_n$, we have
\begin{equation}  \tilde  \rho_n\cdot \zeta =  \left( \tfrac{1}{n\hbar  } \sum_{k\geq 1}    \beta(w, w^k  ) V_{-k} \right) \zeta. \end{equation}
As a consequence of this, 
\begin{equation}\label{betamnident}\beta^{n,m} =  \tilde  \rho_{n+m-1} \cdot \theta^{n,m}. \end{equation}
\end{conjecture}
The second part of this conjecture follows from the first by expanding out $\beta^{m,n}$ (and $\theta^{m,n})$ in terms of $w^k$ and $\beta^{1,k'}$ $(\theta^{1,k'} \in \CF)$ for $k, k'$ by using the fact that they are Hochschild closed (e.g. equation \ref{betahoch}), then applying \ref{rhotildelowform} term by term, as the first part claims that $[ w^k,  \tilde  \rho_n] =0$.

Following from the trace theorem, we know that

\begin{corollary}
In general, $\theta$ and $\beta$ are related by
\begin{equation}
\beta(\zeta,\xi) = \tilde  \rho \cdot \theta(\zeta,\xi) + K(\zeta,\xi)
\end{equation}
where $K(\zeta,\xi) \in \ker \Tr$.

\end{corollary}

Although we have shown that $\theta(\zeta, \xi)$ and $\beta(\zeta, \xi)$ have twisted traces, we might be tempted to claim that there exists some general elements $\gamma(\zeta, \xi)$ such that,
\begin{equation} \beta(\zeta, \xi) =  \rho(F(\gamma(\zeta,\xi))) \theta(\zeta, \xi), \end{equation}
which would give a direct explanation of (\ref{betathetatwistedtraces}). However, this is impossible in general since $\theta(w, \hpsi_\lambda^s - \hpsi_\lambda^t)  =0$, however $\beta(w, \hpsi_\lambda^s - \hpsi_\lambda^t) \neq 0$.

\begin{conjecture}\label{betaequalsrhotheta}
For any $\zeta, \xi$ such that $\partial\Pi(\zeta ,\xi)$ is "good" (c.f. \ref{gooddefn}), then the following holds
\begin{equation}
\beta(\zeta,\xi)=  \rho( F( \partial\Pi(\zeta ,\xi)) \theta(\zeta,\xi). \end{equation}
\end{conjecture}

A proof of this conjecture would immediately yield the twisted trace property \ref{betathetatwistedtraces} for all good pairs $\zeta,\xi$, which is a generic condition. We notice that the earlier conjectured relation \ref{betamnident} is a special case of conjecture \ref{betaequalsrhotheta}. We have also checked it computationally up to degree $|\zeta| + |\xi|= 6$.

To explore this conjecture, we return to the non-`good' case above of $\zeta = w, \xi = \hpsi_\lambda^s - \hpsi_\lambda^t$, rather now we perturb by small $a \neq 1$, to $\xi_a = \hpsi_\lambda^s - a\hpsi_\lambda^t$, where we find the conjecture holds
\begin{eqnarray}
 \rho( F ( \partial\Pi(w, \hpsi_\lambda^s - a\hpsi_\lambda^t)) ) \theta(w, \hpsi_\lambda^s - a\hpsi_\lambda^t) &=&  \rho\left( \frac{\pi_+( \hpsi_\lambda^s - a\hpsi_\lambda^t)}{1-a} \right) (1-a) \hat j_\lambda \\
 &=& \rho(  \hpsi_\lambda^s - a\hpsi_\lambda^t )  \hat j_\lambda \\
  &=& \beta(w, \hpsi_\lambda^s - a\hpsi_\lambda^t ).
\end{eqnarray}
We see directly that the denominator of $F(\partial\Pi(\zeta, \xi))$ cancels out the vanishing coefficient $(1-a)$ of $\hj_\lambda$ in $\theta(\zeta, \xi)$. In general, the following result shows this cancellation always happens.

\begin{lemma}
\begin{equation}
\pi_0 P_{\QZ_\lambda} \theta(\xi,\zeta) = \ztr_\lambda\left( \partial\Pi(\xi,\zeta) \right) \hat j_\lambda.
\end{equation}
\end{lemma}
 \begin{proof}
From equation \ref{pi0theta} that we have $\pi_0 \theta(\xi,\zeta) = \pi_0 \CL  \partial\Pi(\xi,\zeta)$. Then, we know from equation \ref{omegamap} that for any $\Gamma$, we have $\pi_0 \CL \Gamma = \sum_\lambda \ztr_\lambda(\Gamma) \hj_\lambda$.
 \end{proof}
There is reason to doubt that conjecture \ref{betaequalsrhotheta} holds in higher degrees due to the behaviour of the term \ref{piplusthetagen}.
 
However, we show a substantial case in which conjecture \ref{betaequalsrhotheta} holds. 
\begin{proposition}
Conjecture \ref{betaequalsrhotheta} holds in the case where $\xi = \hpsi_\lambda^t$ and $\zeta = \Gamma w$, where $\Gamma \in \CF$. In particular, that is, 
\begin{equation}\beta(w\Gamma,\hpsi_{\lambda}^{t})= \rho( \Fmap( \partial\Pi(w\Gamma,\hpsi_{\lambda}^{t}))) \theta(w\Gamma,\hpsi_{\lambda}^{t}).\end{equation}
\end{proposition}

\begin{proof}
We start by expanding $\Gamma = \sum_\sigma \Gamma_\sigma \hj_\sigma$, to find
\begin{equation}\label{finalrespt1}
\theta(w\Gamma,\hpsi_{\lambda}^{t}) = \Gamma \cdot \theta(w,\hpsi_{\lambda}^{t}) = \Gamma \cdot \hj_\lambda = \sum_{\mu,\sigma} \Gamma_\sigma \hat c_{\sigma\lambda}^\mu \hj_{\mu}.
\end{equation}
For the next step, we start with
\begin{equation} \partial\Pi(w\Gamma,\hpsi_{\lambda}^{t}) = \Gamma \cdot \pi_+\hpsi_{\lambda}^{t}.  \end{equation}
We then use formula \ref{XYjackaction} to show  
\begin{equation}
\ztr_\mu(\Gamma \cdot \pi_+\hpsi_{\lambda}^{t})= \sum_{\sigma}\Gamma_\sigma \ztr_\mu( \hj_\sigma \cdot \pi_+\hpsi_{\lambda}^{t}) = \sum_{\sigma}\Gamma_\sigma \sum_{\nu} \hat c_{\sigma\nu}^{\mu} \ztr_\nu( \pi_+\hpsi_{\lambda}^{t}) = \sum_{\sigma}\Gamma_\sigma \hat c_{\sigma\lambda}^{\mu}. 
\end{equation}
Putting all these statements together, we have
\begin{eqnarray}
\rho( \Fmap( \partial\Pi(w\Gamma,\hpsi_{\lambda}^{t})))\theta(w\Gamma,\hpsi_{\lambda}^{t}) &=& \rho\left( \sum_{\mu} \frac{P_{\QZ_\mu} \Gamma \pi_+\hpsi_{\lambda}^{t} }{ \sum_{\sigma}\Gamma_\sigma \hat c_{\sigma\lambda}^{\mu}} \right)  \sum_{\mu,\sigma} \Gamma_\sigma \hat c_{\sigma\lambda}^\mu \hj_{\mu} \\
&=&  \sum_{\mu}  \rho\left( P_{\QZ_\mu} \Gamma  \pi_+\hpsi_{\lambda}^{t} \right) \hj_{\mu} \\
&=&  \sum_{\mu}  \beta( w,  P_{\QZ_\mu} \Gamma  \pi_+\hpsi_{\lambda}^{t} ) \\
&=&  \beta( w \Gamma, \hpsi_{\lambda}^{t} ),
\end{eqnarray}
in which we have used corollary \ref{generalbeta1form}.
\end{proof}

\appendix
\section{} 
\subsection{Partition Counting}\label{partitionapps}
Let $p(n)$ be the number of integer partitions of $n$. The partition counting function is given by $P(x) := \sum_{n=0}^{\infty} p(n)x^n =  \frac{1}{(x;x)_\infty}= 1+x+2x^2+3x^3 + 5x^4 +\ldots $.

The \emph{corner counting} numbers $p(n,r)$ count partitions of size $n$ with $r$ outer corners, and are represented by the generating function
\begin{equation}
P(x,t) = \sum_{n,r\geq 0} p(n,r) x^n t^r= t + t^2 x + 2t^2x^2 +(2t^2 +t^3)x^3 + (3t^2+2t^3)x^4 + \ldots 
\end{equation}
\begin{lemma}
The corner counting generating function is given by
\begin{equation}\label{partitionwcornersfunction}
P(x,t)=  \prod_{k=1}^{\infty} \frac{1-x^{k-1}(1-t)} {1-x^{k}} = \frac{(1-t;x)_\infty}{(x;x)_\infty}.
\end{equation}
\end{lemma}

\begin{proof}\footnote{The proof is due to Sam Hopkins (\url{https://mathoverflow.net/a/428502/25028}).}
Note that the number of corners is equal to the number of distinct parts plus one (you always have a corner at the `top' of a series of repeated parts, plus one at the very bottom of the partition). 
So the term of 
\begin{equation}
\frac{1-x^k(1-t)}{1-x^k}  = 1 + t \frac{x^k}{1-x^k} = 1 + t(x^k+ x^{2k} + \cdots)
\end{equation}
correspond to choosing how many parts equal to $k$ n your partition you want. There is an extra factor of $1-x^0(1-t)=t$ for the extra corner at the bottom of every partition.

\end{proof}

\begin{corollary}\label{Pxtprops}The following properties of the corner counting generating function hold,
\begin{equation} P(x,1)  = P(x),\end{equation}
\begin{equation}\label{propresult} P(x,1-x) = 1, \end{equation}
\begin{equation}  [ \partial_t P(x,t)]_{t=1} = \frac{1}{1-x} P(x). \end{equation}

\end{corollary}
\begin{proof}
The first two follow directly\footnote{Sam Hopkins provides an alternate probabilistic proof of the result \ref{propresult}. Namely, that weighting each partition $\lambda$ by $x^{|\lambda|}(1-x)^{\#corners(\lambda)}$, gives a probability distribution on the set of all partitions. Imagine constructing a partition as follows. We start with the empty partition. Then we focus on its unique corner. We flip a coin that is heads with probability $x$ and tails with probability $1-x$. If we get heads, we add a box in that corner, and then move on to consider the "next" corner of the partition we've built so far, moving left to right and top to bottom. When we flip a tails at a corner, we leave that box empty, but we still move on to the next corner. Unless we flipped tails at the bottom corner (i.e., in a row with no boxes in it), in which case we stop and output the partition we've made. It is not hard to see that we produce each $\lambda$ with probability $x^{|\lambda|}(1-x)^{\#corners(\lambda)}$.} from formula \ref{partitionwcornersfunction}. For the final statement, we note that
\begin{equation}  [ \partial_t P(x,t)]_{t=1}=\sum_{r,n\geq 0} r\, p(n,r)\,x^n \end{equation}
and by the decomposition (\ref{Zdecomp}) and the formula (\ref{dimensionhnfun}), we know that counting the total number of minima of all partitions is counted by $(1-x)^{-1}P(x)$.
\end{proof}






\subsection{$\tau$ identities}
Let $\{v,\bar v\} = \{ (1,0),(0,1)\} = \addset_{\{1\}}$, and the Spectral Factor
\begin{equation}
T_1([s]) = \frac{[s][s+v+\bar v]}{[s+v][s+\bar v]}
\end{equation}
\begin{lemma}
For $s\neq b \in \addset_{\lambda}$, we have
\begin{equation}\label{tauabTidenti}
 \tau_{\lambda+s}^{b} = T_1([s-b]) \, \tau_{\lambda}^{b} .
\end{equation}
For $s \in \addset_{\lambda}$, $t \in \remsetp_\lambda$, we have 
\begin{equation}\label{tauabTdownidenti}
\tilde \tau_{\lambda+s}^t =T_1([s-t])^{-1} \, \tilde \tau_{\lambda}^t .
\end{equation}
\end{lemma}


\begin{lemma}
\begin{equation}\label{hssvvident}
 \frac{ \hbar }{[s][s+v+\bar v]} =  1 -\frac{[s +v][s+\bar v]}{[s][s+v+\bar v]} =1 - T_1([s])^{-1} 
 \end{equation}
\end{lemma}
\begin{proof}
Note that for any $s$, we have $\hbar = [s][s+v+\bar v] - [s+v][s+\bar v]$.
\end{proof}

\begin{lemma}
For $s+(1,1) \in \remsetp_\lambda$, the following identities hold
\begin{equation}\label{tausumident}
\sum_{q \in \addset {\lambda} } \frac{ \hbar\, \tau_{\lambda}^q }{[s-q][s-q+(1,1)]} = \tau_{\lambda-s}^s
\end{equation}
and
\begin{equation}\label{tildetausumident}
\sum_{t \in \remsetp_\lambda}  \frac{\hbar\, \tilde \tau_{\lambda}^t}{[s-t][s-t+(1,1)]} = - \hbar  + \tilde \tau_{\lambda+s}^{s+(1,1)} 
\end{equation}
\end{lemma}
\begin{proof}
For the first, we use (\ref{hssvvident}) to find the left hand side is
\begin{equation}
\sum_{q \in \addset {\lambda} }  \tau_{\lambda}^q \left( 1 -T_1([s-q])^{-1}\right) = 1 - \sum_{q \in \addset {\lambda} }  \tau_{\lambda}^q T_1([s-q])^{-1}
\end{equation}
Now the possible minima $q=s+v$ and $q=s+\bar v$ of $\lambda$ cant contribute to the sum, since their coefficients vanish in the numerator of $T_1([s-q])^{-1}$. Thus the sum reduces to minima of $\lambda-s$, except $q=s$, and then we use (\ref{tauabTidenti}) to recover
\begin{equation}
 \sum_{q \in \addset {\lambda-s}, q \neq s }  \tau_{\lambda}^q T_1([s-q])^{-1} =  \sum_{q \in \addset {\lambda-s}, q \neq s }  \tau_{\lambda-s}^q =1 -  \tau_{\lambda-s}^s .
 \end{equation}
 For the second identity, 
 \begin{eqnarray}
 \sum_{t \in \remsetp_\lambda} \tilde \tau_{\lambda}^t \frac{\hbar}{[t-(1,1)-s][t-s]}  &=&\sum_{t \in \remsetp_\lambda} \tilde \tau_{\lambda}^t ( 1 - T_1([s-t])^{-1} )\\
&=& |\lambda| \hbar - \sum_{t \in \remsetp_\lambda}\tilde \tau_{\lambda}^t T_1([s-t])^{-1} \\
&=& |\lambda| \hbar - \sum_{t \in \remsetp_\lambda+s, t\neq s+(1,1)} \tilde \tau_{\lambda+s}^t \\
&=& |\lambda| \hbar - |\lambda+s|\hbar + \tilde \tau_{\lambda+s}^{s+(1,1)} 
\end{eqnarray}
On the first line we used (\ref{hssvvident}), and on the second line we used (\ref{tauabTidenti}).
\end{proof}

\section{}\label{hilbseriesapp}
In this appendix, we compute the Hilbert series of the resolution $\CA$ given by \ref{Aresolution}, and show
\begin{proposition}
\begin{equation}\label{hilbpolyapp}
HS_{\!\CF}(\CA,z,x) = (z;x)_\infty.
\end{equation}
\end{proposition}
\begin{proof}
From the definition of the space of $W$ symbols given in the proof of Prop \ref{nullgeneration}, we define
\[ A^n(x) := \sum_{k=1}^{\infty} \left( \dim{}_{\!\CF} \CA^n_k \right) x^k = \sum_{0 \leq a_1 < a_2 < \ldots < a_n} x^{\Sigma_{i=1}^{n}{a_i}}. \]
We claim that 
\begin{equation}\label{aktelescope}
(1-x^{n+1}) A^{n+1}(x) = x^n A^{n}(x)
\end{equation}
We check 
\begin{eqnarray}
 x^{n+1} A^{n+1} &=& \sum_{0\leq a_1  < \ldots  < a_{n+1}} x^{\sum{a_i}+n+1} \\
 &=& \sum_{0\leq a_1 < \ldots  < a_{n+1}} x^{\sum ({a_i}+1)} \\
&= &\sum_{0< a_1+1  < \ldots  < a_{n+1}+1} x^{\sum ({a_i}+1)} \\
 &=& \sum_{0< a'_1  < \ldots < a_{n'+1}} x^{\sum ({a'_i})}.
  \end{eqnarray}
Thus
\begin{eqnarray} (1-x^{n+1}) A^{n+1} &=& \sum_{0 = a_1 < \ldots < a_{n+1}} x^{\sum ({a_i})} \\
& =& \sum_{0 < a_2 < \ldots < a_n < a_{n+1}} x^{\sum ({a_i})} \\
& =& x^n A^n.
  \end{eqnarray}
  
We can now compute the Hilbert series of the resolution, using the telescoping property \ref{aktelescope}, and $A_0(x) = 1$, we find
\begin{equation}\label{hilbtelescope}
HS_{\!\CF}(\CA,z,x) := \sum_{n=0}^{\infty} (-1)^{n} A^n(x)z^n  = \sum_{n=0}^{\infty} (-z)^{n} \prod_{\ell=0}^{n-1} \frac{x^\ell}{1-x^{\ell+1}}.
\end{equation}
We recall the following consequence of the $q$-binomial theorem (at $z=1$) due to Euler,
\begin{equation} \sum_{n=0}^{\infty} (-z)^{n}   \frac{q^{n(n-1)/2}}{(q;q)_{n}} = (z;q)_\infty.
\end{equation}
With this result, we recover the formula \ref{hilbpolyapp}.

\end{proof}

\end{onehalfspace}
\bibliographystyle{sigma}

\bibliography{/Users/ryanmickler/Dropbox/Kennebunk/Archive/MasterArchive}
\end{document}